\documentclass[10pt]{preprint}
\usepackage{authblk}
\usepackage{xcolor}
\usepackage{times}
\usepackage{mathtools}
\usepackage[
  colorlinks   = true,   
  linkcolor    = darkblue,  
  citecolor    = darkblue, 
  urlcolor     = darkblue,  
  linktoc      = page    
]{hyperref}
\usepackage[shortlabels]{enumitem}
\usepackage{mathrsfs}
\usepackage{microtype}
\usepackage{mhequ}
\usepackage{cprotect}
\usepackage[font=small]{caption}
\usepackage{titlesec}
\usepackage{fancyhdr}
\usepackage{xstring}
\usepackage{color}
\usepackage[a4paper,top=1.6in,bottom=1.6in,left=1.6in,right=1.6in]{geometry}
\usepackage{lastpage}
\usepackage[english]{babel}
\usepackage{amsmath,amsthm,amssymb,mathtools}
\usepackage{enumitem}
\usepackage{tikz}
\usepackage{orcidlink}
\usepackage{placeins}

\usetikzlibrary{arrows,decorations.pathmorphing,backgrounds,positioning,fit,petri}
\usetikzlibrary{shapes}
\usetikzlibrary{shapes.misc}
\usetikzlibrary{external}
\colorlet{darkblue}{blue!75!black}

\def\Z{\mathbb{Z}}

\newcommand{\expected}{\mathbb{E}}
\newcommand{\prob}{\mathbb{P}}
\newcommand{\eps}{\varepsilon}

\renewcommand{\pmod}[1]{\mkern4mu({\text{mod}\mkern 3mu #1})}

\DeclareMathOperator*{\mymax}{\text{\normalfont max}}
\DeclareMathOperator*{\mymin}{\text{\normalfont min}}
\DeclareMathOperator*{\mysup}{\text{\normalfont sup}}
\DeclareMathOperator*{\myinf}{\text{\normalfont inf}}
\DeclareMathOperator*{\mylim}{\text{\normalfont lim}}
\DeclareMathOperator*{\mylog}{\text{\normalfont log}}
\DeclareMathOperator*{\mylimsup}{\text{\normalfont limsup}}
\DeclareMathOperator*{\myliminf}{\text{\normalfont liminf}}

\DeclareMathOperator*{\sgn}{\text{\normalfont sgn}}
\DeclareMathOperator*{\de}{\hspace{-.1em}\text{\normalfont d\hspace{-.15em}}}
\DeclareMathOperator{\Var}{Var}

\numberwithin{equation}{section}
\newtheorem{theorem}{Theorem}[section]
\newtheorem{corollary}[theorem]{Corollary}
\newtheorem{lemma}[theorem]{Lemma}

\newtheorem{proposition}[theorem]{Proposition}
\theoremstyle{definition}
\newtheorem{definition}[theorem]{Definition}
\newtheorem{remark}[theorem]{Remark}
\newtheorem*{acknowledgements}{Acknowledgements}

\DeclareFontEncoding{LS1}{}{}
\DeclareFontSubstitution{LS1}{stix}{m}{n}
\DeclareSymbolFont{stixletters}{LS1}{stix}{m}{it}
\DeclareMathAccent{\cev}{\mathord}{stixletters}{"91}
\DeclareMathAccent{\vec}{\mathord}{stixletters}{"92}

\makeatletter
\DeclareRobustCommand{\TitleEquation}[2]{\texorpdfstring{\StrLeft{\f@series}{1}[\@firstchar]$\if%
b\@firstchar\boldsymbol{#1}\else#1\fi$}{#2}}
\makeatother

\def\scal#1{\langle#1\rangle} 
\def\dash{\leavevmode\unskip\kern0.18em--\penalty\exhyphenpenalty\kern0.18em}
\def\slash{\leavevmode\unskip\kern0.15em/\penalty\exhyphenpenalty\kern0.15em}

\begin{document}
\allowdisplaybreaks
\date{\today}
\title{Scaling Limits of a Weakly Perturbed\\ Random Interface Model}
\author{\bfseries Patrícia Gonçalves\textsuperscript{1}\orcidlink{0000-0002-8093-8810}, 
Martin Hairer\textsuperscript{2}\orcidlink{0000-0002-2141-6561}, 
Maria Chiara Ricciuti\textsuperscript{3}\orcidlink{0009-0005-8265-4650}
}
\institute{
University of Lisbon, Portugal. \email{pgoncalves@tecnico.ulisboa.pt}
\and EPFL, Switzerland and Imperial, UK. \email{martin.hairer@epfl.ch}
\and Imperial, UK. \email{maria.ricciuti18@imperial.ac.uk}
}

\maketitle
\thispagestyle{empty}

\begin{abstract} 
We consider a random interface model on the discrete torus with $2n$ sites, obtained from the classical corner flip dynamics but with a weak global perturbation, namely an asymmetry of order $n^{-\gamma}$ of the direction of growth that switches direction based on the sign of the total area under the interface. The slopes of this model can be viewed as a non-simple exclusion process at half filling with globally dependent rates. We show that, for $\gamma=1$, the hydrodynamic equation of the empirical density is given by a time concatenation of the viscous Burgers equation and the heat equation. Moreover, for $n$ prime and $\gamma>\frac{6}{7}$, we establish convergence in law of the equilibrium fluctuations to an infinite-dimensional Ornstein-Uhlenbeck process.
\end{abstract}

%%%%%%%%%%%%%%%%%%%%%%%%%%%%%%%%%%%%%%%%%%%%%%%%%%
%%%CONTENTS%%%%%%%%%%%%%%%%%%%%%%%%%%%%%%%%%%%%%%%
%%%%%%%%%%%%%%%%%%%%%%%%%%%%%%%%%%%%%%%%%%%%%%%%%%
\tableofcontents

%%%%%%%%%%%%%%%%%%%%%%%%%%%%%%%%%%%%%%%%%%%%%%%%%%
%%%INTRODUCTION%%%%%%%%%%%%%%%%%%%%%%%%%%%%%%%%%%%
%%%%%%%%%%%%%%%%%%%%%%%%%%%%%%%%%%%%%%%%%%%%%%%%%%
\section{Introduction}
Interacting particle systems are a class of stochastic models describing the collective evolution of many simple components under random dynamics. They are a central tool in probability theory and statistical mechanics, both for their rich mathematical structure and for the insight they provide into universal scaling laws. Among the most studied examples are \textit{exclusion processes}, where particles perform random walks on a lattice subject to the exclusion rule that forbids more than one particle per site. This simple rule already gives rise to a remarkably broad spectrum of macroscopic phenomena, and variants of the exclusion process underpin much of the modern theory of interacting particle systems.

A cornerstone of this field is the study of \textit{hydrodynamic limits}. At the macroscopic level, the empirical particle density of exclusion processes typically converges, after rescaling, to the solution of a deterministic partial differential equation (PDE). In the nearest-neighbour symmetric case (SSEP) and under diffusive scaling, the limit is governed by the linear heat equation, whereas in the weakly asymmetric case (WASEP) \dash when the asymmetry scales like $N^{-1}$ \dash the density profile evolves according to the viscous Burgers equation. In contrast, when the asymmetry is of order one as in the totally asymmetric simple exclusion process (TASEP), the microscopic dynamics exhibit a privileged direction and the characteristic time scale is much shorter: diffusive scaling no longer applies and the natural hydrodynamic time scale becomes hyperbolic. In this regime, the macroscopic density satisfies the inviscid Burgers equation, a nonlinear conservation law that typically develops shocks from smooth initial data. These results form one of the foundational successes of the subject; see the monograph of Kipnis and Landim~\cite{kl99} for a detailed account. 

Beyond the hydrodynamic law of large numbers, one can study the \textit{fluctuations} around the deterministic behaviour. In the diffusive setting, the stationary fluctuations of SSEP are described by an infinite-dimensional Ornstein-Uhlenbeck process \cite{rav92}, whereas for WASEP \dash this time for an asymmetry scaling as $N^{-\frac{1}{2}}$ \dash they converge to a suitably defined solution of the \textit{stochastic Burgers equation}~\cite{gj14,gjs17}. For TASEP, in the hyperbolic time scale, fluctuations are trivial and get linearly transported in time: in order to get non-trivial fluctuations, one needs to scale time as $N^{\frac{3}{2}}$, yielding a non-Gaussian limiting field described by the \textit{KPZ fixed point}~\cite{qs23} constructed in \cite{mqr21}. 

In a broader setting, considerable attention has been devoted to \textit{multi-species exclusion processes}, in which particles of different types interact and there is more than one conservation law. A paradigmatic example is the ABC model, originally introduced in \cite{ekkm98} as a minimal model for phase separation in one-dimensional driven systems. The hydrodynamic behaviour of a variation of this model has been extensively investigated, showing convergence to coupled systems of nonlinear PDEs with rich phase diagrams; see \cite{cde03, bdlw08, gmo23}. At the level of fluctuations, multi-component exclusion systems fall within the framework of \textit{nonlinear fluctuating hydrodynamics}, which predict long-time correlations induced by the interaction of characteristic modes. Some rigorous results supporting these predictions in related classes of one-dimensional systems have been derived; see, for example, \cite{bfs21, cgmo25, ho25}.

More generally, nonlinear fluctuating hydrodynamics and mode-coupling theory predict that one-dimensional systems with several conserved quantities give rise to an infinite discrete family of dynamical universality classes, with dynamical exponents given by ratios of consecutive Fibonacci numbers. Besides the diffusive and KPZ classes, this includes, in particular, the \textit{Lévy $\frac{3}{2}$}, the \textit{Lévy $\frac{5}{3}$}, and the \textit{golden mean} universality classes. In these cases, dynamical structure functions are predicted to be governed by asymmetric Lévy-stable scaling profiles; see \cite{spo14, pss15, psss16, sch18} and references therein. While numerical evidence has been reported in a variety of driven diffusive systems, the rigorous understanding of these universality classes remains largely open.

It is well known that the exclusion process can be represented in terms of an evolving interface. In one dimension, the occupation variables of the process correspond to the discrete slopes of a height function, and the associated dynamics are given by the so-called \textit{corner flips}. In this dual picture, each particle corresponds to an upward slope and each hole to a downward slope: a rightward particle jump translates into a local maximum flipping downwards, while a leftward jump corresponds to a local minimum flipping upwards. This coupling between exclusion and interface dynamics provides an alternative geometric viewpoint and has played an important role in connecting interacting particle systems to random surfaces and growth phenomena. In particular, the seminal paper \cite{kpz86} introduced the \textit{KPZ equation} as a universal model for one-dimensional interface growth, whose microscopic representatives include corner-flip evolutions. In \cite{bg97}, the authors proved that the weakly asymmetric exclusion process, when written in its height-function formulation and under the characteristic $N^{\frac{1}{2}}$ scaling, converges in distribution to the KPZ equation, providing one of the first rigorous derivations of this stochastic PDE from an interacting particle system. Weakly asymmetric corner flips have also been extensively studied when evolving on the finite one-dimensional lattice; see \cite{el15, lab18} and references therein.

In the present work, we introduce and study a variant of the corner-flip dynamics, which we call the \textit{weakly perturbed interface model}. Equivalently, the slopes of the height function define a particle configuration on a discrete torus of size $2n$ at half filling. The distinctive feature of our model is that the weak asymmetry of the dynamics is not fixed in time, but instead depends on a global observable, namely the signed area under the height function. When the area is positive, the dynamics exhibit a weak downward bias, while when it is negative, the bias points upward. In this sense, the model interpolates between SSEP and WASEP, but with a global feedback mechanism that drives the area back towards zero. 

From a modelling perspective, such a mechanism can represent situations where global balance constraints affect local dynamics. A natural instance is the interface between two metals forming an alloy. Below the miscibility gap, namely their ``mixing temperature", the two components do not mix uniformly throughout, but at the boundary atoms of one metal can still exchange with atoms of the other. This interfacial region can be modelled via corner flips and evolves randomly, yet the total mass of each component is conserved. If the boundary drifts too far into one material, the conservation laws produce a weak restoring effect, so that flips are then biased back towards a balanced position. At the continuum level, the macroscopic evolution of the corresponding phase-field model is classically described by the Cahn-Hilliard equation, introduced in \cite{ch58}. Our weakly perturbed interface model then provides a tractable instance of this type of interaction, where local randomness is tempered by a global conservation law.

\subsection{Main Results} 
Our first main result (Theorem~\ref{thm:hydro}) is the identification of the \textit{hydrodynamic limit} at the critical perturbation scale. We assume that the system starts from a measure which has a relative entropy of order at most $N$ with respect to the stationary measure, and that it converges to an interface profile which lives in the Sobolev space $\mathcal{H}^1$. When the asymmetry is of order $N^{-1}$ ($\gamma=1$), we prove that the rescaled height function converges to the solution of a nonlinear PDE which evolves according to a viscous Burgers-type equation up to the time at which the total signed area of the interface vanishes, and subsequently follows the linear heat equation. At the particle level, this means that initially the density evolves according to the viscous Burgers equation, reflecting the behaviour of WASEP. Once the area of the interface vanishes, however, the nonlinear term cancels out and the limiting equation crosses over to the heat equation, so that the macroscopic dynamics are described by a time concatenation of Burgers and heat flow. This phenomenon has no parallel in the classical exclusion setting and highlights the qualitative effect of the global feedback: even though the microscopic asymmetry is always present, its macroscopic signature disappears after the global observable changes sign. 

At the level of \textit{equilibrium fluctuations}, starting from the invariant measure  of the dynamics, and provided that $\gamma>\frac{6}{7}$ as well as the number of sites satisfies $N=2p$ with $p$ prime, we establish convergence of the fluctuation field to an infinite-dimensional Ornstein-Uhlenbeck process (Theorem~\ref{thm:flucts}). This extends the classical fluctuation results for SSEP and WASEP to a setting with a non-product invariant measure and globally dependent dynamics. The threshold $\frac{6}{7}$ is purely technical and we do not expect it to be optimal, and in particular we conjecture that convergence should hold at least for $\gamma>\frac{1}{2}$, paralleling the corresponding result for the classical WASEP. Since our dynamics are in a sense even ``more symmetric'', this further suggests that the cancellation of the nonlinear terms should persist for at least the same range of $\gamma$. On the other hand, the prime-number condition arises in our proof of correlation bounds. We obtain these estimates via a combinatorial analysis which involves exploiting algebraic properties of primitive roots of unity; although we expect the result to hold for all even $N$, this remains an open technical restriction. 

Our approach follows the standard program developed for interacting particle systems, but with substantial adaptations. At both the hydrodynamic and fluctuation levels, we represent the empirical fields through Dynkin-type martingales, establish tightness of the associated processes, and characterise all possible limit points as solutions to the corresponding deterministic or stochastic PDEs. Uniqueness of solutions then allows us to deduce convergence of the process of interest. While this overall strategy is standard in the literature, the presence of a global, state-dependent asymmetry introduces several substantial technical difficulties compared to classical exclusion models.

For the hydrodynamic limit, we study both the empirical density of the underlying particle system and the appropriately-scaled area of the interface, and show that these converge to the solution of a system of coupled PDEs. Because of the nonlinearity in the martingale equation, one needs to show a \textit{replacement lemma} for the occupation variables in order to close the equation in terms of the empirical density. However, in our setting, the entropy method of \cite{gpv88} cannot be directly applied, and must be modified to incorporate the presence of the sign of the area in front of the nonlinearity. We successfully achieve this in Lemma \ref{lemma:replacement}.

Unlike the classical simple exclusion processes, the stationary measure of our dynamics, given in \eqref{eq:inv_measure}, is not of product form when projected onto particle configurations. At the level of equilibrium fluctuations, this means that its correlation functions must be estimated in order to control the degree-two term appearing in the martingale decomposition. As previously mentioned, we obtain these correlation bounds via a combinatorial analysis; see Appendix \ref{sec:app_correlations}. Moreover, the classical block-type estimates \dash needed to study the degree-two term in the martingale equation \dash require the stationary measure to be invariant under particle jumps, a property that fails in our model. To get around this, we introduce an appropriate correction term which allows us to apply the result of \cite{gjs17} (Lemma \ref{lemma:one_block}). This produces additional error terms, and controlling these extra contributions imposes restrictions on the admissible scaling exponent $\gamma$: after optimising the resulting bounds, this yields the technical threshold $\gamma>\frac{6}{7}$ appearing in our fluctuation theorem.

\subsection{Outline} 
In Section~\ref{sec:model_results} we define the model precisely and state our main results on the hydrodynamic limit and equilibrium fluctuations. Section~\ref{sec:hydro} is devoted to the hydrodynamic limit: we establish tightness via Dynkin’s martingales, prove a replacement lemma adapted to the switching asymmetry, and identify the hydrodynamic PDE. Section~\ref{sec:flucts} focuses on equilibrium fluctuations: we prove tightness of the fluctuation field, establish an appropriate block-type estimate to treat the degree-two term in the martingale equation, and characterise the limit as an Ornstein-Uhlenbeck process. Appendix~\ref{sec:app_energy_uniqueness} contains an energy estimate and the uniqueness proof for weak solutions of the hydrodynamic PDE, while Appendix~\ref{sec:app_correlations} establishes the asymptotics of the correlation functions under the stationary measure of the model.

\begin{acknowledgements}
\small{The research of P.G. is partially funded by Fundação para a Ciência e Tecnologia (FCT), Portugal, through grant No. UID/4459/2025, as well as the ERC/FCT SAUL project. M.C.R. gratefully acknowledges support from the Dean's PhD Scholarship at Imperial College London, as well as the hospitality of EPFL and Instituto Superior Técnico during several research visits between January 2023 and December 2025, when part of this work was carried out.}
\end{acknowledgements}

%%%%%%%%%%%%%%%%%%%%%%%%%%%%%%%%%%%%%%%%%%%%%%%%%%
%%%THE MODEL AND MAIN RESULTS%%%%%%%%%%%%%%%%%%%%%
%%%%%%%%%%%%%%%%%%%%%%%%%%%%%%%%%%%%%%%%%%%%%%%%%%
\section{The Model and Main Results}\label{sec:model_results}
Let $N=2n$ be an even positive integer and call $\mathbb{T}_N := \Z / N\Z$ the one-dimensional discrete torus of size $N$. 
Our state space is the set $\Omega_N$ of \textit{height functions} on $\mathbb{T}_N$, namely
\begin{equation*}
    \Omega_N := \left\{ h:[0, N]\to\mathbb{R}: h(0)=h(N) \ \ \text{and}\ 
    \begin{array}{ll}
    h(x)\in\mathbb{Z}, \\
    |h(x)-h(x+1)|=1, 
    \\ h \ \text{is linear on} \ [x, x+1)\end{array}  \ \ \forall x\in\mathbb{T}_N
    \right\}.
\end{equation*}
Define the map $Y:\Omega_N\to\mathbb{Z}$ which computes the area of a height function as
\begin{equation*}
    Y(h):=\sum_{x\in\mathbb{T}_N}h(x),
\end{equation*}
and let $\mylog \frac{p_N^{\downarrow}(h)}{p_N^{\uparrow}(h)}=2N^{-\gamma}\sgn(Y(h))$ for some $\gamma>0$ and $p_N^{\downarrow}+p_N^{\uparrow}=1$, so that
\begin{equation}\label{eq:rates_up_down}
    p_N^{\downarrow}(h)=\left(1+e^{-2N^{-\gamma}\sgn(Y(h))}\right)^{-1}  \ \ \text{and} \ \ \ \ 
    p_N^{\uparrow}(h)=\left(1+e^{2N^{-\gamma}\sgn(Y(h))}\right)^{-1}
\end{equation}
(our convention will be that $\sgn(0) = 0$). We consider the model given by running the \textit{corner flip dynamics} on $\Omega_N$: at rates $p_N^{\downarrow}$ and $p_N^{\uparrow}$, respectively, every local maximum flips into a local minimum, and vice versa. The generator $\mathcal{L}^N=\mathcal{L}^N(\gamma)$ of this Markov process acts on local functions $f:\Omega_N\to\mathbb{R}$ via
\begin{equation*}
    \mathcal{L}^Nf(h):=\sum_{x\in\mathbb{T}_N}\left[f\left(h^{x}\right)-f(h)\right]\left(p_N^\downarrow(h)\boldsymbol{1}_{\{\Delta h(x)<0\}}+p_N^{\uparrow}(h)\boldsymbol{1}_{\{\Delta h(x)>0\}}\right).
\end{equation*}
Here, $\Delta$ denotes the discrete Laplacian 
\begin{equation*}
    \Delta h(x):=h(x+1)-2h(x)+h(x-1),
\end{equation*}
and $h^{x}$ denotes the configuration obtained from $h$ by flipping the corner at $x$, namely
\begin{equation*}
    h^{x}(y):=\begin{cases}
    h(y) & y\ne x, \\ h(x)+\Delta h(x) & y=x.
    \end{cases}
\end{equation*}
Note that $\Delta h(x) \in \{-2, 0,2\}$, so that the event $\{\Delta h(x)<0\}$ is verified if and only if $h$ has a local maximum at $x$, and $\{\Delta h(x)>0\}$ is verified if and only if $h$ has a local minimum at $x$. We will denote by $\{h^N_t, t\ge0\}$ the process generated by $N^2\mathcal{L}^N$, and call it \textit{weakly perturbed height process} accelerated by $N^2$ and with perturbation strength of $N^{-\gamma}$.

Now let $\xi:\Omega_N\to\{0, 1\}^{\mathbb{T}_N}$ be the projection map defined as
\begin{equation}\label{eq:projection}
    \xi[h](x):=\frac{h(x)-h(x-1)+1}{2}, \ \ x\in\mathbb{T}_N
\end{equation}
and let $\{\xi_t^N, t\ge0\}$ denote the $\{0, 1\}^{\mathbb{T}_N}$-valued process defined as $\xi^N_t=\xi(h_t^N)$. The dynamics of the height process imply that $\{\xi_t^N, t\ge0\}$ is an exclusion process on the torus $\mathbb{T}_N$, which we call \textit{weakly perturbed exclusion process} accelerated by $N^2$ and with perturbation strength $N^{-\gamma}$. Because of \eqref{eq:projection}, adopting the usual convention, we associate a positive slope of a height configuration on $[x-1, x)$ to a particle at $x$, and a negative slope to an empty site. Then, flipping down a local maximum of a height profile corresponds to a right particle jump, and vice versa for a local minimum; see Figure \ref{fig:dynamics}. Note that the fact that height configurations are defined on the torus implies that exactly half of the sites of $\mathbb{T}_N$ are occupied by particles. 

\vspace{.5em}
%%%%%%%%%%%%%%%%%%%%%%%%%%%%%%%%%%%%%%%%%%%%%%%%%%
\begin{figure}[!ht]
\includegraphics[width=0.3\textwidth]{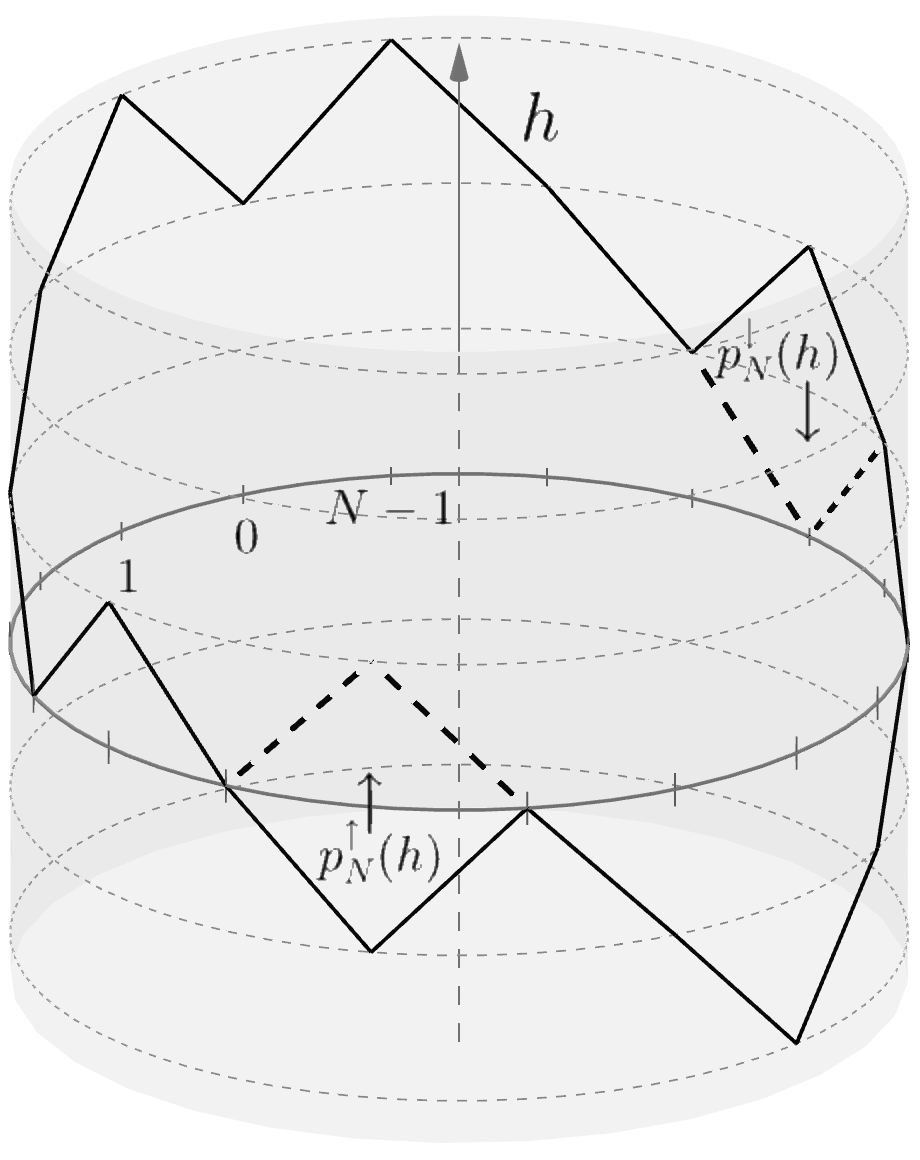} \hspace{1em}
\begin{tikzpicture}[thick, scale=0.45][h!]    
    \draw[step=1cm,lightgray,very thin] (-9,0) grid (8,0);   
    \foreach \y in {-9,...,8}{
    \draw[color=lightgray,very thin] (\y,-0.2)--(\y,0.2);}    
    \draw (-9,-0.7) node [color=black] {$\scriptstyle{1}$};
    \draw (-8,-0.7) node [color=black] {$\scriptstyle{2}$};
    \draw (-0.5,-0.7) node [color=black] {$\scriptstyle{...}$};
    \draw (8,-0.7) node [color=black] {$\scriptstyle{N-1}$};

    \node[ball color=black!30!, shape=circle, minimum size=0.25cm] (1) at (-9,0) {};
    \node[ball color=black!30!, shape=circle, minimum size=0.25cm] (5) at (-5,0) {};
    \node[shape=circle, minimum size=0.25cm] (7) at (-3,0) {};
    \node[ball color=black!30!, shape=circle, minimum size=0.25cm] (8) at (-2,0) {};
    \node[ball color=black!30!, shape=circle, minimum size=0.25cm] (11) at (1,0) {};
    \node[ball color=black!30!, shape=circle, minimum size=0.25cm] (12) at (2,0) {};
    \node[ball color=black!30!, shape=circle, minimum size=0.25cm] (13) at (3,0) {};
    \node[ball color=black!30!, shape=circle, minimum size=0.25cm] (14) at (4,0) {};
    \node[shape=circle, minimum size=0.25cm] (15) at (5,0) {};
    \node[ball color=black!30!, shape=circle, minimum size=0.25cm] (16) at (6,0) {};
    \node[ball color=black!30!, shape=circle, minimum size=0.25cm] (17) at (7,0) {};
    \draw (-9.2,1) node [color=black] {$\xi$};

    \draw[step=1cm,lightgray,very thin] (-9.5,5) grid (8.5,5); 
    \draw[->, color=lightgray, thin] (-9.5,2.5)--(-9.5,8.5); 
    \foreach \y in {-8.5,...,8.5}{
    \draw[color=lightgray,very thin] (\y,4.8)--(\y,5.2);}    
    \draw (-9.2,4.3) node [color=black] {$\scriptstyle{0}$};
    \draw (-8.5,4.3) node [color=black] {$\scriptstyle{1}$};
    \draw (0, 4.3) node [color=black] {$\scriptstyle{...}$};
    \draw (8.5,4.3) node [color=black] {$\scriptstyle{N-1}$};    
    \draw[color=black, semithick] (-9.5,7)--(-8.5,8);
    \draw[color=black, semithick] (-8.5,8)--(-5.5,5);
    \draw[color=black, semithick] (-5.5,5)--(-4.5,6);
    \draw[color=black, semithick] (-4.5,6)--(-2.5,4);
    \draw[color=black, semithick] (-2.5,4)--(-1.5,5);
    \draw[color=black, semithick] (-1.5,5)--(0.5,3);
    \draw[color=black, semithick] (0.5,3)--(4.5,7);
    \draw[color=black, semithick] (4.5,7)--(5.5,6);
    \draw[color=black, semithick] (5.5,6)--(7.5,8);
    \draw[color=black, semithick] (7.5,8)--(8.5,7);

    \draw[dashed,color=gray] (-3.5,5)--(-2.5,6);
    \draw[dashed,color=gray] (-2.5,6)--(-1.5,5);
    \draw[dashed,color=gray] (3.5,6)--(4.5,5);
    \draw[dashed,color=gray] (4.5,5)--(5.5,6);
    \draw[->, color=black, thick] (-2.5, 4.6)--(-2.5,5.4); 
    \draw[<-, color=black, thick] (4.5, 5.6)--(4.5,6.4);
    
    \draw (-2.5,6.5) node [color=black] {\scriptsize{$p_N^\uparrow(h)$}};
    \draw (4.5,7.5) node [color=black] {\scriptsize{$p_N^\downarrow(h)$}};
    \draw (-9.1,8.2) node [color=black] {\scriptsize{$h$}};       
    \path [<-] (7) edge[bend left=75, thin] node[above]{} (8);
    \path [->] (14) edge[bend left=75, thin] node[above] {} (15);   
\end{tikzpicture}
\caption{Microscopic dynamics of the weakly perturbed height process and correspondence between interface profiles and particle configurations.}\label{fig:dynamics}
\end{figure}
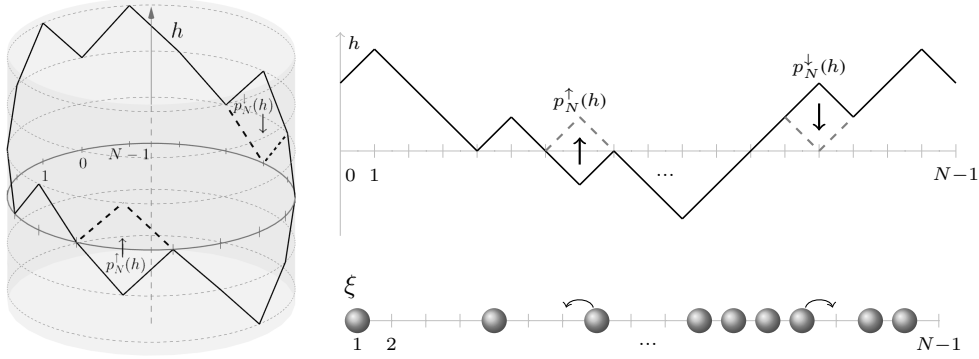
%%%%%%%%%%%%%%%%%%%%%%%%%%%%%%%%%%%%%%%%%%%%%%%%%%

Throughout, to ease the notation, we will denote $\xi[h](x)$ by $\xi(x)$ or $\xi_x$, and we will denote by $h^{x, y}$ the configuration that swaps the steps $\xi(x)$ and $\xi(y)$, so that in particular $h^{x}=h^{x, x+1}$. Finally, for each $x\in\mathbb{T}_N$, we set $p_N(x, x+1, h)=p_N^{\downarrow}(h)$ and $p_N(x+1, x, h)=p_N^{\uparrow}(h)$. With this in mind, the generator of the height process can be rewritten as
\begin{equation*}
    \mathcal{L}^N f(h)=\sum_{x\in\mathbb{T}_N}c_{x, x+1}(h)\left[f\left(h^{x}\right)-f(h)\right],
\end{equation*}
where
\begin{equation*}
    c_{x, x+1}(h)=p_N(x, x+1, h)\xi_x(1-\xi_{x+1})+p_N(x+1, x, h)\xi_{x+1}(1-\xi_x).
\end{equation*}
\indent By Taylor expanding the rates \eqref{eq:rates_up_down} in $N$, we get that 
\begin{equation*}
    p_N(x, x+1, h)=\frac{1}{2}+\frac{1}{2N^\gamma}\sgn(Y(h))+O(N^{-2\gamma})
\end{equation*}
and
\begin{equation*}
    p_N(x+1, x, h)=\frac{1}{2}-\frac{1}{2N^\gamma}\sgn(Y(h))+O(N^{-2\gamma}).
\end{equation*}
This means that the weakly perturbed exclusion process behaves like a weakly asymmetric simple exclusion process (WASEP), but with asymmetry switching from left to right \dash and vice versa \dash based on the sign of the integral of the height process. 
The sign of the asymmetry is such that, on average, this integral always decreases in absolute value. It is in this sense that we can view this process as ``in between'' the symmetric simple exclusion process (SSEP) and the WASEP.

For simple exclusion processes, the conservation of the number of particles is reflected in the existence of a one-parameter family of invariant measures, the Bernoulli product measures $\nu_\alpha$ whose marginals are given by
\begin{equation*}
    \nu_\alpha\left\{\xi\in\{0, 1\}^{\mathbb{T}_N}: \xi_x=1\right\}=\alpha.
\end{equation*}
However, in our case the underlying exclusion process is not simple. Actually, it is not even Markovian, as its jump rates depend on the value of the integral of the corresponding height configuration, which is not ``seen" by the particles alone (although it is easy to see that
the joint process $\{(h_t^N(0), \xi_t^N), t\ge0\}$ is Markovian). In fact, the stationary measure of the weakly perturbed height process is not product when projected onto the particles, and is given by the following.

\begin{lemma}\label{lemma:invariant_measure} Define the measure $\mu_N^*=\mu_{N, \gamma}^*$ on $\Omega_N$ as 
\begin{equation}\label{eq:inv_measure}
    \mu_N^*(h):=\frac{1}{\mathcal{Z}_N}e^{-N^{-\gamma}|Y(h)|},
\end{equation}
where $\mathcal{Z}_N=\mathcal{Z}_{N, \gamma}$ is a normalising constant. Then $\mu_N^*$ is invariant for the weakly perturbed height process $\{h_t^N, t\ge0\}$. 
\end{lemma}

\begin{proof}
By \cite[Proposition~I.2.13]{lig05}, $\mu^*_N$ is invariant for $\{h_t^N, t\ge0\}$ if and only if
\begin{equation}\label{eq:liggett_invariance}
    \int_{\Omega_N}\mathcal{L}^N f\de \mu^*_N=0
\end{equation}
for each function $f$ on $\Omega_N$.
Note that the above is equivalent to
\begin{equation}\label{eq:invariance_equiv}
    \begin{split}
    &\int_{\Omega_N} \sum_{x, y\in\mathbb{T}_N} f(h^{x, y})\xi_x(1-\xi_y)p_N(x, y, h)\de \mu^*_N
    \\&=\int_{\Omega_N}\sum_{x, y\in\mathbb{T}_N} f(h)\xi_x(1-\xi_y)p_N(x, y, h)\de \mu^*_N,
    \end{split}
\end{equation}
where we set $p_N(x, y, h)=0$ for $|x-y|\ne 1$. By performing the change of variable $h^{x, y}\mapsto h$ and exchanging the variables $x$ and $y$ in the first line of \eqref{eq:invariance_equiv}, we get that \eqref{eq:invariance_equiv} is equivalent to
\begin{equation*}
    \begin{split}
    &\int_{\Omega_N}\sum_{x, y\in\mathbb{T}_N} f(h)\xi_x(1-\xi_y)p_N(y, x, h^{x, y})\mu^*_N(h^{x, y})
    \\&=\int_{\Omega_N}\sum_{x, y\in\mathbb{T}_N} f(h)\xi_x(1-\xi_y)p_N(x, y, h)\mu^*_N(h).
    \end{split}
\end{equation*}
Hence, it suffices to prove that, for each $h\in\Omega_N$ and each $x, y\in\mathbb{T}_N$,
\begin{equation*}
    p_N(y, x, h^{x, y})\mu_N^*(h^{x, y})=p_N(x, y, h)\mu_N^*(h).
\end{equation*}
This is easily verified by considering the cases $\sgn(Y(h^{x, y}))=\sgn(Y(h))$ and $\sgn(Y(h^{x, y}))$ $=-\sgn(Y(h))$.
\end{proof}

%%%%%%%%%%%%%%%%%%%%%%%%%%%%%%%%%%%%%%%%%%%%%%%%%%
%%%Hydrodynamic Limit%%%%%%%%%%%%%%%%%%%%%%%%%%%%%
%%%%%%%%%%%%%%%%%%%%%%%%%%%%%%%%%%%%%%%%%%%%%%%%%%
\subsection{Hydrodynamic Limit}
Our first result is a description of the evolution of the particle empirical measure as a solution of a partial differential equation (PDE). We will obtain the solutions of this PDE in a weak sense, so in order to define them we first need to specify a space of test functions. Let $\mathbb{T} := [0, 1]/\{0, 1\}$ denote the one-dimensional torus, and let $C_0^{m,n}([0, T]\times\mathbb{T})$ be the set of continuous functions on $[0, T]\times\mathbb{T}$ which vanish at the terminal time $T$ and are $m$ times differentiable in the first variable and $n$ in the second, with continuous derivatives, where $m$ and $n$ are positive integers. We will also denote by $C_0^m([0, T], \mathbb{R})$ the set of continuous functions on $[0, T]$ which vanish at the terminal time $T$ and are $m$ times differentiable, with continuous derivatives, and by $C([0, T], \mathbb{R})$ the set of continuous functions on $[0, T]$. Moreover, let $\mathcal{H}^1$ denote the Sobolev space on $\mathbb{T}$, equipped with the norm
\begin{equation*}
    \|\cdot\|_{\mathcal{H}^1}^2:=\|\cdot\|_{L^2(\mathbb{T})}^2+\|\nabla \cdot\|_{L^2(\mathbb{T})}^2.
\end{equation*}
We will denote by $L^2([0, T], \mathcal{H}^1)$ the set of measurable functions $f:[0, T]\to\mathcal{H}^1$ with $\int_0^T\|f_t\|_{\mathcal{H}^1}^2\de t<\infty$.

\begin{definition}\label{def:weak_solutions} For $f, g:\mathbb{T}\to\mathbb{R}$ measurable functions, let $\langle f, g\rangle:=\int_\mathbb{T}f(u)g(u)\de u$. Also, let $\rho_0:\mathbb{T}\to[0, 1]$ be measurable and let $Y_0\in\mathbb{R}$. We say that the couple $(\rho, Y)$, with $\rho:[0, T]\times\mathbb{T}\to[0, 1]$ and $Y:[0, T]\to\mathbb{R}$, is a weak solution of the coupled equations 
\begin{equs}[eq:coupled_equations]
    \partial_t\rho_t &=\frac{1}{2}\Delta \rho_t-\sgn(Y_t)\nabla[\rho_t(1-\rho_t)],
    \\ \partial_t Y_t&=-2\sgn(Y_t)\langle\rho_t, 1-\rho_t\rangle, 
\end{equs}
with initial condition $(\rho_0, Y_0)$ if:
\begin{enumerate}[i)]
    \item $\rho\in L^2([0, T], \mathcal{H}^1)$ and $Y\in C([0, T], \mathbb{R})$, 
    
    \item for any $\phi\in C_0^{1, 2}([0, T]\times\mathbb{T})$,
    \begin{equation*}
        \langle\rho_0, \phi_0\rangle+\int_0^T\left\langle\rho_t, \left(\frac{1}{2}\Delta+\partial_t\right)\phi_t\right\rangle\de t+\int_0^T\sgn(Y_t)\left\langle\rho_t(1-\rho_t), \nabla\phi_t\right\rangle\de t=0, 
    \end{equation*}
    
    \item for all $\varphi\in C_0^1([0, T], \mathbb{R})$,
    \begin{equation*}
        Y_0\varphi_0+\int_0^T Y_t\partial_t \varphi_t\de t-\int_0^T2\sgn(Y_t)\left\langle\rho_t, 1-\rho_t\right\rangle\varphi_t \de t=0.
    \end{equation*}
\end{enumerate}
\end{definition}

\begin{lemma}\label{lemma:uniqueness_pde}
The system of coupled equations \eqref{eq:coupled_equations} has a unique weak solution.
\end{lemma}
The proof of this lemma is postponed to Appendix~\ref{sec:app_energy_uniqueness}.

Define now the empirical measure $\pi_t^N := \frac{1}{N}\sum_{x\in\mathbb{T}_N} \xi_t^N(x)\delta_{\frac{x}{N}}$.
Fix $\gamma=1$ and a finite time horizon $[0, T]$, and let $C(\mathbb{T})$ denote the space of continuous functions on the torus. Given a function $\phi\in C(\mathbb{T})$, let 
\begin{equation*}
    \langle \pi_t^N, \phi\rangle:=\frac{1}{N}\sum_{x\in\mathbb{T}_N} \xi_t^N(x)\phi\Big(\frac{x}{N}\Big)
\end{equation*}
denote the integral of $\phi$ with respect to the measure $\pi_t^N$ (not to be confused with the inner product in Definition \ref{def:weak_solutions}). Given two measures $\mu, \nu$ on $\Omega_N$, we denote by  $H(\mu\|\nu)$ their \textit{relative entropy}
\begin{equation*}
    H(\mu\|\nu):=\int_{\Omega_N }\mylog\left(\frac{\mu}{\nu}\right)\de \mu.
\end{equation*}
Throughout, we will denote by $\prob_\mu$ the probability measure that the process $\{h^N_t, t\in[0, T]\}$ induces on the space of càdlàg functions $D([0, T], \Omega_N)$ when starting from a measure $\mu$ on $\Omega_N$.

\begin{theorem}[Hydrodynamic Limit]\label{thm:hydro} Let $N=2n$ with $n$ odd and let $\{\mu_N\}_N$ be a sequence of measures on $\Omega_N$.
Assume that $\{\mu_N\}_N$ is associated to the initial height profile $h_0\in \mathcal{H}^1$: 
\begin{equation}\label{eq:initial_profile}
    \mylim_{N\to\infty}\mu_N\left\{\left|\frac{1}{N^2}\sum_{x\in\mathbb{T}_N} h(x)\psi\left(\frac{x}{N}\right)-\int_\mathbb{T}h_0(u)\psi(u)\de u\right|>\delta\right\}=0
\end{equation}
for any $\delta>0$ and $\psi\in C(\mathbb{T})$. Also, assume that there exists a constant $C_0$ such that the relative entropy of $\mu_N$ with respect to the invariant measure $\mu_N^*$ is bounded by $C_0N$:
\begin{equation}\label{eq:entropy_bound}
    H(\mu_N\|\mu_N^*)\le C_0N.
\end{equation}
Then, for any $t\in[0, T]$, any $\phi\in C(\mathbb{T})$ and any $\delta>0$,
\begin{equation*}
    \mylim_{N\to\infty}\prob_{\mu_N}\left\{\left|\left\langle \pi_t^N, \phi\right\rangle-\int_\mathbb{T}\rho_t(u)\phi(u)\de u\right|>\delta\right\}=0,
\end{equation*}
where $(\rho, Y)$ is the unique weak solution of \eqref{eq:coupled_equations} started at $\left(\frac{\nabla h_0+1}{2}, \int_\mathbb{T} h_0(u)\de u\right)$.
\end{theorem}

The result above means the following. At first \dash assuming that the integral of the initial condition $h_0$ is not zero \dash the integral of the height profile moves towards zero with a drift proportional to the average compressibility of the system, while the density profile evolves according to the viscous Burgers equation. This aligns with the fact that, until the integral changes sign, the particle system 
coincides with a WASEP. Once the integral hits zero, it remains there and the density profile evolves according to the heat equation. Roughly speaking, this means that, after the discrete interface changes sign for the first time, it keeps switching sign on a faster time scale than the diffusive one, so that the nonlinearity in the Burgers equation flips sign very quickly and thus ``cancels out" in the hydrodynamic equation.

\begin{remark} An example of a sequence of initial measures $\{\mu_N\}_N$ satisfying the conditions of Theorem~\ref{thm:hydro} is given by the following. Let $\rho_0:=\frac{\nabla h_0+1}{2}$, and let $\nu_{\rho_0, N}$ denote the product measure on $\{0, 1\}^N$ with marginals given by Bernoulli random variables with parameters $\rho_0\left(\frac{\cdot}{N}\right)$, namely
\begin{equation}
    \nu_{\rho_0, N}(\xi):=\prod_{x\in\mathbb{T}_N} \rho_0\left(\frac{x}{N}\right)^{\xi(x)}\left(1-\rho_0\left(\frac{x}{N}\right)\right)^{1-\xi(x)}.
\end{equation}
We define the probability measure $\mu_N$ on $\Omega_N$ via
\begin{equation*}
    \mu_N(h):=\frac{1}{\mathcal{T}_N} e^{-|h(0)-Nh_0(0)|} \nu_{\rho_0, N}(\xi(h)),
\end{equation*}
with $\mathcal{T}_N$ partition function. Recall \eqref{eq:inv_measure} and  note that
\begin{align*}
    H(\mu_N\| \mu_N^*)&=\sum_{h\in\Omega_N}\mylog\left(\frac{\mu_N(h)}{\mu_N^*(h)}\right)\mu_N(h)\nonumber
    \\&\le\mylog\left(\frac{\mathcal{Z}_N}{\mathcal{T}_N}\right)+\sum_{h\in\Omega_N}\left(\frac{|Y(h)|}{N}-|h(0)-Nh_0(0)|\right)\mu_N(h).
\end{align*}
Now, it is not hard to see that $\mathcal{Z}_N\lesssim 2^N$, and moreover,
\begin{equation*}
    \mathcal{T}_N=\sum_{j=-\infty}^\infty e^{-|j-Nh_0(0)|}\ge \sum_{j=1}^\infty e^{-j}=\frac{1}{e-1}.
\end{equation*}
On the other hand, note that
\begin{equation*}
    \frac{|Y(h)|}{N}-|h(0)-Nh_0(0)|\le \frac{N}{2}+\frac{|h(0)|}{2}-|h(0)|+N|h_0(0)|\le \left(\frac{1}{2}+|h_0(0)|\right)N,
\end{equation*}
and hence \eqref{eq:entropy_bound} is verified. As for \eqref{eq:initial_profile}, it suffices to note that, for any $\delta>0$,
\begin{equation*}
    \mylim_{N\to\infty}\mu_N\left\{\left|\frac{h(0)}{N}-h_0(0)\right|>\delta\right\}=0.
\end{equation*}
Indeed, denoting by $\expected_{\mu_N}[\,\cdot\,]$ the expectation taken under $\mu_N$,
\begin{equation*}
    \expected_{\mu_N}\left|\frac{h(0)}{N}-h_0(0)\right|=\frac{1}{\mathcal{T}_N}\sum_{j=-\infty}^\infty e^{-|j-Nh_0(0)|}\left|\frac{j}{N}-h_0(0)\right|\lesssim \frac{1}{N}\sum_{j=0}^\infty je^{-j},
\end{equation*}
which vanishes as $N\to\infty$.
\end{remark}

\begin{remark} Note that Theorem~\ref{thm:hydro} characterises the hydrodynamic equation of the scaled interface profile $\{\frac{1}{N}h_t^N(Nx), x\in\mathbb{T}, t\in[0, T]\}$ as the PDE
\begin{equation*}
    \partial_t h_t = \frac{1}{2}\Delta h_t+\frac{1}{2}\sgn(Y_0)[1-(\nabla h_t)^2]\boldsymbol{1}_{\{t\le \tau_0\}}
\end{equation*}
started at $h_0$, where $Y_0:=\int_\mathbb{T} h_0(u)\de u$ and $\tau_0:=\myinf\{t\ge0: \int_\mathbb{T} h_t(u)\de u=0\}$.
This can be seen by noting that the limit $h:[0, T]\times\mathbb{T}\to\mathbb{R}$ must satisfy $\rho_t=\frac{\nabla h_t+1}{2}$ and $Y_t=\int_\mathbb{T}h_t(u)\de u$, where $(\rho, Y)$ are as in Theorem~\ref{thm:hydro}.
\end{remark}

%%%%%%%%%%%%%%%%%%%%%%%%%%%%%%%%%%%%%%%%%%%%%%%%%%
%%%Equilibrium Fluctuations%%%%%%%%%%%%%%%%%%%%%%%
%%%%%%%%%%%%%%%%%%%%%%%%%%%%%%%%%%%%%%%%%%%%%%%%%%
\subsection{Equilibrium Fluctuations}
The next natural step is to investigate the fluctuations of the density field around its mean, which is the content of our next result. We assume that the initial distribution is the invariant measure $\mu_N^*$, so that the height process is stationary. Throughout, we will write $\expected_{\mu_N^*}[\,\cdot\,]$ for expectations taken either under the measure $\mu_N^*$ on $\Omega_N$, or under the probability $\prob_{\mu_N^*}$ on the path space; the intended meaning will always be clear from context. Let $\mathcal{S}(\mathbb{T})$ denote the space of $\mathbb{R}$-valued smooth functions on the torus, and define the density fluctuation field $\{\mathscr{U}_t^N, t\in[0, T]\}$ as the $\mathcal{S}'(\mathbb{T})$-valued process such that, for any $\phi\in\mathcal{S}(\mathbb{T})$,
\begin{equation}\label{eq:fluct_field}
    \mathscr{U}_t^N(\phi):=\frac{1}{\sqrt{N}}\sum_{x\in\mathbb{T}_N}\Big(\xi_t^N(x)-\expected_{\mu_N^*}\left[\xi_t^N(x)\right]\Big)\phi\left(\frac{x}{N}\right).
\end{equation}
Since $\mu_N^*$ is invariant under reflection with respect to the $x$-axis, we have $\expected_{\mu_N^*}\left[\xi_t^N(x)\right]=\frac{1}{2}$ for each $x\in\mathbb{T}_N$. We will denote by $\bar\xi_t^N(x)$ the centred random variable $\xi_t^N(x)-\frac{1}{2}$, so that \eqref{eq:fluct_field} becomes
\begin{equation*}
    \mathscr{U}_t^N(\phi)=\frac{1}{\sqrt{N}}\sum_{x\in\mathbb{T}_N}\bar\xi_t^N(x)\phi\left(\frac{x}{N}\right).
\end{equation*}

\begin{definition}
Let $\lambda\in\mathbb{R}$ and $ \sigma>0$. We say that an $\mathcal{S}'(\mathbb{T})$-valued stochastic process $\{\mathscr{U}_t, t\in[0, T]\}$ with continuous trajectories is a \textit{stationary solution of the infinite-dimensional Ornstein-Uhlenbeck equation}
\begin{equ}\label{eq:OU}
    \partial_t\mathscr{U}_t=\lambda\Delta\mathscr{U}_t+\sigma\nabla\dot{ \mathscr{W}}_t
\end{equ}
if it is stationary and, for any $\phi\in\mathcal{S}(\mathbb{T})$, the process $\{\mathscr{M}_t(\phi), t\in[0, T]\}$ defined as
\begin{equation*}
    \mathscr{M}_t(\phi)=\mathscr{U}_t(\phi)-\mathscr{U}_0(\phi)-\lambda\int_0^t\mathscr{U}_s(\Delta\phi)\de s
\end{equation*}
is a continuous martingale with respect to the natural filtration of $\{\mathscr{U}_t, t\in[0, T]\}$ with quadratic variation\footnote{Throughout this article, we will always refer to the predictable quadratic
variations as simply the quadratic variation.}
\begin{equation*}
    \scal{\mathscr{M}(\phi)}_t=t\sigma^2\|\nabla\phi\|^2_{L^2(\mathbb{T})}.
\end{equation*}
\end{definition}

Let $D\left([0, T], \mathcal{S}'(\mathbb{T})\right)$ denote the space of càdlàg functions on $[0, T]$ taking values in $\mathcal{S}'(\mathbb{T})$, equipped with the Skorokhod topology. Let $\mathcal{Q}^N$ denote the probability measure on $D\left([0, T], \mathcal{S}'(\mathbb{T})\right)$ given by the law of the field $\{\mathscr{U}^N_t, t\in[0, T]\}$. 
When $\gamma$ is greater than $\frac{6}{7}$, the sequence of fields $\{\mathscr{U}^N_t, t\in[0, T]\}_N$ converges to an infinite-dimensional stationary Ornstein-Uhlenbeck process:

\begin{theorem}[Equilibrium Fluctuations]\label{thm:flucts} Let $N=2p$ with $p$ prime. For each $\gamma>\frac{6}{7}$, the sequence of probability measures $\{\mathcal{Q}^N\}_N$ converges weakly in $D([0, T], \mathcal{S}'(\mathbb{T}))$ to the law of the stationary solution $\{\mathscr{U}_t, t\in [0, T]\}$ to \eqref{eq:OU} with $\lambda=\sigma=\frac{1}{2}$ and $\mathscr{U}_t(1)=0$. Moreover, for any $\phi\in\mathcal{S}(\mathbb{T})$, the variance of $\mathscr{U}_t^N(\phi)$ converges to $\frac{1}{4}\big(\|\phi\|^2_{L^2(\mathbb{T})}-\langle1, \phi\rangle^2\big)$. 
\end{theorem}

The threshold $\frac{6}{7}$ is purely technical and we do not expect it to be optimal. In fact we conjecture that the convergence holds at least for $\gamma>\frac{1}{2}$. This is supported by the fact that, for $\gamma>\frac{1}{2}$, it holds for the classical WASEP, and our process is even ``more symmetric'', which suggests that the non-linearity should cancel for at least the same values of $\gamma$. Yet, unlike the classical WASEP, the invariant measure of the perturbed process depends on the strength of the asymmetry. As $\gamma$ decreases, the measure concentrates more and more near the ``bad'' configurations where the integral keeps switching sign; this could give rise to an additional, zero-average stochastic term in \eqref{eq:OU} for a critical $\gamma$. 

The result on the variance of the fluctuation field implies that if $\phi$ is constant then the variance vanishes in the limit. This is consistent with the fact that, actually, if $\phi$ is constant, then the fluctuation field vanishes itself, as any height configuration must satisfy $\sum_{x\in\mathbb{T}_N}\bar\xi^N(x)=0$.

\begin{remark} The unusual hypothesis $N=2p$ with $p$ prime comes from the fact that we managed to prove all the correlation bounds that are needed to show this result (Theorems~\ref{thm:2m_correlations}, \ref{thm:restriction>1}, \ref{thm:restriction1} and~\ref{thm:2correlation_limit}, all proved in Appendix~\ref{sec:app_correlations}) only for this particular class of values of $N$. The proof provided makes use of the fact that, given a primitive $p$-th root of unity $\omega$ with $p$ prime, a linear combination of $1, \omega, \ldots, \omega^{p-1}$ can only be zero if all the coefficients are identical. This is not true for general $p$. However, we believe these asymptotic bounds (and hence Theorem~\ref{thm:flucts}) to hold for any (even) $N$.
\end{remark}

\begin{remark}
Theorem~\ref{thm:flucts} only characterises the fluctuations of the \textit{gradient} of the interface profile. As for the fluctuations of the height profile itself, define the \textit{interface fluctuation field} $\{\mathscr{H}_t^N(u), u\in\mathbb{T}, t\in[0, T]\}$ by
\begin{equation*}
    \mathscr{H}_t^N(u):=\frac{1}{\sqrt{N}}h_t^N(Nu).
\end{equation*}
Since $\langle \mathscr{H}_t^N, \nabla_N \phi\rangle=-2\mathscr{U}_t^N(\phi)$ for each $\phi\in\mathcal{S}(\mathbb{T})$, 
it follows from Theorem~\ref{thm:flucts} that the limit of the interface fluctuation field is completely determined by knowledge of the limit of the integral field $\{\mathscr{Y}_t^N, t\in[0, T]\}$ defined by $\mathscr{Y}_t^N:=N^{-\frac{3}{2}}\sum_{x\in\mathbb{T}_N}h_t^N(x)$. One can check that, when space is scaled by $N^{-\alpha}$ (namely, the sum of the heights over $\mathbb{T}_N$ is divided by $N^\alpha)$ and time is accelerated by $N^\beta$, the process $\{\mathscr{\tilde{X}}_t^N, t\in[0, T]\}$ defined by
\begin{equation}\label{eq:tildeX}
    \mathscr{\tilde X}_t^N:=\mathscr{\tilde{Y}}_t^N-\mathscr{\tilde{Y}}_0^N+\frac{N^\beta}{N^{\alpha}}\text{tanh}\left(\frac{1}{N^\gamma}\right)\int_0^t \sgn(\mathscr{Y}_s^N)\sum_{x\in\mathbb{T}_N}[\bar\xi^N_s(x)-\bar\xi^N_s(x+1)]^2\de s
\end{equation}
is a mean-zero martingale with quadratic variation
\begin{equation}\label{eq:[tildeX]}
    \scal{\mathscr{\tilde{X}}^N}_t=\frac{N^\beta}{N^{2\alpha}}\int_0^t\sum_{x\in\mathbb{T}_N}\left[\bar\xi^N_s(x)-\bar\xi^N_s(x+1)\right]^2 \de s.
\end{equation}
Equations \eqref{eq:tildeX} and \eqref{eq:[tildeX]} suggest that the ``correct'' space-time scale for the integral process is $\alpha=\gamma$ and $\beta=2\gamma-1$ (when $\gamma>\frac{1}{2})$, so that the drift term of  the integral field $\mathscr{Y}_t^N$ is given by
\begin{equation*}
    \begin{split}&\frac{1}{2}N^{\frac{3}{2}}\,\text{tanh}\left(\frac{1}{N^\gamma}\right)\int_0^t\sgn(\mathscr{Y}_s^N)\de s
    \\&-2\sqrt{N}\,\text{tanh}\left(\frac{1}{N^\gamma}\right)\int_0^t\sgn(\mathscr{Y}_s^N)\sum_{x\in\mathbb{T}_N}\bar\xi_s^N(x)\bar\xi_s^N(x+1)\de s.
    \end{split}
\end{equation*}
Hence, we get the following:
\begin{enumerate}[i)]
    
    \item $\gamma>\frac{3}{2}$: the drift term vanishes as $N\to\infty$, so the limit $\{\mathscr{H}_t, t\in[0, T]\}$ of the interface fluctuation field is the stationary solution of the stochastic heat equation 
    \begin{equation}\label{eq:she}
        \partial_t\mathscr{H}_t=\frac{1}{2}\Delta\mathscr{H}_t+\frac{1}{2}\dot{\mathscr{W}}_t;
    \end{equation}
    
    \item $\gamma=\frac{3}{2}$: one can check that \eqref{eq:tildeX} and \eqref{eq:[tildeX]} imply that the limit $\{\mathscr{Y}_t, t\in[0, T]\}$ of the integral field satisfies
    \begin{equation}\label{eq:SDE_lim}
        \mathscr{Y}_t=\mathscr{Y}_0-\frac{1}{2}\mylim_{N\to\infty}\int_0^t\sgn(\mathscr{Y}_s^N)\de s+\frac{1}{2}\mathscr{B}_t,
    \end{equation}
    with $\{\mathscr{B}_t, t\in[0, T]\}$ standard Brownian motion. Using the continuity of $\{\mathscr{Y}_t, t\in[0, T]\}$, one can check that, for each $\delta>0$, the limit above must satisfy
    \begin{equation}\label{eq:sign_sandwich}
        \int_0^t\sgn(\mathscr{Y}_s-\delta)\de s\le \mylim_{N\to\infty}\int_0^t\sgn(\mathscr{Y}_s^N)\de s \le \int_0^t\sgn(\mathscr{Y}_s+\delta)\de s
    \end{equation}
    almost surely. Consider now the SDE
    \begin{equation}\label{eq:SDE}
        \mathscr{Y}_t=\mathscr{Y}_0-\frac{1}{2}\int_0^t\sgn(\mathscr{Y}_s)\de s+\frac{1}{2}\mathscr{B}_t
    \end{equation}
    and let $\{\mathscr{Y}^{-\delta}_t, t\in[0, T]\}$ and $\{\mathscr{Y}^{+\delta}_t, t\in[0, T]\}$ be the solutions to the SDEs
    \begin{align*}
        &\mathscr{Y}_t^{-\delta}=\mathscr{Y}_0-\frac{1}{2}\int_0^t\sgn(\mathscr{Y}_s^{-\delta}-\delta)\de s+\frac{1}{2}\mathscr{B}_t,
        \\&\mathscr{Y}_t^{+\delta}=\mathscr{Y}_0-\frac{1}{2}\int_0^t\sgn(\mathscr{Y}_s^{+\delta}+\delta)\de s+\frac{1}{2}\mathscr{B}_t,
    \end{align*}
    respectively, where $\{\mathscr{B}_t, t\in[0, T]\}$ denotes the \textit{same} Brownian motion as in \eqref{eq:SDE_lim}. Then, following a similar argument to the one given for Lemma~\ref{lemma:sgn}, one can check that \eqref{eq:sign_sandwich} yields $\mathscr{Y}_t^{+\delta}\le \mathscr{Y}_t\le \mathscr{Y}_t^{-\delta}$ almost surely. By \cite[Theorem~3.2]{pug08}, this implies that $\{\mathscr{Y}_t, t\in[0, T]\}$ must be the unique solution to \eqref{eq:SDE}, so that $\{\mathscr{H}_t, t\in[0, T]\}$ is the 
    stationary solution of the SPDE
    \begin{equation*}
        \partial_t\mathscr{H}_t=\frac{1}{2}\Delta\mathscr{H}_t-\frac{1}{2}\sgn(\langle\mathscr{H}_t, 1\rangle)\de t+\frac{1}{2}\dot{\mathscr{W}}_t;
    \end{equation*}
    
    \item $\frac{6}{7}<\gamma<\frac{3}{2}$: by the results in Appendix~\ref{sec:app_correlations}, one can check that (for $N=2p$ with $p$ prime) $\expected_{\mu_N^*}\left[(\mathscr{Y}_t^N)^2\right]\lesssim N^{2\gamma-3}$. This implies that $\left\{\mathscr{H}_t, t\in[0, T]\right\}$ is the stationary solution of \eqref{eq:she} with $\langle \mathscr{H}_t, 1\rangle=0$.
\end{enumerate}
\end{remark}

%%%%%%%%%%%%%%%%%%%%%%%%%%%%%%%%%%%%%%%%%%%%%%%%%%
%%%HYDRODYNAMIC LIMIT%%%%%%%%%%%%%%%%%%%%%%%%%%%%%
%%%%%%%%%%%%%%%%%%%%%%%%%%%%%%%%%%%%%%%%%%%%%%%%%%
\section{Hydrodynamic Limit}\label{sec:hydro}
Our strategy to show Theorem~\ref{thm:hydro} is the following. Let $\mathbb{M}^+$ denote the space of 
non-negative measures on $\mathbb{T}$ endowed with the weak convergence topology, and let $Q^N$ be the probability measure on the space of càdlàg functions $D([0, T], \mathbb{M}^+)$, endowed with the Skorokhod topology, induced by the process $\{\pi^N_t, t\in[0, T]\}$ and initial measure $\mu_N$. We start by writing Dynkin's martingales and we show tightness of the sequence of measures $\{Q^N\}_N$ in $D([0, T], \mathbb{M}^+)$ and the existence of a density $\{\rho_t, t\in[0, T]\}$ of limit paths with respect to the Lebesgue measure. By Prokhorov's Theorem, a unique characterisation of limit points then yields convergence in distribution of the sequence. 

In order to uniquely characterise this limit point, we study the (appropriately scaled) integral of the height process $\{Y_t^N, t\in[0, T]\}$ as well as its sign $\{\sgn(Y_t^N), t\in[0, T]\}$, show tightness of the corresponding sequences of measures (Proposition~\ref{prop:integral}), and then show in Proposition~\ref{prop:limit_points} that any limit point of the joint law of $(\pi_t^N, Y_t^N, \sgn(Y_t^N))$ is concentrated on paths $(\pi_t, Y_t, Z_t)$ satisfying $Z_t\in\sigma(Y_t)$ as well as
\begin{equs}
    \partial_t\rho_t &=\frac{1}{2}\Delta \rho_t-Z_t\nabla[\rho_t(1-\rho_t)],
    \\ \partial_t Y_t&=-2Z_t\langle\rho_t, 1-\rho_t\rangle,
\end{equs}
in the sense of Definition \ref{def:weak_solutions}, where $\sigma:\mathbb{R}\to\mathcal{P}([-1, 1])$ is given by
\begin{equation}\label{def:sigma_sgn}
    \sigma(y):=\begin{cases} \{\sgn(y)\} & \text{if\ } y\ne0,
    \\ [-1, 1] & \text{if\ } y=0.
    \end{cases}
\end{equation}
We will then show that this implies that $Z_t=\sgn(Y_t)$, yielding the coupled equations \eqref{eq:coupled_equations}. Uniqueness of solutions will then complete the proof.

The most technical part of our proof is Lemma~\ref{lemma:replacement}, stated and proved in Section~\ref{sec:replacement}, which is a variation of the well-known replacement lemma introduced in \cite{gpv88}. As we will see, compared to the usual replacement lemma, the ``switching" of the weak asymmetry in our process will require some additional arguments.

%%%%%%%%%%%%%%%%%%%%%%%%%%%%%%%%%%%%%%%%%%%%%%%%%%
%%%Dynkin's Martingales%%%%%%%%%%%%%%%%%%%%%%%%%%%
%%%%%%%%%%%%%%%%%%%%%%%%%%%%%%%%%%%%%%%%%%%%%%%%%%
\subsection{Dynkin's Martingales}
By Dynkin's formula (see for example \cite[Appendix~1]{kl99}), for any $\phi\in C^{1, 2}([0, T]\times \mathbb{T})$ the process $\{M^N_t(\phi), t\in[0, T]\}$ defined via
\begin{equation}\label{eq:dynkin_1}
    M_t^N(\phi):=\langle \pi_t^N, \phi_t\rangle-\langle \pi_0^N, \phi_0\rangle-\int_0^t\left(N^2\mathcal{L}^N+\partial_s\right)\langle \pi_s^N, \phi_s\rangle\de s
\end{equation}
is a martingale with respect to the natural filtration of $\{h^N_t, t\in[0, T]\}$ with quadratic variation
\begin{equation}\label{eq:quad_var_1}
    \scal{M^N(\phi)}_t=\int_0^t \left(N^2\mathcal{L}^N\langle\pi_s^N, \phi\rangle^2-2N^2\langle\pi_s^N, \phi\rangle\mathcal{L}^N\langle\pi_s^N, \phi\rangle\right)\de s.
\end{equation}
Given $\psi\in C(\mathbb{T})$, define the process $\{\theta_t^N(\psi), t\in[0, T]\}$ as
\begin{equation*}
    \theta_t^N(\psi):=\frac{1}{N^2}\sum_{x\in\mathbb{T}_N} h_t^N(x)\psi\left(\frac{x}{N}\right).
\end{equation*}
Then, recalling that $\xi_t^N(x)=\frac{h_t^N(x)-h_t^N(x-1)+1}{2}$, for any $\phi\in C(\mathbb{T})$ a summation by parts yields
\begin{align*}
    \theta_t^N(\nabla_N\phi)&=\frac{1}{N}\sum_{x\in\mathbb{T}_N} h^N_t(x)\left\{\phi\left(\frac{x+1}{N}\right)-\phi\left(\frac{x}{N}\right)\right\}
    \\&=\frac{1}{N}\sum_{x\in\mathbb{T}_N} \left\{h^N_t(x-1)-h^N_t(x)\right\}\phi\left(\frac{x}{N}\right)
    \\&=-\frac{2}{N}\sum_{x\in\mathbb{T}_N} \xi_t^N(x)\phi\left(\frac{x}{N}\right)+\frac{1}{N}\sum_{x\in\mathbb{T}_N}\phi\left(\frac{x}{N}\right)
\end{align*}
and hence we have the identity
\begin{equation}\label{eq:theta_Phi}
    \theta_t^N\left(\nabla_N\phi\right)=-2\left\langle\pi_t^N, \phi\right\rangle+\frac{1}{N}\sum_{x\in\mathbb{T}_N}\phi\left(\frac{x}{N}\right).
\end{equation}
We will denote by $\{Y_t^N, t\in[0, T]\}$ the \textit{integral process}, defined as
\begin{equation}\label{eq:integral_process}
    Y_t^N:=\theta_t^N(1)=\frac{1}{N^2}\sum_{x\in\mathbb{T}_N} h_t^N(x).
\end{equation}
Using the identity $\frac{1}{1+e^{-x}}=\frac{1}{2}(1+\text{tanh}(\frac{x}{2}))$ and the fact that $\text{tanh}(\sigma x)=\sigma\,\text{tanh}(x)$ for $\sigma\in\{-1, 1\}$, we can rewrite the rates \eqref{eq:rates_up_down} as
\begin{equation}\label{eq:rates_tanh}
    p_N^{\downarrow}(h)=\frac{1}{2}+\frac{1}{2}\sgn(Y(h))\text{tanh}\left(\frac{1}{N}\right), \quad p_N^{\uparrow}(h)=\frac{1}{2}-\frac{1}{2}\sgn(Y(h))\text{tanh}\left(\frac{1}{N}\right).
\end{equation}
Then, a simple computation shows that \eqref{eq:dynkin_1} can be written as
\begin{equation}\label{eq:martingale_1}
    \begin{split}
    &M_t^N(\phi)=\langle \pi_t^N, \phi_t\rangle-\langle \pi_0^N, \phi_0\rangle-\int_0^t\left\langle\pi_s^N, \left(\frac{1}{2}\Delta_N+\partial_s\right)\phi_s\right\rangle\de s
    \\&+\frac{1}{2}\int_0^t\sum_{x\in\mathbb{T}_N}\text{tanh}\left(\frac{1}{N}\right)\nabla_N\phi_s\left(\frac{x}{N}\right)\sgn(Y^N_s)\left[\xi_s^N(x)-\xi_s^N(x+1)\right]^2\de s,
    \end{split}
\end{equation}
where we set
\begin{equation}\label{eq:gradient}
    \nabla_N\phi\left(\frac{x}{N}\right):=N\left\{\phi\left(\frac{x+1}{N}\right)-\phi\left(\frac{x}{N}\right)\right\}
\end{equation}
and
\begin{equation}\label{eq:laplacian}
    \Delta_N\phi\left(\frac{x}{N}\right):=N^2\left\{\phi\left(\frac{x+1}{N}\right)-2\phi\left(\frac{x}{N}\right)+\phi\left(\frac{x-1}{N}\right)\right\}.
\end{equation}
Let $\{K^N_t(\phi), t\in[0, T]\}$ and $\{B_t^N(\phi), t\in[0, T]\}$ be the processes defined by
\begin{equation*}
    K_t^N(\phi):=\int_0^t\left\langle\pi_s^N, \left(\frac{1}{2}\Delta_N+\partial_s\right)\phi_s\right\rangle\de s
\end{equation*}
and
\begin{equation*}
    B_t^N(\phi):=-\frac{1}{2}\int_0^t\sum_{x\in\mathbb{T}_N}\text{tanh}\left(\frac{1}{N}\right)\nabla_N\phi_s\left(\frac{x}{N}\right)\sgn(Y^N_s)\left[\xi_s^N(x)-\xi_s^N(x+1)\right]^2\de s,
\end{equation*}
so that \eqref{eq:martingale_1} becomes
\begin{equation}\label{eq:four_terms}
    \langle \pi_t^N, \phi_t\rangle=\langle \pi_0^N, \phi_0\rangle+M_t^N(\phi)+K_t^N(\phi)+B_t^N(\phi).
\end{equation}
In the following sections we will study separately each of the terms appearing in \eqref{eq:four_terms}.

%%%%%%%%%%%%%%%%%%%%%%%%%%%%%%%%%%%%%%%%%%%%%%%%%%
%%%Tightness%%%%%%%%%%%%%%%%%%%%%%%%%%%%%%%%%%%%%%%
%%%%%%%%%%%%%%%%%%%%%%%%%%%%%%%%%%%%%%%%%%%%%%%%%%
\subsection{Tightness}
\begin{proposition}\label{prop:tightness_1}
Let $\{\mu_N\}_N$ satisfy \eqref{eq:initial_profile}. The sequence of measures $\{Q^N\}_N$ is tight in $D([0, T], \mathbb{M}^+)$ with respect to the Skorokhod topology, and all limit points are concentrated on measures which are absolutely continuous with respect to the Lebesgue measure and with density bounded by 1.
\end{proposition}

We will actually show tightness of all four terms on the RHS of \eqref{eq:four_terms} by applying the following result.

\begin{lemma}\label{lemma:tightness}
Let $\{X^N\}_N=\{X_t^{N}, t\in[0, T]\}$ be a sequence of real-valued càdlàg processes. If either
\begin{enumerate}[i)]
    
    \item for each $r, t\in[0, T]$, $|X_t^N-X_r^N|\le C|t-r|$ for some constant $C$ independent of $N$, or 
    
    \item $X^N$ is a martingale and for each $r, t\in[0, T]$, $|\scal{X}^N_t-\scal{X}^N_t|\le C|t-r|$ for some constant $C$ independent of $N$,

\end{enumerate}
then the sequence $\{X^N\}_N$ is tight with respect to the Skorokhod topology of the real-valued càdlàg functions $D([0, T], \mathbb{R})$.
\end{lemma}

To show the result above, we will use the following criterion.

\begin{proposition}[Aldous's Criterion, \cite{ald78}]\label{prop:aldous} A sequence of real-valued processes $\{X_t^N,$ $t\in[0, T]\}_N$ is tight with respect to the Skorokhod topology of $D([0, T], \mathbb{R})$ if:
\begin{enumerate}[i)]
    
    \item the sequence of real-valued random variables $\{X_t^N\}_N$ is tight for every $t\in[0, T]$, and
    
    \item for any $\eps>0$,
    \begin{equation*}
        \mylim_{\delta\to0}\mylimsup_{N\to\infty}\mysup_{\gamma\le\delta}\mysup_{\tau\in\mathcal{I}_T}\prob\left\{\left|X^N_{(\tau+\gamma)\wedge T}-X^N_\tau\right|>\eps\right\}=0,
    \end{equation*}
    where $\mathcal{I}_T$ is the set of stopping times almost surely bounded from above by $T$.

\end{enumerate}
\end{proposition}

\begin{proof}[Proof of Lemma~\ref{lemma:tightness}]
i) By a classical argument (see for example \cite[Section~4.1]{kl99}), in order to show tightness of $\{X^N\}_N$ it suffices to show that, for every $\eps>0$,
\begin{equation*}
    \mylim_{\gamma\to0}\mylimsup_{N\to\infty}\prob\left\{\mysup_{|t-r|\le\gamma}\left|X_t^N-X_r^N\right|>\eps\right\}=0,
\end{equation*}
which follows immediately from the fact that $|X_t^N-X_r^N|\le C|t-r|$.

ii) Given a stopping time $\tau$ almost surely bounded by $T$ and $\eps, \gamma>0$, we have that
\begin{align*}
    \prob\left\{\left|X^N_{\tau+\gamma}-X^N_\tau\right|\ge\eps\right\}&\le\frac{1}{\eps^2}\expected\left[\left(X^N_{\tau+\gamma}-X_\tau^N\right)^2\right]
    \\&\le\frac{1}{\eps^2}\expected\left[\left|\scal{X}^N_{\tau+\gamma}-\scal{X}^N_\tau\right|\right]
    \le\frac{C\gamma}{\eps^2}
\end{align*}
for some constant $C$ independent of $N$. Tightness then follows from Aldous's criterion.
\end{proof}

\begin{proof}[Proof of Proposition~\ref{prop:tightness_1}]
By \cite[Proposition~1.7]{kl99}, it suffices to show that, for each function $\phi\in C^2(\mathbb{T})$, the sequence of measures induced by the sequence of real-valued processes $\{\langle\pi_t^N, \phi\rangle\}_N$ is tight. Hence, it suffices to show that each term on the RHS of \eqref{eq:four_terms} is tight, where $\phi\in C^2(\mathbb{T})$ is chosen to be constant in time.

From \eqref{eq:theta_Phi}, we have that
\begin{equation*}
    \left\langle\pi_0^N, \phi\right\rangle=\frac{1}{2N}\sum_{x\in\mathbb{T}_N}\phi\left(\frac{x}{N}\right)-\frac{1}{2}\theta_0^N\left(\nabla_N\phi\right).
\end{equation*}
It is easy to see that $\frac{1}{2N}\sum_{x\in\mathbb{T}_N}\phi\left(\frac{x}{N}\right)$ converges to $\frac{1}{2}\langle1, \phi\rangle$. Also, by the hypothesis \eqref{eq:initial_profile} on the initial profile, it is easy to see that $\frac{1}{2}\theta_0^N\left(\nabla_N\phi\right)$ converges to $\langle h_0, \nabla\phi\rangle$ in probability. This implies that the sequence $\{\langle \pi_0^N, \phi\rangle\}_N$ is convergent, and hence tight. In particular, since $h_0\in \mathcal{H}^1$, an integration by parts yields
\begin{align*}
    \mylim_{N\to\infty}\mu_N\left\{\left|\left\langle\pi_0^N, \phi\right\rangle-\int_\mathbb{T}\frac{1}{2}\left(1+\nabla h_0(u)\right)\phi(u)\de u\right|>\delta\right\}=0
\end{align*}
for each $\delta>0$.

By the mean value theorem, we have the bounds    
\begin{equation*}
    \left|K_t^N(\phi)-K_r^N(\phi)\right|=\left|\frac{1}{2}\int_r^t\frac{1}{N}\sum_{x\in\mathbb{T}_N}\Delta_N\phi\left(\frac{x}{N}\right)\xi^N_s(x)\de s\right|
\le |t-r|\|\Delta \phi\|_\infty
\end{equation*}
and
\begin{align*}
    &\left|B_t^N(\phi)-B_r^N(\phi)\right|
    \\&=\left|\frac{1}{2}\int_r^t \text{tanh}\left(\frac{1}{N}\right)\sum_{x\in\mathbb{T}_N}\nabla_N\phi\left(\frac{x}{N}\right)\sgn(Y^N_s)\left[\xi_s^N(x)-\xi_s^N(x+1)\right]^2\de s\right|
    \\&\le|t-r|\|\nabla \phi\|_\infty,
\end{align*}
which shows tightness of $\{K^N_t(\phi), t\in[0, T]\}_N$ and $\{B_t^N(\phi), t\in[0, T]\}_N$ by Lemma~\ref{lemma:tightness}.

Finally, recalling \eqref{eq:rates_tanh}, an easy computation shows that \eqref{eq:quad_var_1} can be rewritten as 
\begin{equation*}
    \scal{M^N(\phi)}_t=\int_0^t\frac{1}{N^2}\sum_{x\in\mathbb{T}_N}\nabla_N\phi\left(\frac{x}{N}\right)^2\eta_s^N(x)\de s,
\end{equation*}
where
\begin{equation*}
    \eta_s^N(x)=\frac{1}{2}\left[\xi_s^N(x)-\xi_s^N(x+1)\right]^2+\frac{1}{2}\,\text{tanh}\left(\frac{1}{N}\right)\sgn(Y_s^N)\left[\xi_s^N(x)-\xi_s^N(x+1)\right].
\end{equation*}
Since $|\eta_s^N(x)|$ is uniformly bounded in $N$, this again yields tightness of $\{M_t^N(\phi), t\in[0, T]\}$ by Lemma~\ref{lemma:tightness}.

The fact that all limit trajectories have a density $\{\rho_t, t\in[0, T]\}$ with $0\le \rho_t\le 1$ follows from standard arguments.
\end{proof}

%%%%%%%%%%%%%%%%%%%%%%%%%%%%%%%%%%%%%%%%%%%%%%%%%%
%%%Replacement Lemma%%%%%%%%%%%%%%%%%%%%%%%%%%%%%%
%%%%%%%%%%%%%%%%%%%%%%%%%%%%%%%%%%%%%%%%%%%%%%%%%%
\subsection{Replacement Lemma}\label{sec:replacement}
In this section we will show the following lemma, which is a key result in the proof of Theorem~\ref{thm:hydro}. Given an integer $1\le \ell\le N-1$, define the right and left empirical densities of particles on boxes of length $\ell$ as
\begin{equation*}
    \vec\xi^\ell(x):=\frac{1}{\ell}\sum_{y=x+1}^{x+\ell}\xi(y),\qquad
    \cev\xi^\ell(x):=\frac{1}{\ell}\sum_{y=x-\ell}^{x-1}\xi(y),
\end{equation*}
where the sites in the sums are taken modulo $N$. Throughout, if $\ell$ is not an integer, $\vec\xi^\ell$ should be interpreted as $\vec\xi^{\lfloor \ell \rfloor}$, and similarly for $\cev\xi^\ell$. Moreover, recall the definition of $\prob_{\mu}$ as the probability measure that the process $\{h^N_t, t\in[0, T]\}$ induces on the space of càdlàg functions $D([0, T], \Omega_N)$ when starting from a measure $\mu$: throughout, we will denote by $\expected_{\mu}[\,\cdot\,]$ expectations with respect to $\prob_{\mu}$. The following holds.

\begin{lemma}[Replacement Lemma]\label{lemma:replacement} Let $G:[0, T]\times\mathbb{T}\to\mathbb{R}$ be a bounded function, and let $\{\psi_x\}_{x\in\mathbb{T}_N}$ be a uniformly bounded family of real-valued functions on $\Omega_N$ such that, for each $z\in\{x+1, \ldots, x+\eps N\}$, $\psi_x$ is invariant under the change of variable $h\mapsto h^{z, z+1}$.  If $\mu_N$ satisfies \eqref{eq:entropy_bound}, then for any $t\in[0, T]$ we have that
\begin{equation}\label{eq:replacement_integral}
    \mylim_{\eps\to0}\mylim_{N\to\infty}\expected_{\mu_N}\left|\int_0^t  \sgn(Y_s^N)\frac{1}{N}\sum_{x\in\mathbb{T}_N} G_s\left(\frac{x}{N}\right) \psi_ x(h_s^N)\left[\xi_s^N(x)-\vec\xi_s^{\eps N}(x)\right]\de s\right|=0.
\end{equation}
The same result holds with the right average replaced with the left average, as long as $\psi_x$ is invariant under the change of variable $h\mapsto h^{z, z-1}$ for $z\in\{x-\eps N, \ldots, x-1\}$ instead.
\end{lemma}

The result above is almost identical to the well-known replacement lemma introduced in \cite{gpv88}, except for the presence of the term $\sgn(Y_s^N)$ in the integral in \eqref{eq:replacement_integral}. The first part of our proof follows the argument given, for example, in the proof of \cite[Lemma~6.6]{gmo23}. As we shall see, the presence of the sign term will present an additional difficulty, and in order to treat the resulting correction we will need the following elementary bound:

\begin{lemma}\label{lemma:sequence}
Let $\{x_j\}_{j\ge 1}$ be a sequence of positive real numbers such that $\sum_{j=1}^\infty x_j\le M<\infty$. Then
\begin{equation*}
    M\sum_{j=1}^\infty \left(\sqrt{x_j}-\sqrt{x_{j+1}}\right)^2\ge\frac{(x_1)^2}{4}.
\end{equation*}
\end{lemma}

\begin{proof}
By Titu's Lemma, we have that
\begin{equation}
    \sum_{j=1}^\infty \left(\sqrt{x_j}-\sqrt{x_{j+1}}\right)^2=\sum_{j=1}^\infty \left(\frac{x_j-x_{j+1}}{\sqrt{x_j}+\sqrt{x_{j+1}}}\right)^2\ge \frac{\left(\sum_{j=1}^\infty\left(x_j-x_{j+1}\right)\right)^2}{\sum_{j=1}^\infty\left(\sqrt{x_j}+\sqrt{x_{j+1}}\right)^2}.\label{eq:titu}
\end{equation}
But now, by the Cauchy-Schwarz inequality and the fact that $\sum_{j=1}^\infty x_j\le M$, we have the upper bound $\sum_{j=1}^\infty\left(\sqrt{x_j}+\sqrt{x_{j+1}}\right)^2\le 4M$. Also, $\sum_{j=1}^\infty\left(x_j-x_{j+1}\right)=x_1$, and thus
\begin{equation*}
    \eqref{eq:titu}\ge\frac{(x_1)^2}{4M},
\end{equation*}
which concludes the proof.
\end{proof}

We will also use the following fact: given $h\in\Omega_N$ and $x\in\mathbb{T}_N$, let
\begin{equation}\label{eq:q^N}
    q^N_{x, x+1}(h):=\begin{cases}  p_N(x, x+1, h)  &\mbox{if } \xi_x=1, \xi_{x+1}=0,
    \\p_N(x+1, x, h) &\mbox{if } \xi_x=0, \xi_{x+1}=1,
    \\\frac{1}{2} &\mbox{otherwise,}
    \end{cases}
\end{equation}
and, given a function $f:\Omega_N\to\mathbb{R}$ and a measure $\mu$ on $\Omega_N$, define the \textit{Dirichlet form} of $f$ as 
\begin{equation}\label{eq:def_dirichlet}
    \mathscr{D}_N\left(f, \mu\right):=-\left\langle \mathcal{L}^Nf, f\right\rangle_{\mu}
\end{equation}
and the \textit{carré du champ} operator as
\begin{align*}
    \Gamma_N(f, \mu)&:=\frac{1}{2}\int_{\Omega_N}\left\{\mathcal{L}^Nf^2-2f\mathcal{L}^N f\right\}\de \mu
    \\&=\frac{1}{2}\sum_{x\in\mathbb{T}_N}\int_{\Omega_N} q^N_{x, x+1}(h)\left[f(h^{x, x+1})-f(h)\right]^2\de \mu.
\end{align*}
Then, for any $f$ we have that  
\begin{equation}\label{eq:dir=gamma}
    \mathscr{D}_N\left(f, \mu_N^*\right)=\Gamma_N(f, \mu_N^*),
\end{equation}
as a consequence of \eqref{eq:liggett_invariance}. 

\begin{proof}[Proof of Lemma~\ref{lemma:replacement}] To ease the notation, we drop the superscript $N$ and we 
write $Y=Y(h)$ and $Y^{z, z+1}=Y(h^{z, z+1})$. By the entropy inequality (see for example \cite[Proposition~A1.9.1]{kl99}) and Jensen's inequality, for any $M>0$ we have  
\begin{align}
    \expected_{\mu_N}&\left|\int_0^t  \sgn(Y_s)\frac{1}{N}\sum_{x\in\mathbb{T}_N} G_s\left(\frac{x}{N}\right) \psi_ x(h_s)\left[\xi_s(x)-\vec\xi_s^{\eps N}(x)\right]\de s\right|\nonumber
    \\&=\frac{1}{MN}\expected_{\mu_N}\left[\mylog e^{MN\left|\int_0^t\sgn(Y_s)\frac{1}{N}\sum_{x\in\mathbb{T}_N}G_s\left(\frac{x}{N}\right) \psi_ x(h_s)\left[\xi_s(x)-\vec\xi_s^{\eps N}(x)\right]\de s\right|}\right]\nonumber
    \\\begin{split}&\le\frac{H\left(\mu_N\|\mu_N^*\right)}{MN}
    \\&\phantom{\le}+\frac{1}{MN}\mylog\expected_{\mu_N^*}\left[e^{MN\left|\int_0^t  \sgn(Y_s)\frac{1}{N}\sum_{x\in\mathbb{T}_N} G_s\left(\frac{x}{N}\right) \psi_ x(h_s)\left[\xi_s(x)-\vec\xi_s^{\eps N}(x)\right]\de s\right|}\right].\label{eq:feynman_term}
    \end{split}
\end{align}
Note that $e^{|x|}\le e^x+e^{-x}$, and 
\begin{equation*}
    \mylimsup_{N\to\infty}\frac{1}{N}\mylog\{a_N+b_N\}\le \mymax\left\{\mylimsup_{N\to\infty}\frac{1}{N}\mylog a_N, \mylimsup_{N\to\infty}\frac{1}{N}\mylog b_N\right\},
\end{equation*}
so we can remove the absolute value from the exponential. Also, $\mu_N^*$ is invariant so by the Feynman-Kac formula (see for example \cite[Lemma~A1.7.2]{kl99}), the variational formula for the largest eigenvalue of an operator in a Hilbert space and the relative entropy bound \eqref{eq:entropy_bound}, we can bound \eqref{eq:feynman_term} from above by
\begin{equation*}
    \begin{split}
    &\frac{C_0}{M}+\mysup_f\Bigg\{\int_0^t\frac{1}{N}\sum_{x\in\mathbb{T_N}} G_s\left(\frac{x}{N}\right)\left\langle \sgn(Y)\psi_x\left[\xi(x)-\vec\xi^{\eps N}(x)\right], f\right\rangle_{\mu_N^*}\de s
    \\&-\frac{tN}{M}\mathscr{D}_N\left(\sqrt{f}, \mu_N^*\right)\Bigg\},
    \end{split}
\end{equation*}
where the supremum runs over all probability densities $f$ with respect to $\mu_N^*$. By \eqref{eq:dir=gamma}, the expression above is equal to
\begin{equation}\label{eq:sup_2}
\begin{split}
    &\frac{C_0}{M}+\mysup_f\Bigg\{\int_0^t\frac{1}{N}\sum_{x\in\mathbb{T_N}} G_s\left(\frac{x}{N}\right)\left\langle \sgn(Y)\psi_x\left[\xi(x)-\vec\xi^{\eps N}(x)\right], f\right\rangle_{\mu_N^*}\de s
    \\&-\frac{tN}{M}\Gamma_N\left(\sqrt{f}, \mu_N^*\right)\Bigg\}.
    \end{split}
\end{equation}
We rewrite the integrand of the first summand in the supremum as
\begin{align}
    \frac{1}{N}&\sum_{x\in\mathbb{T}_N} G_s\left(\frac{x}{N}\right)\left\langle \sgn(Y)\psi_x\left[\xi(x)-\vec\xi^{\eps N}(x)\right], f\right\rangle_{\mu_N^*}\nonumber
    \\&=\frac{1}{\eps N^2}\sum_{x\in\mathbb{T}_N}G_s\left(\frac{x}{N}\right)\int_{\Omega_N}\sgn(Y)\psi_x(h)\sum_{y=x+1}^{x+\eps N}\left[\xi(x)-\xi(y)\right]f(h)\de \mu_N^*\nonumber
    \\&=\frac{1}{\eps N^2}\sum_{x\in\mathbb{T}_N}G_s\left(\frac{x}{N}\right)\int_{\Omega_N}\sgn(Y)\psi_x(h)\sum_{y=x+1}^{x+\eps N}\sum_{z=x}^{y-1}\left[\xi(z)-\xi(z+1)\right]f(h)\de \mu_N^*\nonumber
    \\\begin{split}&=\frac{1}{2\eps N^2}\sum_{x\in\mathbb{T}_N}G_s\left(\frac{x}{N}\right)\int_{\Omega_N}\sgn(Y)\psi_x(h)\sum_{y=x+1}^{x+\eps N}\sum_{z=x}^{y-1}\left[\xi(z)-\xi(z+1)\right]\times
    \\&\phantom{=}\times\left[f(h)-f(h^{z, z+1})\right]\de \mu_N^*\label{eq:first_term}
    \end{split}
    \\\begin{split}&\phantom{=}+\frac{1}{2\eps N^2}\sum_{x\in\mathbb{T}_N}G_s\left(\frac{x}{N}\right)\int_{\Omega_N}\sgn(Y)\psi_x(h)\sum_{y=x+1}^{x+\eps N}\sum_{z=x}^{y-1}\left[\xi(z)-\xi(z+1)\right]\times\\&\phantom{=}\times\left[f(h)+f(h^{z, z+1})\right]\de \mu_N^*.\label{eq:second_term}
    \end{split}
\end{align}
Since $\sgn^2=1$, by multiplying and dividing by $\sqrt{q^N_{z, z+1}(h)}$, noting that $a-b=(\sqrt{a}-\sqrt{b})(\sqrt{a}+\sqrt{b})$ and applying Young's inequality $ab\le\frac{Aa^2}{2}+\frac{b^2}{2A}$ for any $A>0$, we find that the integral of \eqref{eq:first_term} on $[0, t]$ is bounded from above by
\begin{equs}
    \,&\frac{tA}{4\eps N^2}\sum_{x\in\mathbb{T}_N}\int_{\Omega_N}\sum_{y=x+1}^{x+\eps N}\sum_{z=x}^{y-1}q^N_{z, z+1}(h)\left[\xi(z)-\xi(z+1)\right]^2\left[\sqrt{f(h)}-\sqrt{f(h^{z, z+1})}\right]^2\de \mu_N^*
    \\&+\frac{1}{4A\eps N^2}\int_0^t\de s\sum_{x\in\mathbb{T}_N}G_s\left(\frac{x}{N}\right)^2\int_{\Omega_N}\psi_x(h)^2\sum_{y=x+1}^{x+\eps N}\sum_{z=x}^{y-1}\frac{1}{q^N_{z, z+1}(h)}\left[\xi(z)-\xi(z+1)\right]^2\times
    \\&\times\left[\sqrt{f(h)}+\sqrt{f(h^{z, z+1})}\right]^2\de \mu_N^*.\label{eq:first_young}
\end{equs}
Also, note that
\begin{equation}\label{eq:dirichlet_term}
    \begin{split}
    &\int_{\Omega_N}\sum_{x\in\mathbb{T}_N}\sum_{y=x+1}^{x+\eps N}\sum_{z=x}^{y-1}q^N_{z, z+1}(h)\left[\xi(z)-\xi(z+1)\right]^2\left[\sqrt{f(h)}-\sqrt{f(h^{z, z+1})}\right]^2\de \mu_N^*
    \\&=\frac{\eps N(\eps N+1)}{2} \Gamma_N\left(\sqrt{f}, \mu_N^*\right).
\end{split}
\end{equation}
Hence, by choosing $A=\frac{4N^2}{M(\eps N+1)}$ in \eqref{eq:first_young} and by \eqref{eq:dirichlet_term}, the first term of \eqref{eq:first_young} is equal to $\frac{tN}{2M}\Gamma_N\left(\sqrt{f}, \mu_N^*\right)$, while the second term is of order $\frac{1}{\eps N^2}\cdot\frac{\eps}{N}\cdot(\eps N)^2\cdot N$, so it goes to zero as $N\to\infty$ and $\eps\to0$. 

We now turn to \eqref{eq:second_term}: by performing the change of variable $h\mapsto h^{z, z+1}$, we can rewrite it as
\begin{equation*}
    \begin{split}
    &\frac{1}{2\eps N^2}\sum_{x\in\mathbb{T}_N}G_s\left(\frac{x}{N}\right)\int_{\Omega_N}\sgn(Y)\psi_x(h)\sum_{y=x+1}^{x+\eps N}\sum_{z=x}^{y-1}\left[\xi(z)-\xi(z+1)\right]f(h)\times
    \\&\times\left[1-\frac{\sgn(Y^{z, z+1})\mu_N^*(h^{z, z+1})}{\sgn(Y)\mu_N^*(h)}\right]\de \mu_N^*,
    \end{split}
\end{equation*}
where we used the fact that $\psi_x(h^{z, z+1})=\psi_x(h)$ for each $h\in\Omega_N$, $x\in\mathbb{T}_N$ and $z\in\{x, \ldots, x+\eps N\}$. For $k$ an odd positive integer, let
\begin{equation}\label{eq:Omega_k}
    \Omega_N^k=\{h\in\Omega_N: |Y(h)|=k\}
\end{equation}
so that $\{\Omega_N^j\}_j$ is a partition of $\Omega_N$, and let $\Omega_N^{>1}=\bigcup_{k>1} \Omega_N^j$. When the integral above is restricted to $\Omega_N^{>1}$ we have $\sgn(Y^{z, z+1})=\sgn(Y)$, so this term equals
\begin{equation}
\begin{split}\label{eq:swap_1} 
    &\frac{1}{2\eps N^2}\sum_{x\in\mathbb{T}_N}G_s\left(\frac{x}{N}\right)\int_{\Omega_N^{>1}}\sgn(Y)\psi_x(h)\sum_{y=x+1}^{x+\eps N}\sum_{z=x}^{y-1}\left[\xi(z)-\xi(z+1)\right]f(h)\times
    \\&\times\left[1-\frac{\mu_N^*(h^{z, z+1})}{\mu_N^*(h)}\right]\de \mu_N^*.
    \end{split} 
\end{equation}
Similarly, the restriction of the integral to $\Omega_N^1$ and the restriction of the sum in $z$ to the points that send the configuration outside of $\Omega_1$ yields
\begin{equation}\label{eq:swap_2}
    \begin{split}
    &\frac{1}{2\eps N^2}\sum_{x\in\mathbb{T}_N}G_s\left(\frac{x}{N}\right)\int_{\Omega_N^1}\sgn(Y)\psi_x(h)\sum_{y=x+1}^{x+\eps N}\sum_{\substack{z=x\\ h^{z, z+1}\in\Omega_N^{>1}}}^{y-1}\left[\xi(z)-\xi(z+1)\right]f(h)\times
    \\&\times\left[1-\frac{\mu_N^*(h^{z, z+1})}{\mu_N^*(h)}\right]\de \mu_N^*.
    \end{split}
\end{equation}
But now, $\mu_N^*$ is invariant for nearest neighbour exchanges up to a correction of order $\frac{1}{N}$, as for any $h\in\Omega_N$ and $z\in\mathbb{T}_N$ we have that
\begin{equation}\label{eq:almost_invariant}
    \left|1-\frac{\mu_N^*(h^{z, z+1})}{\mu_N^*(h)}\right|\le\frac{C}{N}
\end{equation}
for some constant $C$ independent of $N$. Hence, the integrals on $[0, t]$ of both \eqref{eq:swap_1} and \eqref{eq:swap_2} are of order $\frac{1}{\eps N^2}\cdot N(\eps N)^2\cdot \frac{1}{N}$, which goes to zero as $N\to\infty$ and $\eps\to0$. 

Thus, the only remaining contribution given by \eqref{eq:second_term} is the integral on $[0, t]$ of
\begin{align*}
    \begin{split}&\frac{1}{2\eps N^2}\sum_{x\in\mathbb{T}_N}G_s\left(\frac{x}{N}\right)\int_{\Omega_N^1}\sgn(Y)\psi_x(h)\sum_{y=x+1}^{x+\eps N}\sum_{\substack{z=x\\ h^{z, z+1}\in\Omega_N^{1}}}^{y-1}\left[\xi(z)-\xi(z+1)\right] 
    \\&\quad\times\left[f(h)+f(h^{z, z+1})\right]\de \mu_N^*
    \end{split}
    \\&=\frac{1}{\eps N^2}\sum_{x\in\mathbb{T}_N} G_s\left(\frac{x}{N}\right)\int_{\Omega_N^1}\psi_x(h)\sum_{y=x+1}^{x+\eps N}\sum_{\substack{z=x\\ h^{z, z+1}\in\Omega_N^1}}^{y-1}\left[\xi(z)-\xi(z+1)\right]^2f(h)\de \mu_N^*.
\end{align*}
The equality above follows from the change of variable $h\mapsto h^{z, z+1}$ and noting that, whenever $h$ is in $\Omega_N^1$ and $z$ is such that $h^{z, z+1}$ is also in $\Omega_N^1$, then $\sgn(Y(h))\left[\xi(z)-\xi(z+1)\right]$ is always non-negative, and in particular is equal to $\left[\xi(z)-\xi(z+1)\right]^2$. Hence, we are only left to estimate
\begin{equation}\label{eq:remaining_zero}
    \begin{split}
    &\frac{1}{\eps N^2}\int_0^t\de s\sum_{x\in\mathbb{T}_N} G_s\left(\frac{x}{N}\right)\int_{\Omega_N^1}\psi_x(h)\sum_{y=x+1}^{x+\eps N}\sum_{\substack{z=x\\ h^{z, z+1}\in\Omega_N^1}}^{y-1}\left[\xi(z)-\xi(z+1)\right]^2f(h)\de \mu_N^*
    \\&-\frac{tN}{2M}\Gamma_N\left(\sqrt{f}, \mu_N^*\right)
    \end{split}
\end{equation}
inside the supremum in \eqref{eq:sup_2}.

Given $h\in\Omega_N$, let $D(h):=\{z\in \mathbb{T}_N: \xi(z)\ne \xi(z+1)\ \text{and}\ |Y(h^{z, z+1})|\le |Y(h)| \}$ be the set of points that make the integral of a height configuration decrease in absolute value, and let $U(h):=\{z\in \mathbb{T}_N: \xi(z)\ne \xi(z+1)\ \text{and}\ |Y(h^{z, z+1})|\ge |Y(h)|\}$ be the set of points that make the integral of a height configuration increase in absolute value. 

\begin{remark}\label{remark:d=u} Note that, if $Y(h)>0$, then $D(h)$ is the set of local maxima of $h$ and $U(h)$ is the set of local minima, whereas the opposite holds for $Y(h)<0$. In particular, $|D(h)|=|U(h)|$.
\end{remark}
Then, the first term of \eqref{eq:remaining_zero} satisfies
\begin{equation}\label{eq:nicer}
    \begin{split}
    &\frac{1}{\eps N^2}\int_0^t\de s\sum_{x\in\mathbb{T}_N} G_s\left(\frac{x}{N}\right)\int_{\Omega_N^1}\psi_x(h)\sum_{y=x+1}^{x+\eps N}\sum_{\substack{z=x\\ z\in D(h)}}^{y-1}f(h)\de \mu_N^*
    \\&\le t\eps\|G\|_\infty\|\psi\|_\infty \int_{\Omega_N^1} f(h)|D(h)|\ \de \mu_N^*,
    \end{split}
\end{equation}
where $\|G\|_\infty=\mysup_{s\in[0, T]}\|G_s\|_{\infty}$ and $\|\psi\|_\infty=\mysup_x\|\psi_x\|_{\infty}$. 

Now let $\Theta_N^k(f):=\int_{\Omega_N^k} f(h)|U(h)|\de \mu_N^*$ for each $k$ odd positive integer. Since the measure $\mu_N^*(h)$ only depends on the integral of $h$, there is a map $k \mapsto \nu_N^*(k)$ such that $\mu_N^*(h) = \nu_N^*(k)$ for any $k$ and any $h\in\Omega_N^k$. Note that $q^N_{x, x+1}(h)\ge\frac{1}{3}$ for any $x\in\mathbb{T}_N$ and $h\in\Omega_N$, so by the inequality $(\sqrt{\sum_i a_i}-\sqrt{\sum_i b_i})^2\le \sum_i(\sqrt{a_i}-\sqrt{b_i})^2$ we have that
\begin{align*}
    \Gamma_N\left(\sqrt{f}, \mu_N^*\right)&\ge\frac{1}{3}\int_{\Omega_N}\sum_{x\in\mathbb{T}_N}\left[\sqrt{f(h)}-\sqrt{f(h^{x, x+1})}\right]^2\de \mu_N^*
    \\&=\frac{1}{3}\sum_{\substack{k=1\\ k\ \text{odd}}}^\infty \nu_N^*(k)\sum_{h\in\Omega_N^k}\sum_{x\in\mathbb{T}_N}\left[\sqrt{f(h)}-\sqrt{f(h^{x, x+1})}\right]^2
    \\&\ge\frac{1}{3} \sum_{\substack{k=1\\ k\  \text{odd}}}^{N-1}\nu_N^*(k)\sum_{h\in\Omega_N^k}\sum_{x\in U(h)}\left[\sqrt{f(h)}-\sqrt{f(h^{x, x+1})}\right]^2
    \\&\ge\frac{1}{3}\sum_{\substack{k=1\\ k\  \text{odd}}}^\infty\left(\sqrt{\sum_{h\in\Omega_N^k}f(h)|U(h)|\nu_N^*(k)}-\sqrt{\sum_{h\in\Omega_N^k}\sum_{x\in U(h)}f(h^{x, x+1})\nu_N^*(k)}\right)^2
    \\&=\frac{1}{3}\sum_{\substack{k=1\\ k\  \text{odd}}}^\infty\left(\sqrt{\int_{\Omega_N^k}f(h)|U(h)|\de \mu_N^*}-\sqrt{\int_{\Omega_N^{k+2}}f(h)|D(h)|\frac{\nu_N^*(k)}{\nu_N^*(k+2)}\de \mu_N^*}\right)^2
    \\&=\frac{1}{3}\sum_{\substack{k=1\\ k\  \text{odd}}}^\infty\left(\sqrt{\Theta_N^k(f)}-\sqrt{\Theta_N^{k+2}(f)}\cdot e^{\frac{1}{2N}}\right)^2.
\end{align*}
The equality in the penultimate line comes from Remark~\ref{remark:d=u} and from the following argument: if $h\in\Omega_N^k$ and $z\in U(h)$, then $t=h^{z, z+1}$ is in $\Omega_N^{k+2}$, and in particular any configuration $t\in\Omega_N^{k+2}$ has exactly $|D(t)|$ preimages in $\Omega_N^k$ via the maps $h\mapsto h^{z, z+1}$. Thus, by \eqref{eq:nicer} and Remark~\ref{remark:d=u} we get that
\begin{equation}\label{eq:remaining}
    \eqref{eq:remaining_zero}\le t\eps\|G\|_\infty \|\psi\|_\infty \Theta_N^1(f)-\frac{tN}{6M}\sum_{\substack{k=1\\k\ \text{odd}}}^\infty\left(\sqrt{\Theta_N^k(f)}-\sqrt{\Theta_N^{k+2}(f)}\cdot e^{\frac{2}{N}}\right)^2.
\end{equation}

It remains to show that the RHS of \eqref{eq:remaining} is bounded from above by zero when $N\to\infty$ and $\eps\to0$. Since $|U(h)|\le N$ for any $h\in\Omega_N$ and $f$ is a density, we have 
\begin{equation*}
    \sum_{\substack{k=1\\k\ \text{odd}}}^\infty \Theta_N^k(f)\le N\sum_{\substack{k=1\\k\ \text{odd}}}^\infty\int_{\Omega_N^k} f(h)\de \mu_N^*=N\;,
\end{equation*}
and we can apply Lemma~\ref{lemma:sequence} to the sequence $\{\Theta_N^k(f)\}_{k\ge 1}$. Together with the Taylor expansion of $e^{\frac{2}{N}}$, this yields $\mylim_{\eps\to0}\mylim_{N\to\infty}\eqref{eq:remaining}\le 0$. Sending  $M\to\infty$ in \eqref{eq:sup_2} finally concludes the proof.
\end{proof}

\begin{remark}\label{remark:measures} It is worth noting that the Feynman--Kac formula is valid even when the measure is not invariant for the process of interest (see for example \cite[Lemma~A.1]{bmns17}). But then, consider a measure $\tilde{\mu}_N$ on $\Omega_N$ which is almost invariant for nearest neighbour exchanges in the sense that \eqref{eq:almost_invariant} holds: by slightly adapting the proof of \cite[Corollary~5.3]{bgj19}, one can show that for any density $f$ we have the bound
\begin{equation*}
    -\mathscr{D}_N\left(\sqrt{f}, \tilde{\mu}_N\right)\le-\frac{1}{4}\Gamma_N\left(\sqrt{f}, \tilde{\mu}_N\right)+\frac{C}{N}
\end{equation*}
where $C$ is a constant independent of $f$ and $N$. Now take a measure $\tilde{\mu}_N$ as above and such that the initial measure $\mu_N$ satisfies the entropy bound \eqref{eq:entropy_bound} with respect to it. Then, it is not hard to see that if we change the measure to $\tilde{\mu}_N$ instead of $\mu_N^*$ in \eqref{eq:feynman_term}, the proof of the replacement lemma goes through in exactly the same way, except for an additional summand $\frac{C}{M}$ in the supremum \eqref{eq:sup_2}, which vanishes as $M\to\infty$. But the proof of the replacement lemma is the only instance where we need the relative entropy bound of the initial measure: this means that we can relax the hypothesis \eqref{eq:entropy_bound} in Theorem~\ref{thm:hydro}, as long as $\mu_N$ satisfies the same entropy bound with respect to \textit{some} measure on $\Omega_N$ which is almost invariant for nearest neighbour exchanges in the sense of \eqref{eq:almost_invariant}.
\end{remark}

%%%%%%%%%%%%%%%%%%%%%%%%%%%%%%%%%%%%%%%%%%%%%%%%%%
%%%The Integral Process%%%%%%%%%%%%%%%%%%%%%%%%%%%
%%%%%%%%%%%%%%%%%%%%%%%%%%%%%%%%%%%%%%%%%%%%%%%%%%
\subsection{The Integral Process}
Consider the integral process $\{Y_t^N, t\in[0, T]\}$ defined in \eqref{eq:integral_process}, which takes values in $\big\{\frac{k}{N^2}, k\in\mathbb{Z}\ \text{odd}\big\}$. By Dynkin's formula, the process $\{X_t^N, t\in[0, T]\}$ defined as
\begin{equation}\label{eq:mg_integral}
    X^N_t:=Y^N_t-Y^N_0+\text{tanh}\left(\frac{1}{N}\right)\int_0^t \sgn(Y^N_s)\sum_{x\in\mathbb{T}_N}\left[\xi_s^N(x)-\xi_s^N(x+1)\right]^2\de s 
\end{equation}
is a martingale with quadratic variation
\begin{equation}\label{eq:qv_integral}
    \scal{X^N}_t=\frac{1}{N^2}\int_0^t \sum_{x\in\mathbb{T}_N}\left[\xi_s^N(x)-\xi_s^N(x+1)\right]^2
    \de s. 
\end{equation} 
Also, let $\{Z^N_t, t\in[0, T]\}$ be the process defined by $Z^N_t:=\sgn(Y_t^N)$, which takes values in $\{-1, 1\}$. We call  $\mathbb{Y}^N$ and $\mathbb{S}^N$ the probability measures on the space of real-valued càdlàg functions $D([0, T], \mathbb{R})$ and the space $L^1([0, T], [-1, 1])$, respectively, induced by $\{Y^N_t, t\in[0, T]\}$ and $\{Z^N_t, t\in[0, T]\}$ when the height process starts from the initial measure $\mu_N$. Then, we have the following result.

\begin{proposition}\label{prop:integral} Let $\{\mu_N\}_N$ satisfy \eqref{eq:initial_profile}. The following holds.
\begin{enumerate}[i)]
    \item  The sequence of measures $\{\mathbb{Y}^N\}_N$ is tight in $D([0, T], \mathbb{R})$ with respect to the Skorokhod topology, and any limit point is concentrated on paths in $C([0, T], \mathbb{R})$.    
    
    \item The sequence of measures $\{\mathbb{S}^N\}_N$ is tight in the space $L^1([0, T], [-1, 1])$ with respect to the weak topology.
    
    \item Let $\nu^N$ be the joint probability measure induced by the process $\{(\pi_t^N, Y_t^N, Z_t^N), t\in[0, T]\}$ and by the initial measure $\mu_N$. Then, any limit point $\nu$ of $\{\nu^N\}_N$ 
    is such that, for $\nu$-almost every triple $(\pi, Y, Z)$,
    $\pi_t$ has a density $\rho_t$ with respect to the Lebesgue measure and the identity
    \begin{equation*}
        Y_0\varphi_0+\int_0^TY_t\partial_t\varphi_t\de t
        -\int_0^T2Z_t\langle\rho_t, 1-\rho_t\rangle \varphi_t\de t=0 \;,
    \end{equation*}
    is satisfied for all $\varphi\in C^1_c([0, T), \mathbb{R})$.
\end{enumerate} 
\end{proposition}
We will be using the following result, whose proof is postponed to Appendix~\ref{sec:app_energy_uniqueness}.

\begin{lemma}\label{lemma:rho_H1}
Let $Q$ be a limit point of the sequence of measures $\{Q^N\}_N$. Then, $Q$ gives full measure to paths $\pi_t(u)=\rho_t(u)\de u$ such that $\rho$ belongs to the Sobolev space $L^2([0, T], \mathcal{H}^1)$. 
\end{lemma}

\begin{proof}[Proof of Proposition~\ref{prop:integral}]
i) Note that $\{Y_0^N\}_N$ converges in probability by the hypothesis on the initial profile, and in particular
\begin{equation*}
    \mylim_{N\to\infty} \mu_N\left\{\left|Y_0^N-\int_\mathbb{T} h_0(u)\de u\right|>\delta\right\}=0
\end{equation*}
for each $\delta>0$. Tightness then follows from Lemma~\ref{lemma:tightness}.

Now let $\mathbb{Y}$ be a limit point of $\{\mathbb{Y}^N\}_N$: in order to prove that $\mathbb{Y}\{Y\in C([0, T], \mathbb{R})\}=1$, it suffices to show that, for every $\eps>0$,
\begin{equation}\label{eq:continuity_Y}
    \mylim_{\gamma\to0}\mylimsup_{N\to\infty} \prob_{\mu_N}\left\{\mysup_{|t-r|\le\gamma}\left|Y_t^N-Y_r^N\right|>\eps\right\}=0.
\end{equation}
From \eqref{eq:mg_integral}, we have that
\begin{align*}
    &\prob_{\mu_N}\left\{\mysup_{|t-r|\le\gamma}\left|Y_t^N-Y_r^N\right|>\eps\right\}
    \\&\le \prob_{\mu_N}\left\{\mysup_{|t-r|\le\gamma}\left|\text{tanh}\left(\frac{1}{N}\right)\int_r^t\sgn(Y^N_s)\sum_{x\in\mathbb{T}_N}\left[\xi_s^N(x)-\xi_s^N(x+1)\right]^2\de s\right|>\frac{\eps}{2}\right\}
    \\&\phantom{\le}+ \prob_{\mu_N}\left\{\mysup_{|t-r|\le\gamma}\left|X_t^N-X_r^N\right|>\frac{\eps}{2}\right\}.
\end{align*}
But now, $\mysup_{|t-r|\le \gamma}\left|\text{tanh}\left(\frac{1}{N}\right)\int_r^t\sgn(Y^N_s)\sum_{x\in\mathbb{T}_N}\left[\xi_s^N(x)-\xi_s^N(x+1)\right]^2\de s\right|\le\gamma$, and by Doob's inequality and \eqref{eq:qv_integral},
\begin{align*}
    \prob_{\mu_N}\left\{\mysup_{|t-r|\le\gamma}\left|X_t^N-X_r^N\right|>\frac{\eps}{2}\right\}&\le \frac{C}{\eps^2}\expected_{\mu_N}\left[X_T^2\right]
    \\&\le\frac{CT}{N\eps^2},
\end{align*}
for some constant $C$ independent of $N$, which yield \eqref{eq:continuity_Y}.

ii) This follows immediately from the fact that the set of measurable functions
from $[0,T]$ to $[-1,1]$ endowed with the weak convergence topology is compact.

iii) For $\eps>0$, define the function $\cev i_\eps$ on the torus as a continuous approximation of the map $\frac{1}{\eps}\boldsymbol{1}_{\left[-\eps, 0\right)}$ in such a way that $\mylim_{\eps\to0}\big\|\frac{1}{\eps}\boldsymbol{1}_{\left[-\eps, 0\right)}-\cev i_\eps \big\|_{L^1}=0$; an example is given by the map $u\mapsto \frac{1}{\eps}\psi\big(\frac{u}{\eps}\big)$ with 
\begin{equation*}
\psi(u):=\frac{e^{-\frac{1}{u(u+1)}}\boldsymbol{1}_{(-1, 0)}(u)}{\int_{\mathbb{T}}e^{-\frac{1}{v(v+1)}}\de v}.  
\end{equation*}Given $\pi\in\mathbb{M}^+$, let 
\begin{equation*}
    \pi\ast \cev i_\eps(u):=\left\langle\pi, \cev i_\eps(\cdot-u) \right\rangle.
\end{equation*}

The existence of a density $\rho$  with respect to the Lebesgue measure for the realisations of $Q$ is guaranteed by Proposition~\ref{prop:tightness_1}, and by Lemma~\ref{lemma:rho_H1} and Lebesgue's differentiation theorem we have that
\begin{equation*}
    \mylim_{\eps\to0} \left|\rho_t(u)-\frac{1}{\eps}\int_{u-\eps}^u\rho_t(v)\de v\right|=0
\end{equation*}
for almost every $t\in[0, T]$ and $u\in\mathbb{T}$, so it suffices to show that, for each $\varphi\in C_0^1([0, T], \mathbb{R})$ and $\delta>0$, the measure $\nu$ satisfies
\begin{equation*}
    \begin{split}
    \nu\bigg\{&\bigg|Y_0\varphi_0+\int_0^TY_t\partial_t\varphi_t\de t-\int_0^T2Z_t\left\langle\pi_t, 1-\pi_t\ast \cev i_\eps\right\rangle\varphi_t\de t \bigg|>\delta\bigg\}\to 0
    \end{split}
\end{equation*}
as $\eps\to0$. Now note that, for any $\eps>0$ and $\varphi\in C_0^1([0, T], \mathbb{R})$, the function 
\begin{equation*}
\begin{split}
    (\pi, Y, Z)\mapsto &\left|Y_0\varphi_0+\int_0^TY_t\partial_t\varphi_t\de t-\int_0^T2Z_t\left\langle\pi_t, 1-\pi_t\ast \cev i_\eps\right\rangle\varphi_t\de t \right|
\end{split}
\end{equation*}
is continuous as a map on $D([0, T], \mathbb{M}^+)\times D([0, T], \mathbb{R})\times L^1([0, T], [-1, 1])$ where the first two spaces are equipped with the Skorokhod topology and the third one with the weak topology. This can be seen by noting that, if $\Phi: \mathbb{M}^+\to\mathbb{R}$ is continuous in the weak topology, then the map $(\pi, Z)\mapsto\int_0^TZ_t\Phi(\pi_t)\de t$ is jointly continuous. Thus, by Portmanteau's Theorem we have that, for each $\eps>0$ and $\varphi\in C_0^1([0, T], \mathbb{R})$,
\begin{align}
    &\nu\left\{\left|Y_0\varphi_0+\int_0^TY_t\partial_t\varphi_t\de t-\int_0^T2Z_t\left\langle\pi_t, 1-\pi_t\ast \cev i_\eps \right\rangle\varphi_t\de t \right|>\delta\right\}
    \nonumber
    \\& \le \myliminf_{N\to\infty}\nu^N\left\{\left|Y_0^N\varphi_0+\int_0^TY_t^N\partial_t\varphi_t\de t-\int_0^T2Z_t^N\left\langle\pi_t^N, 1-\pi_t^N\ast \cev i_\eps\right\rangle\varphi_t\de t \right|>\delta\right\}.
    \label{eq:portmanteau}
\end{align}

We now sum and subtract the integral on $[0, T]$ of the martingale \eqref{eq:mg_integral} multiplied by $\partial_t\varphi_t$ inside the absolute value in \eqref{eq:portmanteau}: by Doob's inequality,
\begin{align*}
    &\prob_{\mu_N}\left\{\left|\int_0^TX_t^N\partial_t\varphi_t\de t\right|>\frac{\delta}{2}\right\}
    \\&\le \prob_{\mu_N}\left\{T\|\partial_t\varphi\|_\infty \mysup_{0\le t\le T}\left|X_t^N\right|>\frac{\delta}{2}\right\}
    \\&\le\frac{4}{\left(\frac{\delta}{2}\right)^2T^2\|\partial_t\varphi\|_\infty^2}\expected_{\mu_N}\left[\left(X^N_T\right)^2\right]
    \\&\le\frac{16}{\delta^2T^2\|\partial_t\varphi\|_\infty^2}\expected_{\mu_N}\left[\int_0^T\frac{1}{N^2}\sum_{x\in\mathbb{T}_N}\left[\xi_s^N(x)-\xi_s^N(x+1)\right]^2\de s\right]\nonumber
    \\&\le\frac{16}{NT\delta^2\|\partial_t\varphi\|_\infty^2},
\end{align*}
which vanishes in the limit. Hence, by performing a summation by parts and a Taylor expansion of $\text{tanh}\left(\frac{1}{N}\right)$, and noting that
\begin{equation*}
    Z_t^N\frac{1}{N}\sum_{x\in\mathbb{T}_N}\left[\xi_t^N(x)^2+\xi_t^N(x+1)^2\right]=2Z_t^N\left\langle\pi_t^N, 1\right\rangle,
\end{equation*}
we easily get that we only need to show that 
\begin{equation}\label{eq:prob_sgn}
\begin{split}
    \mylim_{\eps\to0}\myliminf_{N\to\infty}\prob_{\mu_N}\Bigg\{&\Bigg|\int_0^T\partial_t\varphi_t\de t\int_0^t 2Z_s^N\Bigg\{\frac{1}{N}\sum_{x\in\mathbb{T}_N}\xi_s^N(x)\xi_s^N(x+1)
    \\&-\left\langle\pi_s^N, \pi_s^N\ast \cev i_\eps\right\rangle \Bigg\}\de s\Bigg|>\frac{\delta}{2}\Bigg\}=0.
\end{split}    
\end{equation}

Setting $\delta'=\frac{\delta}{4T\|\partial_t\varphi\|_\infty}$, for each $\eps>0$
the probability in \eqref{eq:prob_sgn} is bounded from above by
\begin{align*}
    \begin{split}
    &\prob_{\mu_N}\left\{\mysup_{t\in[0, T]}\left|\int_0^t Z_s^N\Bigg\{\frac{1}{N}\sum_{x\in\mathbb{T}_N}\xi_s^N(x)\xi_s^N(x+1)-\left\langle\pi_s^N, \pi_s^N\ast \cev i_\eps\right\rangle \Bigg\}\de s\right|>\delta'\right\}
    \end{split}
    \\&=\prob_{\mu_N}\left\{\mysup_{t\in[0, T]}\left|\int_0^t Z_s^N\frac{1}{N}\sum_{x\in\mathbb{T}_N}\xi_s^N\left(\frac{x+1}{N}\right)\Bigg\{\xi_s^N(x)-\pi_s^N\ast \cev i_\eps\left(\frac{x+1}{N}\right)\Bigg\}\de s\right|>\delta'\right\},
\end{align*}
which can further be bounded from above by
\begin{align}
    &\prob_{\mu_N}\left\{\mysup_{t\in[0, T]}\left|\int_0^t Z_s^N\frac{1}{N}\sum_{x\in\mathbb{T}_N}\xi_s^N(x+1)\left[\xi_s^N(x)-\cev\xi_s^{\eps N}(x)\right]\de s\right|>\frac{\delta'}{2}\right\}\label{eq:prob_sgn_1A}
    \\&+\prob_{\mu_N}\left\{\mysup_{t\in[0, T]}\left|\int_0^t Z_s^N\frac{1}{N}\sum_{x\in\mathbb{T}_N}\xi_s^N(x+1)\left[\cev\xi_s^{\eps N}(x)-\pi_s^N\ast \cev i_\eps\left(\frac{x+1}{N}\right)\right]\de s\right|>\frac{\delta'}{2}\right\}.\label{eq:prob_sgn_1B}
\end{align}

By the replacement lemma (Lemma~\ref{lemma:replacement}) with $G_s\equiv 1$ and $\psi_x(h)=\xi(x+1)$, we conclude that \eqref{eq:prob_sgn_1A} vanishes as $N\to\infty$ and $\eps\to0$. Finally, note that we can write $\cev\xi_s^{\eps N}(x)=\pi_s^N\ast \cev i_\eps\left(\frac{x}{N}\right)$ up to an error which vanishes as $\eps\to0$, so that we are only left to bound
\begin{equation*}
    \eqref{eq:prob_sgn_1B}\le \prob_{\mu_N}\left\{\mysup_{t\in[0, T]}\left|\int_0^t \left[\pi_s^N\ast \cev i_\eps\left(\frac{x}{N}\right)-\pi_s^N\ast \cev i_\eps\left(\frac{x+1}{N}\right)\right]\de s\right|>\frac{\delta'}{4}\right\}.
\end{equation*}
By the continuity of $\cev i_\eps$, this yields $\mylim_{\eps\to0}\myliminf_{N\to\infty}\eqref{eq:prob_sgn_1B}=0$.
\end{proof}

%%%%%%%%%%%%%%%%%%%%%%%%%%%%%%%%%%%%%%%%%%%%%%%%%%
%%%Characterisation of Limit Points%%%%%%%%%%%%%%%
%%%%%%%%%%%%%%%%%%%%%%%%%%%%%%%%%%%%%%%%%%%%%%%%%%
\subsection{Characterisation of the Limit Point}
Recall the definition of the probability measure $Q^N$ at the beginning of Section~\ref{sec:hydro}. We finally show that any limit point $Q$ of the sequence $\{Q^N\}_N$ is concentrated on trajectories that are absolutely continuous with respect to the Lebesgue measure with density given by a weak solution of \eqref{eq:coupled_equations}. 

\begin{proposition}\label{prop:limit_points}
Let $\nu^N$ be as in item iii) of Proposition~\ref{prop:integral} and let $\nu$ be a limit point of $\{\nu^N\}_N$. Then,
for $\nu$-almost every triple $(\pi, Y, Z)$, $\pi_t$ has a density $\rho_t$ with respect to the Lebesgue measure, the identity
\begin{equation*}
    \langle\rho_0, \phi_0\rangle+\int_0^T\left\langle\rho_t, \partial_t\phi_t+\frac{1}{2}\Delta\phi_t\right\rangle\de t
    +\int_0^TZ_t\langle\rho_t(1-\rho_t), \nabla \phi_t\rangle\de t=0 \;,
\end{equation*}
is satisfied for all $\phi\in C_0^{1, 2}([0, T]\times\mathbb{T})$, and $Z_t\in\sigma(Y_t)$ for all $t\in[0, T]$, with $\sigma:\mathbb{R}\to\mathcal{P}([-1, 1])$ defined in \eqref{def:sigma_sgn}.
\end{proposition}

\begin{proof} Following the same argument used in the proof of Proposition~\ref{prop:integral}, it suffices to show that, for any $\phi\in C_0^{1, 2}([0, T]\times\mathbb{T})$ and for any $\delta>0$,
\begin{equation}\label{eq:limit_points_1}
    \begin{split}
    \mylim_{\eps\to0}\myliminf_{N\to\infty}\nu^N\Bigg\{&\Bigg|\langle\pi_0^N, \phi_0\rangle+\int_0^T\left\langle\pi_t^N, \partial_t\phi_t+\frac{1}{2}\Delta\phi_t\right\rangle\de t
    \\&+\int_0^TZ_t^N\left\langle\pi_t^N, \left(1-\pi_t^N\ast \cev i_\eps\right)\nabla\phi_t\right\rangle\de t\Bigg|>\delta\Bigg\}=0
    \end{split}
\end{equation}
and
\begin{equation}\label{eq:limit_points_2}
    \nu\big\{Z_t\in\sigma(Y_t)\ \forall t\in[0, T]\big\}=1,
\end{equation}
where $\cev i_\eps$ is defined as in the proof of Proposition~\ref{prop:integral}.
We sum and subtract the martingale in \eqref{eq:martingale_1} (generalised to functions $\phi$ which are also allowed to depend on time) evaluated at time $T$ inside the absolute value in \eqref{eq:limit_points_1}: by Markov's inequality, we have that
\begin{align*}
    \prob_{\mu_N}\left\{\left|M_T^N(\phi)\right|>\frac{\delta}{2}\right\}&\le\frac{\expected_{\mu_N}\left[M^N_T(\phi)^2\right]}{\left(\frac{\delta}{2}\right)^2}
    \\&\le\frac{4}{\delta^2}\expected_{\mu_N}\left[\int_0^T\frac{1}{N^2}\sum_{x\in\mathbb{T}_N}\nabla_N\phi_t\left(\frac{x}{N}\right)^2\eta_t^N(x)\de t\right]\nonumber
    \\&\le\frac{4}{N\delta^2}\int_0^T\|\nabla\phi_t\|^2_\infty\de t,
\end{align*}
which vanishes as $N\to\infty$. Then, in order to show \eqref{eq:limit_points_1}, we are only left to prove that, for each $\phi\in C_0^{1,2}([0, T]\times\mathbb{T})$,
\begin{equation}\label{eq:prob_sup}
\begin{split}
    \prob_{\mu_N}\Bigg\{&\Bigg|\int_0^T\left\langle\pi_t^N, \frac{1}{2}\left(\Delta\phi_t-\Delta_N\phi_t\right)\right\rangle\de t+\int_0^TZ_t^N\left\langle\pi_t^N, \left(1-\pi_t^N\ast\cev i_\eps\right)\nabla\phi_t\right\rangle\de t
    \\&-\frac{1}{2}\int_0^TZ_t^N\text{tanh}\left(\frac{1}{N}\right)\sum_{x\in\mathbb{T}_N}\nabla_N\phi_t\left(\frac{x}{N}\right)\left[\xi_t^N(x)-\xi_t^N(x+1)\right]^2\de t\Bigg|>\frac{\delta}{2}\Bigg\}
\end{split}
\end{equation}
converges to zero as $N\to\infty$ and $\eps\to0$. Now, for each $\eps>0$ we have that \eqref{eq:prob_sup} is bounded from above by
\begin{align}
    &\prob_{\mu_N}\left\{\left|\int_0^T\left\langle\pi_t^N, \frac{1}{2}\left(\Delta\phi_t-\Delta_N\phi_t\right)\right\rangle\de t\right|>\frac{\delta}{6}\right\}\label{eq:ONE}
    \\\begin{split}&+\prob_{\mu_N}\Bigg\{\Bigg|\int_0^TZ_t^N\Bigg[\left\langle\pi_t^N, \nabla\phi_t\right\rangle
    \\&\phantom{+\prob_{\mu_N}\Bigg\{\Bigg|}-\frac{1}{2}\text{tanh}\left(\frac{1}{N}\right)\sum_{x\in\mathbb{T}_N}[\xi_t^N(x)+\xi_t^N(x+1)]\nabla_N\phi_t\left(\frac{x}{N}\right)\Bigg] \de t\Bigg|>\frac{\delta}{6}\Bigg\}
    \end{split}\label{eq:TWO}
    \\\begin{split}
    &+\prob_{\mu_N}\Bigg\{\Bigg|\int_0^T Z_t^N\Bigg[\left\langle\pi_t^N, \left(\pi_t^N\ast \cev i_\eps\right) \nabla\phi_t\right\rangle
    \\&\phantom{+\prob_{\mu_N}\Bigg\{\Bigg|}-\text{tanh}\left(\frac{1}{N}\right)\sum_{x\in\mathbb{T}_N}\xi_t^N(x)\xi_t^N(x+1)\nabla_N\phi_t\left(\frac{x}{N}\right)\Bigg]  
    \de t\Bigg|>\frac{\delta}{6}\Bigg\}\label{eq:THREE}
    \end{split}
\end{align}
To treat \eqref{eq:ONE} and \eqref{eq:TWO}, it is enough to use the Taylor expansion of $\phi_t$ to conclude that they both vanish as $N\to\infty$. As for \eqref{eq:THREE}, calling $\delta'=\frac{\delta}{6}$, we further bound it from above by
\begin{align}
    \begin{split}&\prob_{\mu_N}\Bigg\{\Bigg|\int_0^T Z_t^N\text{tanh}\left(\frac{1}{N}\right)\sum_{x\in\mathbb{T}_N}\xi_t^N(x+1)\nabla\phi_t\left(\frac{x+1}{N}\right)\Bigg[
    \pi_t^N\ast \cev i_\eps\left(\frac{x+1}{N}\right)
    \\&\phantom{\prob_{\mu_N}\Bigg\{\Bigg|}-\cev\xi_t^{\eps N}(x)\Bigg]\de t\Bigg|>\frac{\delta'}{3}
    \Bigg\}
    \end{split}\label{eq:FOUR}
    \\\begin{split}&+\prob_{\mu_N}\Bigg\{\Bigg|\int_0^T Z_t^N\text{tanh}\left(\frac{1}{N}\right)\sum_{x\in\mathbb{T}_N}\nabla\phi_t\left(\frac{x+1}{N}\right)\xi_t^N(x+1)\times
    \\&\phantom{+\prob_{\mu_N}\Bigg\{\Bigg|}\times\left[\xi_t^N(x)-\cev\xi_t^{\eps N}(x)\right]\de t\Bigg|>\frac{\delta'}{3}\Bigg\}
    \end{split}\label{eq:FIVE}
    \\\begin{split}&+\prob_{\mu_N}\Bigg\{\Bigg|\int_0^T Z_t^N\text{tanh}\left(\frac{1}{N}\right)\sum_{x\in\mathbb{T}_N}\xi_t^N(x)\xi_t^N(x+1)\times
    \\&\phantom{+\prob_{\mu_N}\Bigg\{\Bigg|}\times\left[\nabla\phi_t\left(\frac{x}{N}\right)-\nabla_N\phi_t\left(\frac{x+1}{N}\right)\right]\de t\Bigg|>\frac{\delta'}{3}\Bigg\}.\end{split}\label{eq:SIX}
\end{align}

Now, by the same argument used in the proof of Proposition~\ref{prop:integral}, \eqref{eq:FOUR} vanishes as $N\to\infty$ and $\eps\to0$. By the replacement lemma (Lemma~\ref{lemma:replacement}) with $G_t=\nabla\phi_t$ and $\psi_x(h)=\xi(x+1)$, and a Taylor expansion of $\text{tanh}(\frac{1}{N})$, we get that $\mylim_{N\to\infty}\mylim_{\eps\to0}\eqref{eq:FIVE}=0$. Finally, by using the Taylor expansion of $\phi_t$, we have that \eqref{eq:SIX} also vanishes as $N\to\infty$.

It only remains to show \eqref{eq:limit_points_2}, but this follows immediately from the continuity of $\sgn(\cdot)$ on $(-\infty, 0)$ and $(0, +\infty)$.
\end{proof}

We now proceed to show that the result above \dash together with the fact that $Y\in C([0, T], \mathbb{R})$ and $\rho\in L^2([0, T], \mathcal{H}^1)$ \dash implies that any limit point $\nu$ of $\{\nu^N\}_N$ is concentrated on paths satisfying $Z_t=\sgn(Y_t)$.

\begin{lemma}\label{lemma:sgn}
Let $f:[0, T]\to(0, 1]$ be a continuous map and let $Y:[0, T]\to\mathbb{R}$ solve the differential equation
\begin{equation*}
    \partial_t Y(t)=-Z(t) f(t)
\end{equation*}
started at $Y_0$, where $Z:[0, T]\to[-1, 1]$ satisfies $Z(t)\in\sigma(Y(t))$ for all $t\in[0, T]$, with $\sigma:\mathbb{R}\to\mathcal{P}([-1, 1])$ defined in \eqref{def:sigma_sgn}. Then, $Z(t)=\sgn(Y(t))$ for almost every $t\in[0, T]$.
\end{lemma}

\begin{proof}Since $Y$ is necessarily Lipschitz continuous,
we only need to show that, once $Y$ hits zero, it remains there. Assume that $Y_0=0$ and, by contradiction, that there exists $0<t\le T$ such that $Y(t)=\delta$ for some arbitrary $\delta>0$. Let $t_1=\myinf\{s\in [0, T]: Y_s=\delta\}$ and $t_0=\mysup\{s\in [0, t_1]: Y_s=0\}$. By continuity of $Y$, there exists $t_0<\tilde t<t_1$ such that $Y_{\tilde t}=\frac{\delta}{2}$. But then, $Y_r=Y_{\tilde t}-\int_0^r f(s)ds\le \frac{\delta}{2}$ for $r>\tilde t$ until $Y$ hits zero again, which contradicts either the definition of $t_0$, or the fact that $t_1<\infty$. The argument for $\delta<0$ is identical.
\end{proof}

\begin{proof}[Proof of Theorem~\ref{thm:hydro}] The only thing left to show is that $\rho$ belongs to $L^2([0, T], \mathcal{H}^1)$ and that the weak solutions of \eqref{eq:coupled_equations} are unique. Indeed, note that $\rho\in L^2([0, T], \mathcal{H}^1)$ implies, in particular, that $\rho_t$ is continuous on the torus for almost every $t\in [0, T]$, and thus $\langle \rho_t, 1-\rho_t\rangle>0$ for almost every $t\in[0, T]$, which allows us to apply Lemma~\ref{lemma:sgn}. The proof is postponed to Appendix~\ref{sec:app_energy_uniqueness}.
\end{proof}

%%%%%%%%%%%%%%%%%%%%%%%%%%%%%%%%%%%%%%%%%%%%%%%%%%
%%%EQUILIBIRUM FLUCTUATIONS%%%%%%%%%%%%%%%%%%%%%%%
%%%%%%%%%%%%%%%%%%%%%%%%%%%%%%%%%%%%%%%%%%%%%%%%%%
\section{Equilibrium Fluctuations}\label{sec:flucts}

%%%%%%%%%%%%%%%%%%%%%%%%%%%%%%%%%%%%%%%%%%%%%%%%%%
%%%Dynkin's Martingales%%%%%%%%%%%%%%%%%%%%%%%%%%%
%%%%%%%%%%%%%%%%%%%%%%%%%%%%%%%%%%%%%%%%%%%%%%%%%%
\subsection{Dynkin's Martingales}
Given $\phi\in\mathcal{S}(\mathbb{T})$, Dynkin's formula shows that the process $\{\mathscr{M}^N_t(\phi), t\in[0, T]\}$ defined by
\begin{equation}\label{eq:dynkin_2}
    \mathscr{M}_t^N(\phi):=\mathscr{U}^N_t(\phi)-\mathscr{U}^N_0(\phi)-\int_0^tN^2\mathcal{L}^N\mathscr{U}_s^N(\phi)\de s
\end{equation}
is a martingale with quadratic variation
\begin{equation}\label{eq:quad_var_2}
    \scal{\mathscr{M}^N(\phi)}_t=\int_0^t \left\{N^2\mathcal{L}^N\mathscr{U}_s^N(\phi)^2-2N^2\mathscr{U}_s^N(\phi)\mathcal{L}^N\mathscr{U}_s^N(\phi)\right\}\de s.
\end{equation}
One can check that \eqref{eq:dynkin_2} can be written as
\begin{equation}\label{eq:martingale_2}
    \begin{split}&\mathscr{M}_t^N(\phi)=\mathscr{U}^N_t(\phi)-\mathscr{U}_0^N(\phi)+\frac{1}{2\sqrt{N}}\int_0^t\sum_{x\in\mathbb{T}_N}\Delta_N\phi\left(\frac{x}{N}\right)\bar\xi^N_s(x)\de s
    \\&+\sqrt{N}\text{tanh}\left(\frac{1}{N^\gamma}\right)\int_0^t\sgn(Y_s^N)\sum_{x\in\mathbb{T}_N}\nabla_N\phi\left(\frac{x}{N}\right)\bar\xi^N_s(x)\bar\xi^N_s(x+1)\de s,
    \end{split}
\end{equation}
where $\Delta_N$ and $\nabla_N$ are defined in \eqref{eq:laplacian} and \eqref{eq:gradient}, respectively. For $\phi\in \mathcal{S}(\mathbb{T})$, define the processes $\{\mathscr{K}^N_t(\phi), t\in[0, T]\}$ and $\{\mathscr{B}^N_t(\phi), t\in[0, T]\}$ as 
\begin{equation}\label{eq:mathcal_K}
    \mathscr{K}^N_t(\phi):=\frac{1}{2\sqrt{N}}\int_0^t\sum_{x\in\mathbb{T}_N}\Delta_N\phi\left(\frac{x}{N}\right)\bar\xi^N_s(x)\de s
\end{equation}
and 
\begin{equation}\label{eq:mathcal_B}
    \mathscr{B}^N_t(\phi):=\sqrt{N}\text{tanh}\left(\frac{1}{N^\gamma}\right)\int_0^t\sgn(Y_s^N)\sum_{x\in\mathbb{T}_N}\nabla_N\phi\left(\frac{x}{N}\right)\bar\xi^N_s(x)\bar\xi^N_s(x+1)\de s,
\end{equation}
so that \eqref{eq:martingale_2} becomes
\begin{equation}\label{eq:fluctuations}
    \mathscr{U}^N_t(\phi)=\mathscr{U}_0^N(\phi)+\mathscr{M}_t^N(\phi)+\mathscr{K}_t^N(\phi)+\mathscr{B}_t^N(\phi).
\end{equation}
Similarly to the hydrodynamic case, equation \eqref{eq:fluctuations} will be the building block of our argument, and in the following sections we will study each of its terms separately.

%%%%%%%%%%%%%%%%%%%%%%%%%%%%%%%%%%%%%%%%%%%%%%%%%%
%%%Correlations and Tightness%%%%%%%%%%%%%%%%%%%%%
%%%%%%%%%%%%%%%%%%%%%%%%%%%%%%%%%%%%%%%%%%%%%%%%%%
\subsection{Correlations and Tightness}
From now on we let $N=2p$ with $p$ prime, and we proceed to show tightness of the sequence of measures $\{\mathcal{Q}^N\}_N$ with respect to the Skorokhod topology of $D([0, T], \mathcal{S}'(\mathbb{T}))$. We will actually show tightness of each of the terms on the RHS of \eqref{eq:fluctuations}, as well as some properties of the limit points:
\begin{proposition}\label{prop:tightness_2}
For $\gamma>\frac{6}{7}$ and starting from $\mu_N^*$, the sequences of processes $\{\mathscr{U}_t^N, t\in[0, T]\}_N, \{\mathscr{M}_t^N, t\in[0, T]\}_N, \{\mathscr{K}_t^N, t\in[0, T]\}_N$ and $\{\mathscr{B}_t^N, t\in[0, T]\}_N$ are tight with respect to the Skorokhod topology of $D([0, T], \mathcal{S}'(\mathbb{T}))$.
\end{proposition}

\begin{remark} In fact, one should show tightness of the ``original" terms appearing in \eqref{eq:dynkin_2}. However, it is easy to check that, once we have tightness of \eqref{eq:fluctuations}, the same argument yields tightness of the remainder terms. 
\end{remark}

As well as Aldous's Criterion (Proposition~\ref{prop:aldous}), the following results will be useful.

\begin{proposition}[Mitoma's Criterion, \cite{mit83}]\label{prop:mitoma} Let $\mathcal{S}$ be a nuclear Fréchet space. A sequence of $\mathcal{S}'$-valued stochastic processes $\{\mathscr{Y}^N_t, t\in[0, T]\}_N$ taking values in $D([0, T], \mathcal{S}')$ is tight with respect to the Skorokhod topology if and only if, for any $\phi\in\mathcal{S}$, the sequence of real-valued processes $\{\mathscr{Y}^N_t(\phi), t\in[0, T]\}_N$ is tight with respect to the Skorokhod topology of $D([0, T], \mathbb{R})$. 
\end{proposition}

\begin{proposition}[Kolmogorov's Criterion]\label{prop:kolmogorov} A sequence of real-valued stochastic processes $\{X^N_t, t\in[0, T]\}_N$ with continuous trajectories is tight with respect to the uniform topology of $C([0, T], \mathbb{R})$ if the sequence of real-valued random variables $\{X_0^N\}_N$ is tight and if there exist constants $K, \gamma_1, \gamma_2>0$ such that
\begin{equation*}
    \expected\left[\left|X_t^N-X_r^N\right|^{\gamma_1}\right]\le K\left|t-r\right|^{1+\gamma_2}
\end{equation*}
for any $r, t\in[0, T]$ and any $N$. Moreover, for any $\alpha<\frac{\gamma_2}{\gamma_1}$, any limit point is concentrated on $\alpha$-H\"older continuous paths.
\end{proposition}

A fundamental ingredient will be given by the following asymptotic result about correlations with respect to the invariant measure, whose proof is postponed to Appendix~\ref{sec:app_correlations}:

\begin{theorem}[Bound for the $2m$-point Correlations]\label{thm:2m_correlations}
Let $N=2p$ with $p$ prime. For each $\gamma>0$ and $m$ positive integer, there exists a constant $C$ such that, for $N$ sufficiently big and for any $x_1, \ldots, x_{2m}$ in $\mathbb{T}_N$ pairwise distinct,
\begin{equation*}
    \left|\expected_{\mu_{N}^*}\left[\prod_{i=1}^{2m}\bar\xi^N(x_i)\right]\right|\le \frac{C}{N^m}.
\end{equation*}
\end{theorem}

\begin{proof}[Proof of Proposition~\ref{prop:tightness_2}] 
Since $\mathcal{S}(\mathbb{T})$ is a nuclear Fréchet space, by Mitoma's Criterion it suffices to show tightness of the sequences of fields evaluated at $\phi$ for any $\phi\in\mathcal{S}(\mathbb{T})$.  Note that, by Theorem~\ref{thm:2m_correlations},
\begin{align*}
    \expected_{\mu_N^*}\left[\mathscr{U}_0^N(\phi)^2\right]&=\expected_{\mu_N^*}\left[\left(\frac{1}{\sqrt{N}}\sum_{x\in\mathbb{T}_N}\bar\xi_0^N(x)\phi\left(\frac{x}{N}\right)\right)^2\right]
    \\&=\frac{1}{N}\sum_{x\in\mathbb{T}_N}\expected_{\mu_N^*}\left[\bar\xi_0^N(x)^2\right]\phi\left(\frac{x}{N}\right)^2
    \\&\phantom{=}+\frac{1}{N}\sum_{\substack{x, y\in\mathbb{T}_N\\x\ne y}}\expected_{\mu_N^*}\left[\bar\xi_0^N(x)\bar\xi_0^N(y)\right]\phi\left(\frac{x}{N}\right)\phi\left(\frac{y}{N}\right)
    \\&\le\|\phi\|_{\infty}^2\left(\sigma^2+C\right),
\end{align*}
where $\sigma^2:=\Var(\bar\xi^N(x))=\frac{1}{4}$. Thus, $\{\mathscr{U}_0^N(\phi)\}_N$ is tight.

Consider the martingale term: it is not hard to check that the quadratic variation \eqref{eq:quad_var_2} can be rewritten as
\begin{equation}\label{eq:quad_var_2b}
    \scal{\mathscr{M}^N(\phi)}_t=\int_0^t\frac{1}{N}\sum_{x\in\mathbb{T}_N}\nabla_N\phi\left(\frac{x}{N}\right)^2\eta_s^N(x)\de s,
\end{equation}
where
\begin{equation}\label{eq:eta_2}
    \eta_s^N(x):=\frac{1}{2}\left[\bar\xi_s^N(x)-\bar\xi_s^N(x+1)\right]^2+\frac{1}{2}\,\text{tanh}\left(\frac{1}{N^\gamma}\right)\sgn(Y_s^N)\left[\xi_s^N(x)-\xi_s^N(x+1)\right].
\end{equation}
Since $|\eta_s^N(x)|$ is uniformly bounded in $N$, given $\eps>0$ and a stopping time 
$\tau$ almost surely bounded by $T$, we have
\begin{align*}
    \prob_{\mu_N^*}\left(\left|\mathscr{M}^N_{\tau+\gamma}(\phi)-\mathscr{M}^N_\tau(\phi)\right|\ge\eps\right)&\le\frac{1}{\eps^2}\expected_{\mu_N^*}\left[\left(\mathscr{M}^N_{\tau+\gamma}(\phi)-\mathscr{M}_\tau^N(\phi)\right)^2\right]
    \\&\le\frac{1}{\eps^2}\expected_{\mu_N^*}\left[\int_\tau^{\tau+\gamma}\frac{1}{N}\sum_{x\in\mathbb{T}_N}\nabla_N\phi\left(\frac{x}{N}\right)^2\eta_s^N(x)\de s\right]\nonumber
    \\&\le\frac{C\gamma\|\nabla\phi\|^2_\infty}{\eps^2}
\end{align*}
for some constant $C$ independent of $N$. Also, from \eqref{eq:quad_var_2b} we get that $\expected_{\mu_N^*}[\mathscr{M}_t^N(\phi)^2]\le t\|\nabla\phi\|_{\infty}^2$. Tightness of $\left\{\mathscr{M}_t^N(\phi), t\in[0, T]\right\}_N$ then follows from Aldous's Criterion.

Now consider \eqref{eq:mathcal_K}: by the stationarity of $\mu_N^*$, the Cauchy--Schwarz inequality and Theorem~\ref{thm:2m_correlations}, for any $0\le r<t\le T$ we have that
\begin{align*}
    \expected_{\mu_N^*}\left[\left(\mathscr{K}_t^N(\phi)-\mathscr{K}_r^N(\phi)\right)^2\right]&=\expected_{\mu_N^*}\left[\left(\frac{1}{2\sqrt{N}}\int_r^t\sum_{x\in\mathbb{T}_N}\Delta_N\phi\left(\frac{x}{N}\right)\bar\xi^N_s(x)\de s\right)^2\right]
    \\&\le (t-r)\int_r^t\expected_{\mu_N^*}\left[\left(\frac{1}{2\sqrt{N}}\sum_{x\in\mathbb{T}_N}\Delta_N\phi\left(\frac{x}{N}\right)\bar\xi^N_s(x)\right)^2\right]\de s
    \\&=\frac{t-r}{4N}\int_r^t\Bigg\{\sum_{x\in\mathbb{T}_N}\Delta_N\phi\left(\frac{x}{N}\right)^2\expected_{\mu_N^*}\left[\bar\xi^N_s(x)^2\right]
    \\&\phantom{=}+\sum_{\substack{x, y\in\mathbb{T}_N\\x\ne y}}\Delta_N\phi\left(\frac{x}{N}\right)\Delta_N\phi\left(\frac{y}{N}\right)\expected_{\mu_N^*}\left[\bar\xi^N_s(x)\bar\xi^N_s(y)\right]\Bigg\}\de s
    \\& \le \frac{(t-r)^2}{4}\|\Delta\phi\|^2_\infty(\sigma^2+C).
\end{align*}
Hence, by Kolmogorov's Criterion, the sequence of processes $\{\mathscr{K}^N_t(\phi), t\in[0, T]\}_N$ is tight with respect to the uniform topology of $C([0, T], \mathbb{R})$ and any limit point has $\alpha$-H\"older continuous trajectories for any $\alpha<\frac{1}{2}$.

Finally, tightness of \eqref{eq:mathcal_B} follows directly from Proposition~\ref{prop:nonlinear_term} in the next section.
\end{proof}

%%%%%%%%%%%%%%%%%%%%%%%%%%%%%%%%%%%%%%%%%%%%%%%%%%
%%%One-Block Estimate and the Nonlinear Term%%%%%%
%%%%%%%%%%%%%%%%%%%%%%%%%%%%%%%%%%%%%%%%%%%%%%%%%%
\subsection{One-Block Estimate and the Nonlinear Term}
The goal of this section is to show that the variance of the non-linear term \eqref{eq:mathcal_B} vanishes in the limit as $N\to\infty$:

\begin{proposition}\label{prop:nonlinear_term}
For $\gamma>\frac{6}{7}$ and any $\phi\in\mathcal{S}(\mathbb{T})$, we have that $\mylim_{N\to\infty}\expected_{\mu_N^*}[\mathscr{B}^N_t(\phi)^2]=0.$
\end{proposition}

In order to prove the result above, we will need a version of the classical one-block scheme introduced in \cite{gpv88}. Given a function $G:\mathbb{T}\to\mathbb{R}$, let
\begin{equation*}
    \|G\|_{2, N}^2:=\frac{1}{N}\sum_{x\in\mathbb{T}_N}G\left(\frac{x}{N}\right)^2,
\end{equation*}
and given $\varphi:\Omega_N\to\mathbb{R}$, let
\begin{equation*}
    \|\varphi\|_{L^2(\mu_N^*)}^2:=\int_{\Omega_N}\varphi(h)^2\de \mu_N^*.
\end{equation*}
Given an integer $1\le L\le N$, define 
\begin{equation*}
    \vec \Sigma^L(x):=\frac{1}{L}\sum_{y=x+2}^{x+L}\sum_{z=x+1}^{y-1}\left[\bar\xi(z)-\bar\xi(z+1)\right]\left[\frac{1+\sigma(z)}{2}\right]
\end{equation*}
and
\begin{equation*}
    \cev \Sigma^L(x):=\frac{1}{L}\sum_{y=x-L}^{x-2}\sum_{z=y+1}^{x-1}\left[\bar\xi(z)-\bar\xi(z+1)\right]\left[\frac{1+\sigma(z)}
    {2}\right],
\end{equation*}
where 
\begin{equation}\label{eq:sigma}
    \sigma(z)=\sigma_z(h):=\frac{\sgn(Y(h^{z, z+1}))\mu_N^*(h^{z, z+1})}{\sgn(Y(h))\mu_N^*(h)}.
\end{equation}
Throughout, if $L$ is not an integer, we will keep the notation $L$ to denote $\lfloor L\rfloor$. Also, for each $x\in\mathbb{T}_N$, let $\tau_x$ denote the translation operator by $x$, namely $\tau_xh(\cdot):=h(x+\cdot)$. Then, the following holds:

\begin{lemma}[One-Block Estimate]\label{lemma:one_block} Let $\ell_0\in\mathbb{N}$ and let $\varphi:\Omega_N\to\mathbb{R}$ be a  mean-zero function with respect to $\mu_N^*$, and such that its support does not intersect the set of points $\{0, \ldots, \ell_0\}$. There exists a constant $C$ independent of $N$ such that, for any $t>0$ and any bounded $G:\mathbb{T}\to\mathbb{R}$,
\begin{equation*}
    \expected_{\mu_N^*}\left[\left(\sum_{x\in\mathbb{T}_N} G\left(\frac{x}{N}\right)\int_0^t\sgn(Y_s^N)\varphi(\tau_xh_s)\vec \Sigma_s^{\ell_0}(x)\de s\right)^2\right]\le \frac{Ct\ell_0^2}{N}\|\varphi\|_{L^2(\mu_N^*)}^2\|G\|_{2, N}^2.
\end{equation*}
The same result holds with $\vec \Sigma^{\ell_0}$ replaced with $\cev \Sigma^{\ell_0}$, as long as the support of $\varphi$ does not intersect the set of points $\{-\ell_0, \ldots, 0\}$ instead.
\end{lemma}

This lemma is almost identical to \cite[Proposition~5]{gjs17}, except for the presence of $\frac{1+\sigma(z)}{2}$: indeed, calling
\begin{equation}\label{eq:right_avg_2}
    \vec\xi^{L}(x):=\frac{1}{L}\sum_{y=x+1}^{x+L}\bar\xi(y) 
\end{equation}
the right average of centred particles, we have that
\begin{equation*}
    \bar\xi(x+1)-\vec\xi^{\ell_0}(x)=\frac{1}{L}\sum_{y=x+2}^{x+L}\sum_{z=x+1}^{y-1}\left[\bar\xi(z)-\bar\xi(z+1)\right].
\end{equation*}
Compared to the statement and hypotheses of that result, note that not only do we have the additional factor $\sgn(Y_s^N)$, but also \dash and more importantly \dash our measure is \textit{not} invariant under nearest-neighbour exchanges. By following an argument similar to the one given in \cite{gjs17}, we will see how the introduction of $\frac{1+\sigma(z)}{2}$ precisely makes up for the resulting correction.

\begin{proof}[Proof of Lemma~\ref{lemma:one_block}]
By the Kipnis--Varadhan inequality \cite[Lemma~2.4]{klo12} and \eqref{eq:dir=gamma}, the expectation in the statement is bounded from above by
\begin{align}
    &Ct\left\|\sum_{x\in\mathbb{T}_N} G\left(\frac{x}{N}\right)\sgn(Y(h))\varphi(\tau_xh)\vec \Sigma^{\ell_0}(x)\right\|_{-1}^2\nonumber
    \\\begin{split}&=Ct\cdot\mysup_{f\in L^2(\mu_N^*)}\bigg\{2\int_{\Omega_N}\sum_{x\in\mathbb{T}_N} G\left(\frac{x}{N}\right)\sgn(Y(h))\varphi(\tau_xh)\vec\Sigma^{\ell_0}(x)f(h)\de \mu_N^*
    \\&\phantom{=}-N^2\mathscr{D}_N\left(f, \mu_N^*\right)\bigg\},\label{eq:sup_3}
    \end{split}
\end{align}
where the Dirichlet form $\mathscr{D}_N(f, \mu_N^*)$ was defined in \eqref{eq:def_dirichlet}. To ease the notation, we drop the superscript $N$ and we call $Y=Y(h)$ and $Y^{z, z+1}=Y(h^{z, z+1})$. Now we write $\vec \Sigma^{\ell_0}(x)$ as
\begin{align}
    \vec \Sigma^{\ell_0}(x)=&\ \frac{1}{\ell_0}\sum_{y=x+2}^{x+\ell_0}\sum_{z=x+1}^{y-1}\left[\bar\xi(z)-\bar\xi(z+1)\right]\label{eq:S1}
    \\&+\frac{1}{\ell_0}\sum_{y=x+2}^{x+\ell_0}\sum_{z=x+1}^{y-1}\left[\bar\xi(z)-\bar\xi(z+1)\right]\left[\frac{\sigma(z)-1}{2}\right]\label{eq:S2}
\end{align}
in \eqref{eq:sup_3}: then, the integral term in \eqref{eq:sup_3} corresponding to \eqref{eq:S1} can be written as
\begin{align}
    &2\int_{\Omega_N}\sum_{x\in\mathbb{T}_N} G\left(\frac{x}{N}\right)\sgn(Y)\varphi(\tau_xh)\left(\frac{1}{\ell_0}\sum_{y=x+2}^{x+\ell_0}\sum_{z=x+1}^{y-1}\left[\bar\xi(z)-\bar\xi(z+1)\right]\right)f(h)\de \mu_N^*\nonumber
    \\&=\int_{\Omega_N}\sum_{x\in\mathbb{T}_N} G\left(\frac{x}{N}\right)\sgn(Y)\varphi(\tau_xh)\left(\frac{1}{\ell_0}\sum_{y=x+2}^{x+\ell_0}\sum_{z=x+1}^{y-1}\left[\bar\xi(z)-\bar\xi(z+1)\right]\right)\times\nonumber
    \\&\phantom{=}\times\left[f(h)-f(h^{z, z+1})\right]\de \mu_N^* \label{eq:first_add}
    \\&\phantom{=}+\int_{\Omega_N}\sum_{x\in\mathbb{T}_N} G\left(\frac{x}{N}\right)\sgn(Y)\varphi(\tau_xh)\left(\frac{1}{\ell_0}\sum_{y=x+2}^{x+\ell_0}\sum_{z=x+1}^{y-1}\left[\bar\xi(z)-\bar\xi(z+1)\right]\right)\times\nonumber
    \\&\phantom{=}\times\left[f(h)+f(h^{z, z+1})\right]\de \mu_N^*.\label{eq:second_add} 
\end{align}
By Young's inequality, for any set of positive real numbers $\{A_x\}_{x\in\mathbb{T}_N}$ we have that \eqref{eq:first_add} is bounded from above by
\begin{align*}
    &\frac{1}{\ell_0}\sum_{x\in\mathbb{T}_N}\sum_{y=x+2}^{x+\ell_0}\sum_{z=x+1}^{y-1} G\left(\frac{x}{N}\right)\frac{A_x}{2}\int_{\Omega_N}q^N_{z, z+1}(h)\left[f(h)-f(h^{z, z+1})\right]^2\de \mu_N^* 
    \\&+\frac{1}{\ell_0}\sum_{x\in\mathbb{T}_N}\sum_{y=x+2}^{x+\ell_0}\sum_{z=x+1}^{y-1} G\left(\frac{x}{N}\right)\frac{1}{2A_x}\int_{\Omega_N}\varphi(\tau_xh)^2\frac{\left[\bar\xi(z)-\bar\xi(z+1)\right]^2}{q_{z, z+1}^N(h)}\de \mu_N^*.
\end{align*}
By \eqref{eq:dir=gamma}, choosing $A_x=\frac{N^2}{4\ell_0G\left(\frac{x}{N}\right)}$ the first integral term is bounded from above by $N^2\mathscr{D}_N(f, \mu_N^*)$. As for the the second term, we get
\begin{align*}
    &\frac{2}{N^2}\sum_{x\in\mathbb{T}_N}\sum_{y=x+2}^{x+\ell_0}\sum_{z=x+1}^{y-1} G\left(\frac{x}{N}\right)^2\int_{\Omega_N}\varphi(\tau_xh)^2\frac{\left[\bar\xi(z)-\bar\xi(z+1)\right]^2}{q_{z, z+1}^N(h)}\de \mu_N^*
    \\&\le \frac{C}{N^2}\sum_{x\in\mathbb{T}_N}\sum_{y=x+2}^{x+\ell_0}\sum_{z=x+1}^{y-1} G\left(\frac{x}{N}\right)^2\|\varphi\|_{L^2(\mu_N^*)}^2
    \\&\le \frac{C\ell_0^2}{N}\|G\|_{2, N}^2\|\varphi\|_{L^2(\mu_N^*)}^2
\end{align*}
for some positive constant $C$ independent of $N$.

We now turn to \eqref{eq:second_add}: by performing the change of variable $h\mapsto h^{z, z+1}$ and by using the hypothesis on the support of $\varphi$, we get that this term is equal to
\begin{equation*}
\begin{split}
    &\int_{\Omega_N}\sum_{x\in\mathbb{T}_N} G\left(\frac{x}{N}\right)\sgn(Y)\varphi(\tau_xh)\Bigg(\frac{1}{\ell_0}\sum_{y=x+2}^{x+\ell_0}\sum_{z=x+1}^{y-1}\left[\bar\xi(z)-\bar\xi(z+1)\right]\times
    \\&\times\left[1-\frac{\sgn(Y^{z, z+1})\mu_N^*(h^{z, z+1})}{\sgn(Y)\mu_N^*(h)}\right]\Bigg)f(h)\de \mu_N^*.
\end{split}
\end{equation*}
But then, this term cancels with the contribution in \eqref{eq:sup_3} given by \eqref{eq:S2}, so we are done.
\end{proof}
We will also use the following additional asymptotic results about correlations of slopes, proved in Appendix~\ref{sec:app_correlations}. For $j\in\mathbb{Z}$ odd, let $\Omega_N^j=\{h\in\Omega_N: Y(h)=j\}$. Note that this is the same notation used in \eqref{eq:Omega_k} but with different meaning, as now we consider the value of $Y(h)$ with its sign. Also, let $\Omega_N^{>1}=\bigcup_{j>1}\Omega_N^j$ and $\Omega_N^{<-1}=\bigcup_{j<-1}\Omega_N^j$. Then, the following holds:

\begin{theorem}\label{thm:restriction>1} Let $N=2p$ with $p$ prime. For each $\gamma>0$ and $m$ non-negative integer, there exists a constant $C$ such that, for $N$ sufficiently big and for any $x_1, \ldots, x_{2m}\in\mathbb{T}_N$ pairwise distinct,
\begin{equation}
    \left|\expected_{\mu_{N}^*}\left[\boldsymbol{1}_{\Omega_N^1}\prod_{i=1}^{2m}\bar\xi^N(x_i)\right]\right|\le\frac{C}{N^{m+\gamma}}.\label{eq:Lm}
\end{equation}
The same holds when $\boldsymbol{1}_{\Omega_N^1}$ is substituted with $\boldsymbol{1}_{\Omega_N^{-1}}$.
\end{theorem}

\begin{theorem}\label{thm:restriction1} Let $N=2p$ with $p$ prime. For each $\gamma>0$ and $m$ positive integer, there exists a constant $C$ such that, for $N$ sufficiently big and for any $x_1, \ldots, x_{2m}\in\mathbb{T}_N$ pairwise distinct,
\begin{equation*}
    \left|\expected_{\mu_{N}^*}\left[\boldsymbol{1}_{\Omega_N^{>1}}\prod_{i=1}^{2m}\bar\xi^N(x_i)\right]\right|\le\frac{C}{N^m}.
\end{equation*}
The same holds when $\boldsymbol{1}_{\Omega_N^{>1}}$ is substituted with $\boldsymbol{1}_{\Omega_N^{<-1}}$.
\end{theorem}

\begin{proof}[Proof of Proposition~\ref{prop:nonlinear_term}] Throughout, $C$ will denote a generic positive constant independent of $N$ but changing from line to line. Recalling the definition of the non-linear term in \eqref{eq:mathcal_B}, by the inequality $(x+y)^2\le 2x^2+2y^2$ and performing a Taylor expansion of $\text{tanh}(\frac{1}{N})$, for any $0\le t\le T$ and any $1\le \ell_0\le N-1$ we have that
\begin{align}
    &\expected_{\mu_N^*}\left[\mathscr{B}_t^N(\phi)^2\right]\nonumber
    \\&=N\,\text{tanh}^2\left(\frac{1}{N^\gamma}\right)\expected_{\mu_N^*}\left[\left(\int_0^t\sgn(Y_s^N)\sum_{x\in\mathbb{T}_N}\nabla_N\phi\left(\frac{x}{N}\right)\bar\xi^N_s(x)\bar\xi^N_s(x+1)\de s\right)^2\right]\nonumber
    \\&\le CN^{1-2\gamma}\expected_{\mu_N^*}\left[\left(\int_0^t\sgn(Y_s^N)\sum_{x\in\mathbb{T}_N}\nabla_N\phi\left(\frac{x}{N}\right)\bar\xi^N_s(x)\left[\bar\xi^N_s(x+1)-\vec\xi_s^{\ell_0}(x)\right]\de s\right)^2\right]\label{eq:a}
    \\&\phantom{\le}+CN^{1-2\gamma}\expected_{\mu_N^*}\left[\left(\int_0^t\sgn(Y_s^N)\sum_{x\in\mathbb{T}_N}\nabla_N\phi\left(\frac{x}{N}\right)\bar\xi^N_s(x)\vec\xi_s^{\ell_0}(x)\de s\right)^2\right], \label{eq:b}
\end{align}
where $\vec\xi^{\ell_0}$ is defined in \eqref{eq:right_avg_2}. By the Cauchy-Schwarz inequality and Theorem~\ref{thm:2m_correlations}, the second term satisfies
\begin{align*}
    \eqref{eq:b}&\le Ct^2N^{1-2\gamma}\expected_{\mu_N^*}\left[\left(\sum_{x\in\mathbb{T}_N}\nabla_N\phi\left(\frac{x}{N}\right)\bar\xi^N(x)\vec\xi^{\ell_0}(x)\right)^2\right]
    \\&\le\frac{Ct^2N^{1-2\gamma}}{\ell_0^2}\|\nabla\phi\|_\infty^2\expected_{\mu_N^*}\left[\sum_{x\in\mathbb{T}_N}\sum_{y\in\mathbb{T}_N}\sum_{i=1}^{\ell_0}\sum_{j=1}^{\ell_0}\bar\xi^N(x)\bar\xi^N(y)\bar\xi^N(x+i)\bar\xi^N(y+j)\right]
    \\&\le\frac{Ct^2N^{1-2\gamma}}{\ell_0^2}\|\nabla\phi\|_\infty^2\left(N\ell_0+\frac{CN\ell_0^2}{N}+\frac{CN^2\ell_0^2}{N^2}\right)
    \\&\le Ct^2\|\nabla\phi\|_\infty^2\left(\frac{N^{2-2\gamma}}{\ell_0}+N^{1-2\gamma}\right).
\end{align*}
Now, note that
\begin{equation*}
    \bar\xi(x+1)-\vec\xi^{\ell_0}(x)=\frac{1}{\ell_0}\sum_{y=x+2}^{x+\ell_0}\sum_{z=x+1}^{y-1}\left[\bar\xi(z)-\bar\xi(z+1)\right],
\end{equation*}
so by summing and subtracting $\sigma_s(z)=\sigma_z(h_s)$, defined in \eqref{eq:sigma}, we get that \eqref{eq:a} is bounded from above by
\begin{align}
    \begin{split}&CN^{1-2\gamma}\expected_{\mu_N^*}\Bigg[\Bigg(\int_0^t\sgn(Y_s^N)\sum_{x\in\mathbb{T}_N}\nabla_N\phi\left(\frac{x}{N}\right)\bar\xi^N_s(x)\times
    \\&\times\left\{\frac{1}{\ell_0}\sum_{y=x+2}^{x+\ell_0}\sum_{z=x+1}^{y-1}\left[\bar\xi^N_s(z)-\bar\xi^N_s(z+1)\right]\left[\frac{1+\sigma_s(z)}{2}\right]\right\}\de s\Bigg)^2\Bigg]\label{eq:C}
    \end{split}
    \\\begin{split}&+CN^{1-2\gamma}\expected_{\mu_N^*}\Bigg[\Bigg(\int_0^t\sgn(Y_s^N)\sum_{x\in\mathbb{T}_N}\nabla_N\phi\left(\frac{x}{N}\right)\bar\xi^N_s(x)\times
    \\&\times\left\{\frac{1}{\ell_0}\sum_{y=x+2}^{x+\ell_0}\sum_{z=x+1}^{y-1}\left[\bar\xi^N_s(z)-\bar\xi^N_s(z+1)\right]\left[\frac{1-\sigma_s(z)}{2}\right]\right\}\de s\Bigg)^2\Bigg].\label{eq:D}
    \end{split}
\end{align}
Applying Lemma~\ref{lemma:one_block} with $\varphi(h)=\xi(0)$ and $G=\nabla_N\phi$, we obtain 
\begin{align*}
    \eqref{eq:C}\le \frac{Ct\ell_0^2}{N^{2\gamma}}\|\nabla\phi\|_\infty^2.
\end{align*} 
As for \eqref{eq:D}, calling $\alpha^N(z):=\left[\bar\xi^N(z)-\bar\xi^N(z+1)\right]\left[1-\sigma^N(z)\right]$, by the Cauchy-Schwarz inequality we have that 
\begin{align}
    \eqref{eq:D}&\le\frac{Ct^2N^{1-2\gamma}}{\ell_0^2}\|\nabla\phi\|_\infty^2\expected_{\mu_N^*}\left[\left(\sum_{x\in\mathbb{T}_N}\bar\xi^N(x)\sum_{y=x+1}^{x+\ell_0}\sum_{z=x+1}^{y-1}\alpha^N(z)\right)^2\right]\nonumber
    \\\begin{split}&=\frac{Ct^2N^{1-2\gamma}}{\ell_0^2}\|\nabla\phi\|_\infty^2\expected_{\mu_N^*}\Bigg[\sum_{x\in\mathbb{T}_N}\sum_{i=x+1}^{x+\ell_0}\sum_{z=x+1}^{i-1}\bar\xi^N(x)\alpha^N(z)\times\\&\phantom{=}\times\sum_{y\in\mathbb{T}_N}\sum_{j=y+1}^{y+\ell_0}\sum_{w=y+1}^{j-1}\bar\xi^N(y)\alpha^N(w)\Bigg].\label{eq:six_sum}
    \end{split}
\end{align}
We split \eqref{eq:six_sum} into four terms by multiplying the random variable inside the expectation by the sum of the indicator functions of $\Omega_N^{>1}, \Omega_N^1, \Omega_N^{-1}$ and $\Omega_N^{<-1}$. First, consider the restriction to $\Omega_N^{>1}$: by splitting the sum inside of the expectation into the two cases $x=y$ and $x\ne y$ on $\Omega_N^{>1}$ we get that the expectation in \eqref{eq:six_sum} is bounded from above by
\begin{align}           
    &\expected_{\mu_N^*}\left[\boldsymbol{1}_{\Omega_N^{>1}}\sum_{x\ne y}\sum_{i=x+1}^{x+\ell_0}\sum_{z=x+1}^{i-1}\sum_{j=y+1}^{y+\ell_0}\sum_{w=y+1}^{j-1}\bar\xi^N(x)\bar\xi^N(y)\alpha^N(z)\alpha^N(w)\right]\label{eq:larger}    \\&+\expected_{\mu_N^*}\left[\boldsymbol{1}_{\Omega_N^{>1}}\sum_{x\in\mathbb{T}_N}\bar\xi^N(x)^2\left(\sum_{i=x+1}^{x+\ell_0}\sum_{z=x+1}^{i-1}\alpha^N(z)\right)^2\right].\label{eq:smaller}
\end{align}
Now, on $\Omega_N^{>1}$, for any $z\in\mathbb{T}_N$ we have that
\begin{align*}
    \alpha^N(z)&=1-\frac{\mu_N^*(h^{z, z+1})}{\mu_N^*(h)}
    \\\begin{split}&=\left(1-e^{2N^{-\gamma}}\right)\left[\frac{1}{2}+\bar\xi^N(z)\right]\left[\frac{1}{2}-\bar\xi^N(z+1)\right]
    \\&\phantom{=}+\left(1-e^{-2N^{-\gamma}}\right)\left[\frac{1}{2}-\bar\xi^N(z)\right]\left[\frac{1}{2}+\bar\xi^N(z+1)\right].
    \end{split}
\end{align*}
By Taylor expanding $e^{2N^{-\gamma}}$ and $e^{-2N^{-\gamma}}$, we get that $|\alpha^N(z)|\le\frac{C}{N^\gamma}$, so by applying Theorem~\ref{thm:restriction>1} with $m=2$, it is straightforward to see that
\begin{equation*}
   \eqref{eq:larger}\le\frac{C\ell_0^4}{N^{2\gamma-1}}.
\end{equation*}
Using the same bound $|\alpha^N(z)|\le\frac{C}{N^\gamma}$, we also get that
\begin{equation*}
    \eqref{eq:smaller}\le\frac{C\ell_0^4}{N^{2\gamma-1}}.
\end{equation*}
The estimate on $\Omega_N^{<-1}$ is identical. As for $\Omega_N^1$, the computation is still very similar: once again, we bound from above the expectation in \eqref{eq:six_sum} restricted to $\Omega_N^1$ by
\begin{align}
    &\expected_{\mu_N^*}\left[\boldsymbol{1}_{\Omega_N^1}\sum_{x\ne y}\sum_{i=x+1}^{x+\ell_0}\sum_{z=x+1}^{i-1}\sum_{j=y+1}^{y+\ell_0}\sum_{w=y+1}^{j-1}\bar\xi^N(x)\bar\xi^N(y)\alpha^N(z)\alpha^N(w)\right]\label{larger2}   \\&+\expected_{\mu_N^*}\left[\boldsymbol{1}_{\Omega_N^1}\sum_{x\in\mathbb{T}_N}\bar\xi^N(x)^2\left(\sum_{i=x+1}^{x+\ell_0}\sum_{z=x+1}^{i-1}\alpha^N(z)\right)^2\right].\label{eq:smaller2}
\end{align}
This time, on $\Omega_N^1$ we have instead
\begin{equation*}
    \begin{split}\alpha^N(z)&=2\left[\frac{1}{2}+\bar\xi^N(z)\right]\left[\frac{1}{2}-\bar\xi^N(z+1)\right]
    \\&\phantom{=}+\left(1-e^{-2N^{-\gamma}}\right)\left[\frac{1}{2}-\bar\xi^N(z)\right]\left[\frac{1}{2}+\bar\xi^N(z+1)\right].
    \end{split}
\end{equation*}
Using the trivial estimate $|\alpha^N(z)|\le C$ and applying Theorem~\ref{thm:restriction1} with $m=2$, it is straightforward to see that 
\begin{equation*}
    \eqref{larger2}\le\frac{C\ell_0^4}{N^{\gamma-1}}
\end{equation*}
By the same estimate $|\alpha^N(z)|\le5$ and by Theorem~\ref{thm:restriction1} with $m=0$, we also get that
\begin{equation*}
    \eqref{eq:smaller2}\le \frac{C\ell_0^4}{N^{\gamma-1}}.
\end{equation*}
The estimate on $\Omega_N^{-1}$ is identical. Combining all the bounds for the four terms, we get that
\begin{equation*}
    \eqref{eq:six_sum}\le Ct^2\|\nabla\phi\|_\infty^2\left\{\frac{\ell_0^2}{N^{4\gamma-2}}+\frac{\ell_0^2}{N^{3\gamma-2}}\right\}.
\end{equation*}
Finally, combining everything together yields
\begin{equation*}
    \expected_{\mu_N^*}\left[\mathscr{B}_t^N(\phi)^2\right]\le Ct\|\nabla\phi\|_\infty^2\left\{t\left(\frac{N^{2-2\gamma}}{\ell_0}+N^{1-2\gamma}+\frac{\ell_0^2}{N^{4\gamma-2}}+\frac{\ell_0^2}{N^{3\gamma-2}}\right)+\frac{\ell_0^2}{N^{2\gamma}}\right\}
\end{equation*}
for any $1\le \ell_0\le N-1$. But then, for $\gamma>\frac{6}{7}$, choosing $\ell_0=N^\alpha$ for some $\alpha\in\left(2-2\gamma, \frac{3}{2}\gamma-1\right)$ and sending $N\to\infty$ concludes the proof.
\end{proof}

%%%%%%%%%%%%%%%%%%%%%%%%%%%%%%%%%%%%%%%%%%%%%%%%%%
%%%The Martingale Part%%%%%%%%%%%%%%%%%%%%%%%%%%%%
%%%%%%%%%%%%%%%%%%%%%%%%%%%%%%%%%%%%%%%%%%%%%%%%%%
\subsection{The Martingale Part}
In this section we show convergence of the sequence of $\mathcal{S}'(\mathbb{T})$-valued processes $\{\mathscr{M}_t^N, t\in[0, T]\}_N$.
\begin{proposition}\label{prop:martingale_conv}
For any $\phi\in\mathcal{S}(\mathbb{T})$, the sequence of real-valued martingales $\{\mathscr{M}_t^N(\phi), t\in[0, T]\}_N$ converges in law to $\frac{1}{2}\|\nabla\phi\|_{L^2(\mathbb{T})}\mathscr{W}_t(\phi)$, where $\{\mathscr{W}_t(\phi), t\in[0, T]\}$ is a standard one-dimensional Brownian motion.
\end{proposition}
To show the proposition above, we need some preliminary results. The following theorem is a straightforward corollary of \cite[Theorem~VIII.3.11]{js03}. 

\begin{theorem}\label{thm:jacod} Let $\{M^N_t, t\in[0, T]\}_N$ be a sequence of real-valued càdlàg martingales and let $\scal{M^N}_t$ denote the quadratic variation of $M^N_t$. Let $f:[0, T]\to[0, \infty)$ be a deterministic, continuous function. Assume that:
\begin{enumerate}[i)]
    \item there exists a constant $K$ such that, for each $N$ and each $s\in[0, T]$, $\left|M^N_s-M^N_{s^-}\right|\le K$ almost surely,
    
    \item $\mylim_{N\to\infty}\expected\left[\mysup_{0\le s\le T}\left|M^N_s-M^N_{s^-}\right|\right]=0$,
    
    \item for any $t\in[0, T]$, the sequence of random variables $\{\scal{M^N}_t\}_N$ converges in probability to $f(t)$.
\end{enumerate}
Then, the sequence $\{M^N_t, t\in[0, T]\}_N$ converges in law in $D([0, T], \mathbb{R})$ to a mean-zero Gaussian martingale on $[0, T]$ with continuous trajectories and with quadratic variation given by $f$.
\end{theorem}

We start by checking that, for each $\phi\in\mathcal{S}(\mathbb{T})$, the sequence $\{\mathscr{M}_t^N(\phi), t\in[0, T]\}_N$ satisfies items i) and ii) of the theorem above.

\begin{lemma}\label{lemma:mart_1}
For any $\phi\in\mathcal{S}(\mathbb{T})$, 
\begin{equation}\label{eq:jump_bound}
    \left|\mathscr{M}^N_s(\phi)-\mathscr{M}^N_{s-}(\phi)\right|\le \|\nabla \phi\|_\infty
\end{equation}
almost surely, and 
\begin{equation}\label{eq:jump_limit}
    \mylim_{N\to\infty} \expected_{\mu_N^*} \left[\mysup_{0\le s\le T}\left|\mathscr{M}^N_s(\phi)-\mathscr{M}^N_{s-}(\phi)\right|\right]=0.
\end{equation}
\end{lemma}

\begin{proof}
From \eqref{eq:dynkin_2}, it is easy to see that, for any $\phi$,
\begin{equation*}
    \left|\mathscr{M}^N_s(\phi)-\mathscr{M}^N_{s-}(\phi)\right|=\left|\mathscr{U}^N_s(\phi)-\mathscr{U}^N_{s-}(\phi)\right|.
\end{equation*}
Now assume that a flip happens at time $s$ and position $y$: then
\begin{equation*}
    \left|\mathscr{U}^N_s(\phi)-\mathscr{U}^N_{s-}(\phi)\right|=\frac{1}{\sqrt{N}}\left|\phi\left(\frac{y}{N}\right)-\phi\left(\frac{y+1}{N}\right)\right|\le\frac{1}{N^{\frac{3}{2}}}\|\nabla\phi\|_\infty,
\end{equation*}
which immediately implies \eqref{eq:jump_bound} and \eqref{eq:jump_limit}.
\end{proof}

In the two lemmas that follow, we move on to verifying item iii) by proving a stronger convergence result for the sequence $\{\scal{\mathscr{M}^N(\phi)}_t, t\in[0, T]\}_N$.

\begin{lemma}\label{lemma:mart_2}
For any $\phi\in\mathcal{S}(\mathbb{T})$ and any $t\in[0, T]$, the quadratic variation \eqref{eq:quad_var_2} of $\mathscr{M}^N_t(\phi)$ satisfies
\begin{equation*}
    \mylim_{N\to\infty} \expected_{\mu_N^*}\left[\scal{\mathscr{M}^N(\phi)}_t\right]=\frac{t}{4}\|\nabla\phi\|_{L^2(\mathbb{T})}^2.
\end{equation*}
\end{lemma}

\begin{proof}
By \eqref{eq:quad_var_2b} and by the stationarity of $\mu_N^*$, 
\begin{align*}
    \expected_{\mu_N^*}\left[\scal{\mathscr{M}^N(\phi)}_t\right] &=\expected_{\mu_N^*}\left[\int_0^t\frac{1}{N}\sum_{x\in\mathbb{T}_N}\nabla_N\phi\left(\frac{x}{N}\right)^2\eta_s^N(x)\de s\right]
\end{align*}
with $\eta_s^N(x)$ defined in \eqref{eq:eta_2}.
By Theorem~\ref{thm:2m_correlations}, $\left|\expected_{\mu_N^*}\left[\eta^N(x)\right]-\frac{1}{4}\right|\le\frac{C}{N}$
for some constant $C$ independent of $N$, which yields
\begin{equation*}
    \mylim_{N\to\infty} \expected_{\mu_N^*}\left[\scal{\mathscr{M}^N(\phi)}_t\right]=\mylim_{N\to\infty}\frac{t}{4N}\sum_{x\in\mathbb{T}_N}\nabla_N\phi\left(\frac{x}{N}\right)^2,
\end{equation*}
whence the claim follows.
\end{proof}

\begin{lemma}\label{lemma:mart_3}
For any $\phi\in\mathcal{S}(\mathbb{T})$ and any $t\in[0, T]$, 
\begin{equation*}
    \mylim_{N\to\infty} \expected_{\mu_N^*} \left[\left(\scal{\mathscr{M}^N(\phi)}_t-\expected_{\mu_N^*}\left[\scal{\mathscr{M}^N(\phi)}_t\right]\right)^2\right]=0.
\end{equation*}
\end{lemma}

\begin{proof}
Note that by \eqref{eq:quad_var_2b}, the stationarity of $\mu_N^*$, and the Cauchy-Schwarz inequality,
\begin{align*}
    &\expected_{\mu_N^*} \left[\left(\scal{\mathscr{M}^N(\phi)}_t-\expected_{\mu_N^*}\left[\scal{\mathscr{M}^N(\phi)}_t\right]\right)^2\right]
    \\&=\expected_{\mu_N^*}\left[\left(\int_0^t\frac{1}{N}\sum_{x\in\mathbb{T}_N}\nabla_N\phi\left(\frac{x}{N}\right)^2\left\{\eta_s^N(x)-\expected_{\mu_N^*} \left[\eta_s^N(x)\right]\right\}\de s\right)^2\right]
    \\\begin{split}&\le \frac{t^2}{N^2}\sum_{x, y\in\mathbb{T}_N}\nabla_N\phi\left(\frac{x}{N}\right)^2\nabla_N\phi\left(\frac{y}{N}\right)^2\times\\&\phantom{\le}\times\expected_{\mu_N^*}\left[\left\{\eta^N(x)-\expected_{\mu_N^*}\left[\eta^N(x)\right]\right\}\left\{\eta^N(y)-\expected_{\mu_N^*}\left[\eta^N(y)\right]\right\}\right],
    \end{split}
    \end{align*}
where $\eta_s^N(x)$ is defined in \eqref{eq:eta_2}.
For $|x-y|\le 1$, since
\begin{equation*}
    \bigl|\expected_{\mu_N^*}[\{\eta^N(x)-\expected_{\mu_N^*}\left[\eta^N(x)]\}\left\{\eta^N(y)-\expected_{\mu_N^*}\left[\eta^N(y)\right]\right\}\right]\bigr| \le C\;,
\end{equation*}
this contribution vanishes in the limit. For $|x-y|>1$, Theorem~\ref{thm:2m_correlations} yields
\begin{align*}
    \left|\expected_{\mu_N^*}\left[\left\{\eta^N(x)-\expected_{\mu_N^*}\left[\eta^N(x)\right]\right\}\left\{\eta^N(y)-\expected_{\mu_N^*}\left[\eta^N(y)\right]\right\}\right]\right|\le\frac{C}{N}
\end{align*}
for some constant $C$ independent of $N$, so this contribution also vanishes in the limit.
\end{proof}

\begin{proof}[Proof of Proposition~\ref{prop:martingale_conv}] Consider the sequence of martingales $\{\mathscr{M}^N_t(\phi), t\in[0, T]\}_N$: the assumptions of Theorem~\ref{thm:jacod} follow directly from Lemmas~\ref{lemma:mart_1}, \ref{lemma:mart_2} and \ref{lemma:mart_3} above with $f(t)=\frac{t}{4}\|\nabla\phi\|^2_{L^2(\mathbb{T})}$. Lévy's characterisation of Brownian motion concludes the proof.
\end{proof}

%%%%%%%%%%%%%%%%%%%%%%%%%%%%%%%%%%%%%%%%%%%%%%%%%%
%%%Characterisation of the Limit Point%%%%%%%%%%%%
%%%%%%%%%%%%%%%%%%%%%%%%%%%%%%%%%%%%%%%%%%%%%%%%%%
\subsection{Characterisation of the Limit Point}
To conclude the proof of Theorem~\ref{thm:flucts}, we need a more precise result about the $2$-point correlations as given by the following theorem, whose proof is 
postponed to Appendix~\ref{sec:app_correlations}:

\begin{theorem}[Limit of the $2$-point correlations]\label{thm:2correlation_limit}
Let $N=2p$ with $p$ prime. For any $x_1\ne x_2$ in $\mathbb{T}_N$,
\begin{equation*}
    \mylim_{N\to\infty}N\hspace{.1em}\expected_{\mu_N^*}\left[\bar\xi^N(x_1)\bar\xi^N(x_2)\right]=-\frac{1}{4}.
\end{equation*}
\end{theorem}

\begin{proof}[Proof of Theorem~\ref{thm:flucts}]
As seen in Proposition~\ref{prop:tightness_2}, all four terms on the RHS of \eqref{eq:fluctuations} are tight, so consider a subsequence for which they all converge, which we still denote by $N$. Let $\mathscr{U}_t, \mathscr{M}_t, \mathscr{K}_t$ and $\mathscr{B}_t$ be limit points of $\mathscr{U}_t^N, \mathscr{M}_t^N, \mathscr{K}_t^N$ and $\mathscr{B}_t^N$ respectively. Note that \eqref{eq:mathcal_K} can be written as
\begin{equation*}
    \mathscr{K}_t^N(\phi)=\frac{1}{2}\int_0^t \mathscr{U}_s^N(\Delta_N\phi)\de s,
\end{equation*}
from which it is easy to see that the limit points $\{\mathscr{K}_t, t\in[0, T]\}$ and $\{\mathscr{U}_t, t\in[0, T]\}$ must be related via
\begin{equation*}
    \mathscr{K}_t(\phi)=\frac{1}{2}\int_0^t\mathscr{U}_s(\Delta\phi)\de s.
\end{equation*}
Also, Proposition~\ref{prop:nonlinear_term} implies that $\mathscr{B}_t=0$, and hence we get the equation
\begin{equation*}
    \mathscr{U}_t(\phi)=\mathscr{U}_0(\phi)+\frac{1}{2}\int_0^t\mathscr{U}_s(\Delta\phi)\de s+\mathscr{M}_t(\phi)
\end{equation*}
where, by Proposition~\ref{prop:martingale_conv}, $\{\mathscr{M}_t(\phi), t\in[0, T]\}$ is a continuous martingale of quadratic variation $\frac{1}{4} t\|\nabla\phi\|^2_{L^2(\mathbb{T}_N)}$. Hence, $\{\mathscr{U}_t, t\in[0, T]\}$ is a stationary solution of \eqref{eq:OU} with $\lambda=\sigma=\frac{1}{2}$. 

As for its variance, note that for any $\phi\in\mathcal{S}(\mathbb{T})$,
\begin{align*}
    &\expected_{\mu_N^*}\left[\mathscr{U}_t^N(\phi)^2\right]
    \\&=\expected_{\mu_N^*}\left[\left(\frac{1}{\sqrt{N}}\sum_{x\in\mathbb{T}_N}\bar\xi_t^N(x)\phi\left(\frac{x}{N}\right)\right)^2\right]
    \\&=\frac{1}{N}\sum_{x\in\mathbb{T}_N}\expected_{\mu_N^*}\left[\bar\xi_t^N(x)^2\right]\phi\left(\frac{x}{N}\right)^2+\frac{1}{N}\sum_{\substack{x, y\in\mathbb{T}_N\\x\ne y}}\expected_{\mu_N^*}\left[\bar\xi_t^N(x)\bar\xi_t^N(y)\right]\phi\left(\frac{x}{N}\right)\phi\left(\frac{y}{N}\right).
\end{align*}
Since $\expected_{\mu_N^*}[\bar\xi^N_t(x)^2]=\frac{1}{4}$, the first term clearly converges to $\frac{1}{4}\|\phi\|_{L^2(\mathbb{T})}^2$. As for the second term, by Theorem~\ref{thm:2correlation_limit} we have that
\begin{align*}
    &\mylim_{N\to\infty}\frac{1}{N}\sum_{\substack{x, y\in\mathbb{T}_N\\x\ne y}}\expected_{\mu_N^*}\left[\bar\xi_t^N(x)\bar\xi_t^N(y)\right]\phi\left(\frac{x}{N}\right)\phi\left(\frac{y}{N}\right)
    \\&=\mylim_{N\to\infty}-\frac{1}{4N^2}\sum_{\substack{x, y\in\mathbb{T}_N\\x\ne y}}\phi\left(\frac{x}{N}\right)\phi\left(\frac{y}{N}\right)
    \\&=\mylim_{N\to\infty}-\frac{1}{4}\left(\frac{1}{N}\sum_{x\in\mathbb{T}_N}\phi\left(\frac{x}{N}\right)\right)^2,
\end{align*}
and thus
\begin{equation*}
    \mylim_{N\to\infty}\expected_{\mu_N^*}\left[\mathscr{U}_t^N(\phi)^2\right]=\frac{1}{4}\left(\|\phi\|_{L^2(\mathbb{T})}^2-\left\langle1, \phi\right\rangle^2\right),
\end{equation*}
which completes the proof. 
\end{proof}

\appendix

%%%%%%%%%%%%%%%%%%%%%%%%%%%%%%%%%%%%%%%%%%%%%%%%%%
%ENERGY ESTIMATE AND UNIQUENESS OF WEAK SOLUTIONS%
%%%%%%%%%%%%%%%%%%%%%%%%%%%%%%%%%%%%%%%%%%%%%%%%%%
\section{Energy Estimate and Uniqueness of Weak Solutions}\label{sec:app_energy_uniqueness}

In this appendix we complete the proof of the characterisation of the hydrodynamic limit by showing that any limit point $Q$ of $\{Q^N\}_N$ must be concentrated on trajectories whose density $\rho$ belongs to the Sobolev space $L^2([0, T], \mathcal{H}^1)$, and that the weak solutions of \eqref{eq:coupled_equations} are unique.

%%%%%%%%%%%%%%%%%%%%%%%%%%%%%%%%%%%%%%%%%%%%%%%%%%
%%%Energy Estimate%%%%%%%%%%%%%%%%%%%%%%%%%%%%%%%%
%%%%%%%%%%%%%%%%%%%%%%%%%%%%%%%%%%%%%%%%%%%%%%%%%%
\subsection{Energy Estimate}
In order to show that $\rho$ belongs to $L^2([0, T], \mathcal{H}^1)$, it suffices to prove the following proposition, which is usually called an \textit{energy estimate}. The result then follows from Riesz's Representation Theorem; for a full argument, see for example \cite[Section~5.2]{fgn13} or \cite[Section~4.3]{fgn15}.

Let $C_k^{m,n}([0, T]\times\mathbb{T})$ be the set of continuous functions on $[0, T]\times\mathbb{T}$ with compact support which are $m$ times differentiable in the first variable and $n$ in the second, with continuous derivatives, where $m$ and $n$ are positive integers.

\begin{proposition} Given any limit point $Q$ of $\{Q^N\}_N$, there exists a positive constant $\kappa>0$ such that
\begin{equation*}
    \expected_{Q}\left[\mysup_H \left\{\int_0^T\int_{\mathbb{T}} \partial_u H(s, u)\rho_s(u)\de u\de s-\kappa\int_0^T\int_{\mathbb{T}} H(s, u)^2\de u\de s\right\}\right]<\infty,
\end{equation*}
where the supremum is along all functions $H\in C_k^{0, 2}([0, T]\times\mathbb{T})$.
\end{proposition}

\begin{proof} To ease the notation, we drop the superscript $N$, and we denote by $C$ a generic finite positive constant changing from time to time or even within the same line. Following the same argument given in the proofs of the replacement lemma (Lemma~\ref{lemma:replacement}) and \cite[Lemma~6.10]{gmo23}, it suffices to show that there exists a positive constant $\kappa>0$ such that
\begin{equation}\label{eq:expectation_1}
    \mylimsup_{N\to\infty}\frac{1}{N}\mylog\expected_{\mu_N^*}\left[e^{\int_0^T  \left\{\sum_{x\in\mathbb{T}_N} \xi_s^N(x)\partial_u H\left(s, \frac{x}{N}\right)-\kappa N\|H(s, \cdot)\|^2_{L^2(\mathbb{T})}\right\}\de s}\right]<CT.
\end{equation}
By the Feynman-Kac formula, the expression in the limit in \eqref{eq:expectation_1} is bounded from above by
\begin{equation}\label{eq:feynman_term_2}
    \mysup_f \left\{\int_0^T\left[\frac{1}{N}\sum_{x\in\mathbb{T}_N} \partial_u H\left(s, \frac{x}{N}\right)\left\langle \xi(x), f\right\rangle_{\mu_N^*}-\kappa \|H(s, \cdot)\|^2_{L^2}-N\mathscr{D}_N\big(\sqrt{f}, \mu_N^*\big)\right]\de s\right\},
\end{equation}
where the supremum runs along all probability densities $f$ with respect to $\mu_N^*$. Note that $\partial_u H\big(s, \frac{x}{N}\big)=N\left[H\big(s, \frac{x+1}{N}\big)-H\big(s, \frac{x}{N}\big)\right]+O\big(\frac{1}{N}\big)$, so we can perform a summation by parts and rewrite 
\begin{align}
    \frac{1}{N}&\sum_{x\in\mathbb{T}_N} \partial_u H\left(s, \frac{x}{N}\right)\left\langle \xi(x), f\right\rangle_{\mu_N^*}=\int_{\Omega_N}\sum_{x\in\mathbb{T}_N}H\left(s, \frac{x}{N}\right)\left[\xi(x)-\xi(x+1)\right]f(h)\de \mu_N^*\nonumber
    \\&=\frac{1}{2}\int_{\Omega_N}\sum_{x\in\mathbb{T}_N}H\left(s, \frac{x}{N}\right)\left[\xi(x)-\xi(x+1)\right]\left[f(h)+f\left(h^{x, x+1}\right)\right]\de \mu_N^*\label{eq:H1}
    \\&\phantom{=}+\frac{1}{2}\int_{\Omega_N}\sum_{x\in\mathbb{T}_N}H\left(s, \frac{x}{N}\right)\left[\xi(x)-\xi(x+1)\right]\left[f(h)-f\left(h^{x, x+1}\right)\right]\de \mu_N^*\label{eq:H2}
\end{align}
up to an error of order at most $\frac{1}{N}$. If we multiply and divide by $q^N_{x, x+1}(h)$, defined in \eqref{eq:q^N}, then by Young's inequality, for any $A>0$ we get the bound 
\begin{align}
    \eqref{eq:H2}&\le\frac{A}{4}\int_{\Omega_N}\sum_{x\in\mathbb{T}_N} q^N_{x, x+1}(h)\left[\sqrt{f(h)}-\sqrt{f\left(h^{x, x+1}\right)}\right]^2\de \mu_N^*\label{eq:H5}
    \\&\phantom{\le}+\frac{1}{4A}\int_{\Omega_N}\sum_{x\in\mathbb{T}_N}H\left(s, \frac{x}{N}\right)^2 \frac{\left[\sqrt{f(h)}-\sqrt{f\left(h^{x, x+1}\right)}\right]^2}{q^N_{x, x+1}(h)}\de \mu_N^*.\label{eq:H6}
\end{align}
Choosing $A=2N$ and using \eqref{eq:dir=gamma}, we obtain that
\begin{equation*}
    \eqref{eq:H5}\le N\mathscr{D}_N\big(\sqrt{f}, \mu_N^*\big)
\end{equation*}
and
\begin{equation*}
    \eqref{eq:H6}\le \frac{C}{N}\sum_{x\in\mathbb{T}_N}H\left(s, \frac{x}{N}\right)^2.
\end{equation*}
As for \eqref{eq:H1}, by performing the change of variable $h\mapsto h^{x, x+1}$, we see that
\begin{equation*}
    \eqref{eq:H1}=\frac{1}{2}\int_{\Omega_N}\sum_{x\in\mathbb{T}_N}H\left(s, \frac{x}{N}\right)\left[\xi(x)-\xi(x+1)\right]f\left(h^{x, x+1}\right)\left[1-\frac{\mu_N^*\left(h^{x, x+1}\right)}{\mu_N^*(h)}\right]\de \mu_N^*.
\end{equation*}
Then, recalling that $\mu_N^*$ is invariant for nearest neighbour exchanges up to a correction of order $\frac{1}{N}$ (see \eqref{eq:almost_invariant}), by Young's inequality, for any $B>0$ we get the bound
\begin{align*}
    \eqref{eq:H1}&\le \frac{B}{4}\int_{\Omega_N}\sum_{x\in\mathbb{T}_N}H\left(s, \frac{x}{N}\right)^2 \de \mu_N^*
    \\&\phantom{\le}+\frac{1}{4B}\int_{\Omega_N}\sum_{x\in\mathbb{T}_N}f\left(h^{x, x+1}\right)^2\left[1-\frac{\mu_N^*\left(h^{x, x+1}\right)}{\mu_N^*(h)}\right]^2\de \mu_N^*.
\end{align*}
Choosing $B=\frac{1}{4N}$ then yields
\begin{equation*}
    \eqref{eq:H1}\le \frac{C}{N}\sum_{x\in\mathbb{T}_N}H\left(s, \frac{x}{N}\right)^2+C.
\end{equation*}
Combining everything together, we finally get 
\begin{equation*}
    \eqref{eq:feynman_term_2}\le \int_0^T \left\{\frac{C}{N}\sum_{x\in\mathbb{T_N}} H\left(s, \frac{x}{N}\right)^2+C+\frac{C}{N}-\kappa \int_{\mathbb{T}}H(s, u)^2\de u\right\} \de s.
\end{equation*}
The conclusion follows from noting that $\frac{1}{N}\sum_{x\in\mathbb{T}_N}H(s, \frac{x}{N})^2$ converges to $\int_\mathbb{T} H(s, u)^2\de u$.
\end{proof}

%%%%%%%%%%%%%%%%%%%%%%%%%%%%%%%%%%%%%%%%%%%%%%%%%%
%%%Uniqueness of Weak Solutions%%%%%%%%%%%%%%%%%%%
%%%%%%%%%%%%%%%%%%%%%%%%%%%%%%%%%%%%%%%%%%%%%%%%%%
\subsection{Uniqueness of Weak Solutions}
\begin{proof}[Proof of Lemma~\ref{lemma:uniqueness_pde}]
First, let $Z:[0, T]\to[-1, 1]$ and assume that $\rho:[0, T]\times\mathbb{T}\to[0,1]$ is fixed and belongs to $L^2([0, T], \mathcal{H}^1)$. It is then immediate that the equation
$\partial_tY_t =-2Z_t\langle\rho_t, 1-\rho_t\rangle$
admits a unique weak solution with initial condition $Y_0$ in the sense of Definition~\ref{def:weak_solutions}. We now show that, if $Z:[0, T]\to[-1, 1]$ is fixed, the PDE
\begin{equation}\label{eq:Z_pde}
    \begin{cases}
    \partial_t\rho_t =\frac{1}{2}\Delta \rho_t-Z_t\nabla[\rho_t(1-\rho_t)],
    \\ \rho(0, \cdot)=\rho_0,
    \end{cases}
\end{equation}
also has a unique weak solution in the sense of Definition~\ref{def:weak_solutions} (in particular, satisfying $\rho\in L^2([0, T], \mathcal{H}^1)$). Let $\rho, \tilde\rho$ be two weak solutions of \eqref{eq:Z_pde} and let $\bar\rho:=\rho-\tilde\rho$. Note that we can write
\begin{equation*}
    Z_t\left[\rho_t(1-\rho_t)-\tilde{\rho}_t(1-\tilde{\rho}_t)\right]=Z_t\bar\rho_t(1-\rho_t-\tilde{\rho}_t),
\end{equation*}
and thus, since $\bar\rho(0, \cdot)\equiv 0$, calling $b_t:=Z_t(1-\rho_t-\tilde{\rho}_t)$, we have that $\bar\rho$ satisfies
\begin{equation}\label{eq:weak_1}
    \int_0^T\left\langle \bar\rho_t, \partial_t\phi_t+\frac{1}{2}\Delta \phi_t+b_t \nabla\phi_t\right\rangle \de t=0
\end{equation}
for each test function $\phi\in C_0^{1, 2}([0, T]\times\mathbb{T})$. Since $\rho$ and $\tilde\rho$ are in $L^2([0, T], \mathcal{H}^1)$ and $Z$ is in $L^\infty[0, T]$, we see that $b$ belongs to $L^2([0, T], \mathcal{H}^1)$, so we can find smooth functions $\{b^\eps\}_{\eps>0}$ such that $b^\eps\to b$ in $L^2([0, T], \mathcal{H}^1)$ as $\eps\to0$; see, for example, \cite[Theorem~5.3.2.2]{eva10}. Hence, we have that
\begin{equation}\label{eq:lim_eps_b}
    \|b^\eps- b\|^2_{L^2([0, T]\times\mathbb{T})}+ \|\nabla b^\eps- \nabla b\|^2_{L^2([0, T]\times\mathbb{T})}\to 0\ \text{as} \ \eps\to0.
\end{equation}
For each $\eps>0$, given a smooth function $\psi\in C_0^\infty([0, T]\times\mathbb{T})$, consider the PDE
\begin{equation}\label{eq:Phi_pde}
    \begin{cases}
    \partial_t\phi_t+\frac{1}{2}\Delta\phi_t+b^\eps_t \nabla\phi_t = \psi_t \ \text{on}\ [0, T)\times\mathbb{T},
    \\ \phi(T, \cdot) = 0.
    \end{cases}
\end{equation}
Then, its (unique) solution $\phi^\eps$ is in $C_0^{1, 2}([0, T]\times\mathbb{T})$, so we can use it as a test function in \eqref{eq:weak_1} and get that, for each $\eps>0$,
\begin{equation}\label{eq:weak_2}
    \int_0^T\langle \bar\rho_t, \psi_t\rangle \de t+\int_0^T\langle \bar\rho_t, (b_t-b_t^\eps)\nabla\phi^\eps_t\rangle\de t=0.
\end{equation}
Now, from the Cauchy-Schwarz inequality we have that
\begin{align}
    \left|\int_0^T\langle \bar\rho_t, (b_t-b_t^\eps)\nabla\phi^\eps_t\rangle\de t\right|&\le\|\bar\rho\|_{L^\infty([0, T]\times\mathbb{T})}\int_0^T \|b_t-b_t^\eps\|_{L^2(\mathbb{T})}\|\nabla\phi_t^\eps\|_{L^2(\mathbb{T})}\de t\nonumber
    \\&\le \|\bar\rho\|_{L^\infty} \|b-b^\eps\|_{L^2([0, T]\times\mathbb{T})}\|\nabla\phi^\eps\|_{L^2([0, T]\times\mathbb{T})}.\label{eq:bound_1}
\end{align}
From \eqref{eq:lim_eps_b}, we have that both $\|b^\eps\|_{L^2([0, T]\times\mathbb{T})}$ and $\|\nabla b^\eps\|_{L^2([0, T]\times\mathbb{T})}$ are uniformly bounded in $\eps$, and hence, since $\phi^\eps$ solves \eqref{eq:Phi_pde}, by a standard argument we can also get a uniform bound in $\eps$ for $\|\nabla\phi^\eps\|_{L^2([0, T]\times\mathbb{T})}$, namely there exists a constant $C<\infty$ independent of $\eps$ such that
\begin{equation}\label{eq:bound_2}
    \|\nabla\phi^\eps\|_{L^2([0, T]\times\mathbb{T})}\le C.
\end{equation}
From \eqref{eq:weak_2}, \eqref{eq:bound_1}, \eqref{eq:lim_eps_b}, \eqref{eq:bound_2} and the Dominated Convergence Theorem, we get that for each $\psi\in C_0^\infty([0, T]\times\mathbb{T})$ we have that $\int_0^T\langle \bar\rho_t, \psi_t\rangle \de t=0$, which yields $\bar\rho=0$ almost everywhere on $[0, T]\times\mathbb{T}.$

Finally, assume that $(\rho, Y)$ and $(\tilde \rho, \tilde Y)$ are both weak solutions of \eqref{eq:coupled_equations}, and call $Z:=\sgn(Y)$, $\tilde Z:=\sgn(\tilde Y)$, $\tau:=\myinf\{t\in[0, T]: Z_t=0\}$ and $\tilde\tau:=\myinf\{t\in[0, T]: \tilde Z_t=0\}$. Note that, until $\tau_{\mymin}:=\tau \wedge \tilde\tau$, we have that $Z=\tilde Z$, so the solutions $(\rho, Y)$ and $(\tilde\rho, \tilde Y)$ must coincide by our above arguments. But then, $\tau$ and $\tilde\tau$ must also coincide, since they are both equal to $\myinf\{t\in[0, T]: \int_0^t Z_s\langle \rho_s(1-\rho_s), 1\rangle\de s=Y_0\}$. Now, since both $Y$ and $\tilde{Y}$ solve an equation of the form $\partial_t Y_t=-\sgn(Y_t)f(t)$ for some positive $f$, it is easy to see that after they hit zero they remain equal to zero indefinitely; in particular,  after $\tau_{\mymin}$, $Z=\tilde{Z}=0$, so that $\rho$ and $\tilde\rho$ must again coincide by the above argument. This completes the proof.
\end{proof}

%%%%%%%%%%%%%%%%%%%%%%%%%%%%%%%%%%%%%%%%%%%%%%%%%%
%%%ASYMPTOTICS OF THE CORRELATION FUNCTIONS%%%%%%%
%%%%%%%%%%%%%%%%%%%%%%%%%%%%%%%%%%%%%%%%%%%%%%%%%%
\section{Asymptotics of the Correlation Functions}\label{sec:app_correlations}
In this appendix we show Theorems~\ref{thm:2correlation_limit}, \ref{thm:2m_correlations}, \ref{thm:restriction>1} and~\ref{thm:restriction1}. For the rest of the section we let $N=2p$ with $p$ odd prime and we denote by $\expected_{\mu_N^*}[\,\cdot\,]$ the expected value with respect to the measure $\mu_N^*=\mu_{N, \gamma}^*$ on $\Omega_N$.

We start by showing an asymptotic result for the partition function $\mathcal{Z}_{N, \gamma}$,  defined in Lemma~\ref{lemma:invariant_measure}. In order to do so, we will use the following preliminary result, whose proof contains the main ideas necessary to show all the results proved in this section. Let $A=A(N)=\{\xi\in\{-1, 1\}^N: \sum_{i=1}^N \xi_i=0\}$. For $k=1, 3,\ldots, N-1$, let 
\begin{equation*}
    A_k=A_k(N):=\left\{\xi\in A: \sum_{i=1}^N i\xi_i\equiv k\pmod N\right\}
\end{equation*}
and let $\alpha_k=\text{card}(A_k)$.
\begin{theorem}\label{thm:cardinality}
For any $k=1, 3, \ldots, 2p-1$,
\begin{equation*}
    \left|\alpha_k-\frac{1}{p}\binom{2p}{p}\right|\le2.
\end{equation*}
\end{theorem}

\begin{proof} Given a configuration $\xi$ in $A$, let $j_1, \ldots, j_p$ be the indices of the $-1$'s, namely $\{j_1, \ldots, j_p\}$ is the subset of $\{1, \ldots, 2p\}$ such that $\xi_{j_1}=\ldots =\xi_{j_p}=-1$. Then,
\begin{align*}
    \sum_{i=1}^{2p} i\xi_i\equiv k\pmod{2p}&\iff \sum_{i=1}^{2p}i-2\sum_{i=1}^{p}
    j_i\equiv k\pmod {2p}
    \\&\iff 2\sum_{i=1}^{p} j_i\equiv p-k\pmod{2p}
    \\&\iff\sum_{i=1}^p j_i\equiv\frac{p-k}{2}\pmod{p}.
\end{align*}
Hence, computing the $\alpha_k$'s is equivalent to calculating how many $p$-element subsets of $\{1, 2, \ldots, 2p\}$ have a sum congruent to $m\pmod p$ for each $m=0, 1, \ldots, p-1$. Let $\omega$ be a primitive $p$-th root of unity: then note that we can write
\begin{equation*}
    (x^p-1)^2=(x-\omega)(x-\omega^2)\cdots (x-\omega^{2p}),
\end{equation*}
but on the other hand we can also write 
\begin{equation*}
    (x^p-1)^2=x^{2p}-2x^p+1.
\end{equation*}
Hence, the coefficient of $x^p$ on the RHS of the two expressions above must be the same, so we get 
\begin{equation*}
    -\sum_{1\le i_1<\ldots <i_p\le 2p}\omega^{i_1}\cdots \omega^{i_p} = -2.
\end{equation*}
Since $\omega^p=1$, we can read the powers of $\omega$ modulo $p$ on the LHS of the equality above, and write it as
\begin{equation}\label{eq:eq_1}
    \sum_{k=0}^{p-1} a_k \omega^k=2,
\end{equation}
where $a_k$ is the number of $p$-element subsets of $\{1, \ldots, 2p\}$ whose elements have a sum congruent to $k\pmod p$; in particular, the $a_k$'s are a permutation of the $\alpha_k$'s. But now, a linear combination of $1, \omega, \ldots, \omega^{p-1}$ can only be zero if all the coefficients are the same, so \eqref{eq:eq_1} implies that
\begin{equation*}
    a_0-2=a_1=\ldots=a_{p-1}.
\end{equation*}
Since $\sum_{k=0}^{p-1}\alpha_k=\binom{2p}{p}$, the conclusion easily follows.
\end{proof}

\begin{corollary}\label{corol:norm_const} We have that
\begin{equation*}
    \mylim_{N\to\infty} N^{1-\gamma}\binom{N}{\frac{N}{2}}^{-1}\mathcal{Z}_{N, \gamma}=1.
\end{equation*}
 In particular, for $N$ sufficiently large, $\frac{1}{\mathcal{Z}_{N, \gamma}}\le 2N^{1-\gamma}\binom{N}{\frac{N}{2}}^{-1}$.
\end{corollary}

\begin{proof} 
The following remark will be used repeatedly. For $k\in\mathbb{Z}$ odd, let $\tilde{A}_k=\{h\in\Omega_N: Y(h)=k\}$ and let $\tilde{\alpha_k}=\text{card}(\tilde{A}_k)$. By shifting a height configuration upwards by one, its integral increases by $N$, and hence, if $k\equiv j\pmod N$, then there is a bijection between $\tilde{A}_k$ and $\tilde{A}_j$. 
On the other hand, for $k=1, 3, \ldots N-1$, there is a bijection between $\tilde{A}_k$ and $A_k$: this can be seen by mapping $\xi\in A_k$ with $\sum i\xi_i=aN+k$ to the height configuration $h\in\tilde{A}_k$ starting at $-a$ and with slopes given by $\xi$. This implies that, for any $m\in\mathbb{Z}$ and $k=1, 3, \ldots, N-1$, we have that
\begin{equ}\label{eq:cards_equal}
    \tilde{\alpha}_{mN+k}=\alpha_k.
\end{equ}
But then, by Theorem~\ref{thm:cardinality} we get that
\begin{align*}
    \mylim_{N\to\infty} N^{1-\gamma}&\binom{N}{\frac{N}{2}}^{-1}\mathcal{Z}_{N, \gamma}=\mylim_{N\to\infty}N^{1-\gamma} \binom{N}{\frac{N}{2}}^{-1}\sum_{\substack{k=1\\k\ \text{odd}}}^\infty\tilde{\alpha}_k e^{-\frac{k}{N^\gamma}}
    \\&= \mylim_{N\to\infty} \frac{2}{N^\gamma}\sum_{\substack{k=1\\k\ \text{odd}}}^\infty e^{-\frac{k}{N^\gamma}}
    =\mylim_{N\to\infty} \frac{2}{N^\gamma}\frac{e^{-\frac{1}{N^\gamma}}}{1-e^{-\frac{2}{N^\gamma}}}
    =1,
\end{align*}
as required.
\end{proof}

%%%%%%%%%%%%%%%%%%%%%%%%%%%%%%%%%%%%%%%%%%%%%%%%%%
%%%Limit of the 2-point Correlations%%%%%%%%%%%%%%
%%%%%%%%%%%%%%%%%%%%%%%%%%%%%%%%%%%%%%%%%%%%%%%%%%
\subsection[Limit of the $2$-Point correlations]{Limit of the $\boldsymbol{2}$-Point Correlations}
We now proceed to show Theorem~\ref{thm:2correlation_limit}, for which we will need the following preliminary result. Fix two distinct points $x_1, x_2$ in $\mathbb{T}_N$ and let 
\begin{align*}
    &B^0=B^0(N; x_1, x_2):=\{\xi\in A: \xi(x_1)=\xi(x_2)=-1\},
    \\&B^1=B^1(N; x_1, x_2):=\{\xi\in A: \xi(x_1)=1, \xi(x_2)=-1\},
    \\&B^2=B^2(N; x_1, x_2):=\{\xi\in A: \xi(x_1)=\xi(x_2)=1\}.
\end{align*}
The index in the $B$'s denotes how many sites out of $x_1$ and $x_2$ we are setting to be equal to $1$. Also, for $k=1, 3, \ldots, N-1$ and $j=0, 1, 2$ let
\begin{equation*}
    B_k^j:=\left\{\xi\in B^j: \sum_{i=1}^N i\xi_i\equiv k\pmod N\right\}
\end{equation*}
and let $\beta_k^j=\text{card}(B_k^j)$.

\begin{theorem}\label{thm:two_cardinality}
For each $k=1, 3,\ldots, 2p-1$ and $j=0, 1, 2$ we have that
\begin{equation*}
    \left|\beta_k^j-\frac{1}{p}\binom{2p-2}{p-j}\right|\le p.
\end{equation*}
\end{theorem}
\begin{proof}
The proof follows the same argument given to show Theorem~\ref{thm:cardinality}. We will show the result for $x_1=1$ and $x_2=2$, but the same proof can be adapted to any couple of different sites. By a symmetry argument, it is easy to see that we only need to show the result for $j=1$ and $j=2$. First, let $j=1$ and, given a configuration $\xi$ in $B^1$, let $j_1, \ldots, j_{p-1}$ be the remaining indices of the $-1$'s, namely $\{j_1, \ldots, j_{p-1}\}$ is the subset of $\{3, \ldots, 2p\}$ such that $\xi_{j_1}=\ldots =\xi_{j_{p-1}}=-1$. Then,
\begin{align*}
    \sum_{i=1}^{2p} i\xi_i\equiv k\pmod{2p}&\iff 1-2+\sum_{i=3}^{2p} i-2\sum_{i=1}^{p-1}
    j_i\equiv k\pmod {2p}
    \\&\iff 2\sum_{i=1}^{p-1} j_i\equiv p-k-4\pmod{2p}
    \\&\iff \sum_{i=1}^{p-1} j_i\equiv \frac{p-k}{2}-2\pmod{p}.
\end{align*}
Hence, computing the $\beta_k^1$'s is equivalent to calculating how many $(p-1)$-element subsets of the set $\{3, 4, \ldots, 2p\}$ have a sum congruent to $m\pmod p$ for each $m=0, \ldots, p-1$. Now let $\omega$ be a primitive $p$-th root of unity: note that we can write
\begin{equation}\label{eq:poly_1}
    \frac{(x^p-1)^2}{(x-\omega)(x-\omega^2)}=(x-\omega^3)(x-\omega^4)\cdots (x-\omega^{2p}).
\end{equation}
On the other hand, we can also write
\begin{align}
    &\frac{(x^p-1)^2}{(x-\omega)(x-\omega^2)}\nonumber
    \\&=\left(\frac{x^p-1}{x-\omega}\right)\left(\frac{x^p-1}{x-\omega^2}\right)\nonumber
    \\&=(x^{p-1}+\omega x^{p-2}+\ldots +\omega^{p-1})(x^{p-1}+\omega^2 x^{p-2}+\ldots +\omega^{2(p-1)}).\label{eq:poly_2}
\end{align}
Thus, the coefficient of $x^{p-1}$ in \eqref{eq:poly_1} and \eqref{eq:poly_2} must be the same, so we get
\begin{equation}\label{eq:equality}
    \sum_{3\le i_1<\ldots <i_{p-1}\le 2p}\omega^{i_1}\cdots \omega^{i_{p-1}}
    =\omega^{2(p-1)}+\omega^{2(p-2)+1}+\ldots +\omega^{p-1}.
\end{equation}

Now, since $\omega^p=1$, we can read the powers of $\omega$ modulo $p$ on both sides of the equality above. Also, as $j$ ranges from $1$ to $p$, then $2(p-j)+(j-1)\pmod p$ takes all the values $0, \ldots, p-1$, so we can write \eqref{eq:equality} as
\begin{equation}\label{eq:linear_comb}
    \sum_{k=0}^{p-1} a_k\omega ^k
    =1+\omega+\omega^2+\ldots +\omega^{p-1},
\end{equation}
where $a_k$ is the number of $(p-1)$-element subsets of $\{3, \ldots, 2p\}$ whose elements have a sum congruent to $k\pmod p$; in particular, the $a_k$'s are a permutation of the $\beta_k^1$'s. But now, a linear combination of $1, \omega, \ldots, \omega^{p-1}$ can only be zero if all the coefficients are the same, so \eqref{eq:linear_comb} implies that
\begin{equation*}
    a_0=a_1=\ldots =a_{p-1}.
\end{equation*}
Since $\sum_{k=0}^{p-1}\beta_k^1=\binom{2p-2}{p-1}$, it is easy to conclude. The proof for $j=2$ follows from the very same argument, except that one needs to look at the coefficient of $x^{p-2}$ in \eqref{eq:poly_1} and \eqref{eq:poly_2}.
\end{proof}

\begin{proof}[Proof of Theorem~\ref{thm:2correlation_limit}] For $k=1, 3, \ldots, N-1$ and $j=0, 1, 2$, let $\beta_k^j:=\beta_k^j(x_1, x_2)$. Calling $\tilde{\beta}_k^j$ the analogue of $\tilde{\alpha}_k$ defined in the proof of Corollary~\ref{corol:norm_const}, by the same argument that led to \eqref{eq:cards_equal} we get $\tilde{\beta}_{mN+k}^j=\beta_k^j$ for any $m\in\mathbb{Z}$ and $k=1, 3, \ldots, N-1$. Hence,
\begin{align*}
    &\expected_{\mu_N^*}[\xi(x_1)\xi(x_2)]
    \\&=\mu_N^*\{\xi(x_1)\xi(x_2)=1\}-\mu_N^*\{\xi(x_1)\xi(x_2)=-1\}
    \\&=\left(\mu_N^*\{B^2(x_1, x_2)\}+\mu_N^*\{B^0(x_1, x_2)\}\right)-\left(\mu_N^*\{B^1(x_1, x_2)\}+\mu_N^*\{B^1(x_2, x_1)\}\right)
    \\&=\frac{1}{\mathcal{Z}_{N, \gamma}}\sum_{\substack{k=1\\k \ \text{odd}}}^\infty\left[\left(\tilde{\beta}_k^2+\tilde{\beta}_k^0\right)-\left(\tilde{\beta}_k^1+\tilde{\beta}_k^1(x_2, x_1)\right)\right]e^{-\frac{k}{N^\gamma}}.
\end{align*}
Now, by Corollary~\ref{corol:norm_const} and Theorem~\ref{thm:two_cardinality}, for each $k$ we have that
\begin{align*}
    &\mylim_{N\to\infty} \frac{N^{\gamma+1}\left[\left(\tilde{\beta}_k^2+\tilde{\beta}_k^0\right)-\left(\tilde{\beta}_k^1+\tilde{\beta}_k^1(x_2, x_1)\right)\right]}{\mathcal{Z}_{N, \gamma}}
    \\&=\mylim_{N\to\infty}\frac{N^{\gamma+1}}{N^{\gamma-1}\binom{N}{\frac{N}{2}}}\left[\frac{2\binom{N-2}{\frac{N}{2}-2}-2\binom{N-2}{\frac{N}{2}-1}}{\frac{N}{2}}\right]
    \\&=\mylim_{N\to\infty}4N\left[\frac{\binom{N-2}{\frac{N}{2}-2}-\binom{N-2}{\frac{N}{2}-1}}{\binom{N}{\frac{N}{2}}}\right]
    \\&=\mylim_{N\to\infty}\frac{-2N}{N-1}
    \\&=-2.
\end{align*}
Hence, we get
\begin{align*}
    \mylim_{N\to\infty}N\hspace{.1em}\expected_{\mu_N^*}[\xi(x_1)\xi(x_2)]&=\mylim_{N\to\infty}-\frac{2}{N^\gamma}\sum_{\substack{k=1\\ k\ \text{odd}}}^\infty e^{-\frac{k}{N^\gamma}}
    \\&=\mylim_{N\to\infty} -\frac{2}{N^\gamma}\frac{e^{-\frac{1}{N^\gamma}}}{1-e^{-\frac{2}{N^\gamma}}}
    \\&=-1.
\end{align*}
Recall that, for any $x\in\mathbb{T}_N$, we have $\bar\xi^N(x)\in\left\{-\frac{1}{2}, \frac{1}{2}\right\}$ instead of $\{-1, 1\}$, and thus by rescaling the variables we get the result of Theorem~\ref{thm:2correlation_limit}.
\end{proof}

%%%%%%%%%%%%%%%%%%%%%%%%%%%%%%%%%%%%%%%%%%%%%%%%%%
%%%Bound for the 2m-point Correlations%%%%%%%%%%%%
%%%%%%%%%%%%%%%%%%%%%%%%%%%%%%%%%%%%%%%%%%%%%%%%%%
\subsection[Bound for the $2m$-Point Correlations]{Bound for the $\boldsymbol{2m}$-Point Correlations}
We now show Theorem~\ref{thm:2m_correlations}. The argument is, again, very similar to that used in the proof of Theorem~\ref{thm:2correlation_limit}, but it gets a bit more involved. To make the presentation clearer, we first give a full argument for the case $m=2$, and then briefly explain the steps necessary to generalise it to any $m$.

Fix four pairwise distinct points $x_1, \ldots, x_4$ in $\mathbb{T}_N$ and, for $j=0, 1, \ldots, 4$, let 
\begin{align*}
    C^j&=C^j(N; x_1, x_2, x_3, x_4)
    \\&:=\{\xi\in A: \xi(x_1)=\ldots=\xi(x_j)=1, \xi(x_{j+1})=\ldots=\xi(x_4)=-1\}.
\end{align*}
Again, the index in the $C$'s denotes how many, out of the four considered sites, we are setting to be equal to $1$. Also, as before, for $k=1, 3,\ldots, N-1$ and $j=0, 1, \ldots, 4$ let 
\begin{equation*}
    C_k^j:=\left\{\xi\in C^j: \sum_{i=1}^N i\xi_i\equiv k\pmod N\right\}
\end{equation*}
and $\gamma_k^j=\text{card}(C_k^j)$.
\begin{theorem}\label{thm:four_cardinality}
For each $k=1, 3,\ldots, 2p-1$ and $j=0, 1,\ldots, 4$, we have that
\begin{equation*}
    \left|\gamma_k^j-\frac{1}{p}\binom{2p-4}{p-j}\right|\le p^3.
\end{equation*}
\end{theorem}

\begin{proof}
We assume $x_i=i$ for each $i=1,\ldots, 4$, but again the proof can be adapted to any set of four pairwise distinct sites. By a symmetry argument, we only need to show the result for $j=2, 3, 4$. First, let $j=2$: then, by the same argument used in the proof of Theorem~\ref{thm:two_cardinality}, computing the $\gamma_k^2$'s is equivalent to calculating how many $(p-2)$-element subsets of $\{5, 6, \ldots, 2p\}$ have a sum congruent to $m\pmod{p}$ for each $m=0, 1,\ldots, p-1$. Now, if $\omega$ is a primitive $p$-th root of unity, we can write
\begin{equation}\label{eq:poly_3}
    \frac{(x^p-1)^2}{\prod_{i=1}^4(x-\omega^i)}=(x-\omega^5)(x-\omega^6)\cdots(x-\omega)^{2p}.
\end{equation}
On the other hand, we can also write 
\begin{align}
    \frac{(x^p-1)^2}{\prod_{i=1}^4(x-\omega^i)}&=\frac{x^p-1}{(x-\omega)(x-\omega^2)}\cdot \frac{x^p-1}{(x-\omega^3)(x-\omega^4)}\nonumber
    \\&=\left(\sum_{l=0}^{p-2} x^{p-2-l}\sum_{\theta\in\Theta_l}\omega^\theta\right)\left(\sum_{l=0}^{p-2} x^{p-2-l}\sum_{\lambda\in\Lambda_l}\omega^\lambda\right)\label{eq:poly_4}
\end{align}
where $\Theta_0, \ldots, \Theta_{p-2}$ and $\Lambda_0, \ldots, \Lambda_{p-2}$ are some subsets of $\{0, 1, \ldots, p-1\}$ such that, for each $l=0, \ldots, p-2$, $\text{card}(\Theta_l)=\text{card}(\Lambda_l)=l$.
This implies that the coefficient of $x^{p-2}$ in \eqref{eq:poly_3} and \eqref{eq:poly_4} must be the same, so we get that
\begin{equation*}
    -\sum_{5\le i_1<\ldots <i_{p-2}\le 2p} \omega^{i_1}\cdots \omega^{i_{p-2}}=\sum_{l=0}^{p-2}\left(\sum_{\theta\in\Theta_l}\omega^\theta\cdot\sum_{\lambda\in\Lambda_{p-2-l}}\omega^\lambda\right).
\end{equation*}

By reading the powers of $\omega$ modulo $p$, the equality above can be written as
\begin{equation*}
    \sum_{k=0}^{p-1} a_k \omega^k=-\sum_{l=0}^{p-2}\left(\sum_{\theta\in\Theta_l}\omega^\theta\cdot\sum_{\lambda\in\Lambda_{p-2-l}}\omega^\lambda\right),
\end{equation*}
where the $a_k$'s are a permutation of the $\gamma^2_k$'s. Now, the total number of powers of $\omega$ that get added in the RHS of the equation above is $\sum_{l=0}^{p-2} l(p-2-l)\le p^3$. This implies that, for each $k, j=0, \ldots, p-1$, we have that $|a_k-a_j|\le p^3$. Now let $a_{\mymin}=\mymin_k a_k$ and $a_{\mymax}=\mymax_k a_k$. Since $\sum_{k=0}^{p-1}a_k=\binom{2p-4}{p-2}$, then
\begin{equation*}
    \sum_{k=0}^{p-1} (a_{\mymin}+p^3)\ge\binom{2p-4}{p-2}
\end{equation*}
which yields $a_{\mymin}\ge\frac{1}{p}\binom{2p-4}{p-2}-p^3$;  similarly we get $a_{\mymax}\le \frac{1}{p}\binom{2p-4}{p-2}+p^3$. The proof for $j=3$ and $j=4$ follows from the exact same argument, except that one needs to look at the coefficient of $x^{p-3}$ and $x^{p-4}$, respectively, in \eqref{eq:poly_3} and \eqref{eq:poly_4}.
\end{proof}

\begin{proof}[Proof of Theorem~\ref{thm:2m_correlations}, $m=2$] Let $\gamma_k^j:=\gamma_k^j(x_1, x_2, x_3, x_4)$ for each $k=1, 3, \ldots, N-1$ and $j=0, \ldots, 4$.  By Theorem~\ref{thm:four_cardinality}, for any permutation $\sigma$ of $(x_1,\ldots, x_4)$ we have that
\begin{equation*}
    |\gamma_k^j-\gamma_k^j(\sigma)|\le \frac{N^3}{4}.
\end{equation*}
Hence, by the same argument used in the proof of Theorem~\ref{thm:2correlation_limit}, we get that
\begin{align}
    \left|\expected_{\mu_N^*}\left[\prod_{i=1}^4\xi(x_i)\right]\right|&=\left|\mu_N^*\left\{\prod_{i=1}^4 \xi(x_i)=1\right\}-\mu_N^*\left\{\prod_{i=1}^4 \xi(x_i)=-1\right\}\right|\nonumber
    \\&\le\frac{1}{\mathcal{Z}_{N, \gamma}}\sum_{\substack{k=1\\ k \ \text{odd}}}^\infty\left(\left|\gamma_k^4-4\gamma_k^3+6\gamma_k^2-4\gamma_k^1+\gamma_k^0\right|+3N^3\right)e^{-\frac{k}{N^\gamma}}.\label{eq:correlation_sum}
\end{align}
Now, by Corollary~\ref{corol:norm_const} and Theorem~\ref{thm:four_cardinality}, we have that
\begin{align*}
    &\frac{\left|\gamma_k^4-4\gamma_k^3+6\gamma_k^2-4\gamma_k^1+\gamma_k^0\right|+3N^3}{\mathcal{Z}_{N, \gamma}}
    \\&\le\frac{1}{N\mathcal{Z}_{N, \gamma}}\left[2\binom{N-4}{\frac{N}{2}-4}+6\binom{N-4}{\frac{N}{2}-2}-8\binom{N-4}{\frac{N}{2}-3}+2N^3\right]+\frac{3N^3}{\mathcal{Z}_{N, \gamma}}
    \\&\le 2N^{-\gamma}\left[4N^3\binom{N}{\frac{N}{2}}^{-1}+\frac{6}{2N^2-8N+6}\right]
    \\&\le\frac{C}{N^{2+\gamma}}
\end{align*}
for some positive constant $C$ independent of $N$, and hence
\begin{align}
    \left|\expected_{\mu_N^*}\left[\prod_{i=1}^4\xi(x_i)\right]\right|&\le\frac{C}{N^{2+\gamma}}\sum_{\substack{k=1\\ k\ \text{odd}}}^\infty e^{-\frac{k}{N^\gamma}}\label{eq:sum_analogue}
    \\&=\frac{C}{N^{2+\gamma}}\frac{e^{-{\frac{1}{N}}}}{1-e^{-\frac{2}{N^\gamma}}},\nonumber
\end{align}
which completes the proof.
\end{proof}
The key to generalising the result to any positive integer $m$ lies in the following identity.

\begin{lemma}\label{lemma:binom_identity} Let $N, m$ be positive integers with $m\le N$. Then, 
\begin{equation}\label{eq:binoms}
    \frac{\sum_{j=0}^{2m}\binom{2m}{j}\binom{2N-2m}{N-j}(-1)^j}{\binom{2N}{N}}=\frac{(-1)^m\binom{N}{m}}{\binom{2N}{2m}}.
\end{equation}
\end{lemma} 

\begin{proof}
Consider a set $X$ of $2N$ elements, fix a subset $Y\subset X$ of cardinality $2m$, and sample a subset $Z$ of $X$ uniformly from the subsets of $X$ of cardinality $N$. Then, the LHS of \eqref{eq:binoms} can be read as the probability that the cardinality of $Z\cap Y$ is even minus the probability that it is odd. But, since we are sampling uniformly at random, this is equivalent to the same difference of probabilities when first fixing $Z\subset X$ of cardinality $N$ and then sampling $Y$ uniformly at random from the subsets of $X$ of cardinality $2m$. Hence, we have the identity
\begin{equation}\label{eq:binoms_2}
    \frac{\sum_{j=0}^{2m}\binom{2m}{j}\binom{2N-2m}{N-j}(-1)^j}{\binom{2N}{N}}=\frac{\sum_{j=0}^{2m}\binom{N}{j}\binom{N}{2m-j}(-1)^j}{\binom{2N}{2m}}.
\end{equation}
Now consider the polynomial $(x^2-1)^{N}$: the coefficient of $x^{2m}$ is
\begin{equation*}
    \binom{N}{m}(-1)^{N-m}.
\end{equation*}
On the other hand, by writing $(x^2-1)^{N}=(x+1)^N(x-1)^N$, we see that the coefficient of $x^{2m}$ is also equal to
\begin{equation*}
    \sum_{j=0}^{2m}\binom{N}{j}\binom{N}{2m-j}(-1)^{N-j}.
\end{equation*}
This implies that the RHS of \eqref{eq:binoms_2} is equal to the RHS of \eqref{eq:binoms}, which concludes the proof.
\end{proof}

\begin{proof}[Sketch of proof of Theorem~\ref{thm:2m_correlations}] It is (tedious but) straightforward to verify that the analogue of Theorems~\ref{thm:two_cardinality} and~\ref{thm:four_cardinality} gives a polynomial bound in $p$ for any positive integer $m$. Then, by following the very same argument that led to \eqref{eq:correlation_sum}, it all comes down to verifying that, for $N$ sufficiently big,
\begin{align*}
    \left|\frac{1}{N}\binom{N}{\frac{N}{2}}^{-1}\sum_{j=0}^{2m}\binom{2m}{j}\binom{N-2m}{\frac{N}{2}-j}(-1)^j\right|\le\frac{C}{N^{m+1}}.
\end{align*}
for some positive constant $C$ independent of $N$. But this is a straightforward consequence of Lemma~\ref{lemma:binom_identity}.
\end{proof}

\begin{remark} By paying more attention and taking limits instead of just bounding terms from above, one could actually get the same precise asymptotic result for the $2m$-point correlation as done for the $2$-point correlations, and in particular one could show that the bounds in Theorem~\ref{thm:2m_correlations} are sharp. However, this is not necessary for our goals.
\end{remark}

%%%%%%%%%%%%%%%%%%%%%%%%%%%%%%%%%%%%%%%%%%%%%%%%%%
%%%Restrictions to Subspaces of Fixed Integral%%%%
%%%%%%%%%%%%%%%%%%%%%%%%%%%%%%%%%%%%%%%%%%%%%%%%%%
\subsection{Restrictions to Subspaces of Fixed Integral}
We finally show Theorems~\ref{thm:restriction>1} and~\ref{thm:restriction1}.

\begin{proof}[Proof of Theorems~\ref{thm:restriction>1} and~\ref{thm:restriction1}] Both results are by-products of the proofs presented above. 

To show Theorem~\ref{thm:restriction>1}, we follow the same steps as in the proofs of Theorem~\ref{thm:2m_correlations}. The only difference is that, when we get to the point of computing the analogue of the sum in \eqref{eq:sum_analogue}, the range of $k$ starts from $3$ instead of $1$, so we end up with the same bounds as before.

For Theorem~\ref{thm:restriction1} instead, when computing the bounds for the restricted correlations, this time the analogue of the sum in \eqref{eq:sum_analogue} is only indexed by $k=1$, which yields \eqref{eq:Lm}. The remaining bound for the total measure is a direct consequence of Theorem~\ref{thm:cardinality} and Corollary~\ref{corol:norm_const}, as for $N$ sufficiently big we have that
\begin{equation*}
    \expected_{\mu_N^*}\left[\boldsymbol{1}_{\Omega_N^1}\right]=\frac{\alpha_1e^{-\frac{1}{N^\gamma}}}{\mathcal{Z}_{N, \gamma}}\le\frac{2N^{1-\gamma}e^{-\frac{1}{N^\gamma}}}{N}.
\end{equation*}
This completes the proof.
\end{proof}

\endappendix

%%%%%%%%%%%%%%%%%%%%%%%%%%%%%%%%%%%%%%%%%%%%%%%%%%
%%%REFERENCES%%%%%%%%%%%%%%%%%%%%%%%%%%%%%%%%%%%%%
%%%%%%%%%%%%%%%%%%%%%%%%%%%%%%%%%%%%%%%%%%%%%%%%%%
\bibliographystyle{Martin} 
\bibliography{references}

@article {bg97,
    AUTHOR = {Bertini, Lorenzo and Giacomin, Giambattista},
     TITLE = {{Stochastic Burgers and KPZ Equations from Particle Systems}},
   JOURNAL = {Comm. Math. Phys.},
  FJOURNAL = {Communications in Mathematical Physics},
    VOLUME = {183},
      YEAR = {1997},
    NUMBER = {3},
     PAGES = {571--607},
      ISSN = {0010-3616,1432-0916},
   MRCLASS = {60K35 (60H15 82C22 82C24)},
  MRNUMBER = {1462228},
MRREVIEWER = {Ellen\ Saada},
       DOI = {10.1007/s002200050044},
       URL = {https://doi.org/10.1007/s002200050044},
}

@book {lig05,
    AUTHOR = {Liggett, Thomas M.},
     TITLE = {{Interacting Particle Systems}},
    SERIES = {Classics in Mathematics},
 PUBLISHER = {Springer-Verlag, Berlin},
      YEAR = {2005},
      ISBN = {3-540-22617-6},
   MRCLASS = {60-02 (60K35 82C22)},
  MRNUMBER = {2108619},
MRREVIEWER = {Michael\ Pr\"ahofer},
       DOI = {10.1007/b138374},
       URL = {https://doi.org/10.1007/b138374},
}

@article {gj14,
    AUTHOR = {Gonçalves, Patrícia and Jara, Milton},
     TITLE = {{Nonlinear Fluctuations of Weakly Asymmetric Interacting Particle Systems}},
   JOURNAL = {Arch. Ration. Mech. Anal.},
  FJOURNAL = {Archive for Rational Mechanics and Analysis},
    VOLUME = {212},
      YEAR = {2014},
    NUMBER = {2},
     PAGES = {597--644},
      ISSN = {0003-9527,1432-0673},
   MRCLASS = {60K35 (35R60 60H15 60H40 82C22)},
  MRNUMBER = {3176353},
MRREVIEWER = {Mathew\ Joseph},
       DOI = {10.1007/s00205-013-0693-x},
       URL = {https://doi.org/10.1007/s00205-013-0693-x},
}

@article {bgj19,
    AUTHOR = {Bernardin, C. and Gonçalves, P. and {Jiménez-Oviedo}, B.},
     TITLE = {{Slow to Fast Infinitely Extended Reservoirs for the Symmetric Exclusion Process with Long Jumps}},
   JOURNAL = {Markov Process. Related Fields},
  FJOURNAL = {Markov Processes and Related Fields},
    VOLUME = {25},
      YEAR = {2019},
    NUMBER = {2},
     PAGES = {217--274},
      ISSN = {1024-2953},
   MRCLASS = {60K35},
  MRNUMBER = {3967543},
MRREVIEWER = {Tertuliano\ Franco},
}

@book {kl99,
    AUTHOR = {Kipnis, Claude and Landim, Claudio},
     TITLE = {{Scaling Limits of Interacting Particle Systems}},
 PUBLISHER = {Springer-Verlag, Berlin},
      YEAR = {1999},
      ISBN = {3-540-64913-1},
   MRCLASS = {60-02 (60F05 60F10 60K35 82B05 82C22)},
  MRNUMBER = {1707314},
MRREVIEWER = {Timo\ Sepp\"al\"ainen},
       DOI = {10.1007/978-3-662-03752-2},
       URL = {https://doi.org/10.1007/978-3-662-03752-2},
}

@article {gmo23,
    AUTHOR = {Gonçalves, Patrícia and Misturini, Ricardo and Occelli, Alessandra},
     TITLE = {{Hydrodynamics for the ABC Model with Slow/Fast Boundary}},
   JOURNAL = {Stochastic Process. Appl.},
  FJOURNAL = {Stochastic Processes and their Applications},
    VOLUME = {161},
      YEAR = {2023},
     PAGES = {350--384},
      ISSN = {0304-4149,1879-209X},
   MRCLASS = {60K35 (82C22)},
  MRNUMBER = {4580986},
       DOI = {10.1016/j.spa.2023.04.002},
       URL = {https://doi.org/10.1016/j.spa.2023.04.002},
}

@article {ald78,
    AUTHOR = {Aldous, David},
     TITLE = {{Stopping Times and Tightness}},
   JOURNAL = {Ann. Probab.},
  FJOURNAL = {The Annals of Probability},
    VOLUME = {6},
      YEAR = {1978},
    NUMBER = {2},
     PAGES = {335--340},
      ISSN = {0091-1798},
   MRCLASS = {60B10},
  MRNUMBER = {474446},
MRREVIEWER = {Don\ McLeish},
       DOI = {10.1214/aop/1176995579},
       URL = {https://doi.org/10.1214/aop/1176995579},
}

@article {gjs17,
    AUTHOR = {Gonçalves, Patrícia and Jara, Milton and Simon, Marielle},
     TITLE = {{Second Order Boltzmann-Gibbs Principle for Polynomial Functions and Applications}},
   JOURNAL = {J. Stat. Phys.},
  FJOURNAL = {Journal of Statistical Physics},
    VOLUME = {166},
      YEAR = {2017},
    NUMBER = {1},
     PAGES = {90--113},
      ISSN = {0022-4715,1572-9613},
   MRCLASS = {82B05},
  MRNUMBER = {3592852},
MRREVIEWER = {Antonio\ Scotti},
       DOI = {10.1007/s10955-016-1686-6},
       URL = {https://doi.org/10.1007/s10955-016-1686-6},
}

@article {mit83,
    AUTHOR = {Mitoma, Itaru},
     TITLE = {{Tightness of Probabilities on {$C([0,1];{\cal S}\sp{\prime})$}\ and {$D([0,1];{\cal S}\sp{\prime} )$}}},
   JOURNAL = {Ann. Probab.},
  FJOURNAL = {The Annals of Probability},
    VOLUME = {11},
      YEAR = {1983},
    NUMBER = {4},
     PAGES = {989--999},
      ISSN = {0091-1798,2168-894X},
   MRCLASS = {60B11 (60B12)},
  MRNUMBER = {714961},
       DOI = {10.1214/aop/1176993447},
       URL = {https://doi.org/10.1214/aop/1176993447},
}

@book {klo12,
    AUTHOR = {Komorowski, Tomasz and Landim, Claudio and Olla, Stefano},
     TITLE = {{Fluctuations in Markov Processes}},
      NOTE = {{Time Symmetry and Martingale Approximation}},
 PUBLISHER = {Springer, Heidelberg},
      YEAR = {2012},
      ISBN = {978-3-642-29879-0},
   MRCLASS = {60J25 (60F05 60G44 60J10 60J60 60K35 60K37)},
  MRNUMBER = {2952852},
MRREVIEWER = {B\'alint\ T\'oth},
       DOI = {10.1007/978-3-642-29880-6},
       URL = {https://doi.org/10.1007/978-3-642-29880-6},
}

@article {cgmo25,
    AUTHOR = {Cannizzaro, G. and Gonçalves, P. and Misturini, R. and Occelli, A.},
     TITLE = {{From ABC to KPZ}},
   JOURNAL = {Probab. Theory Related Fields},
  FJOURNAL = {Probability Theory and Related Fields},
    VOLUME = {191},
      YEAR = {2025},
    NUMBER = {1-2},
     PAGES = {361--420},
      ISSN = {0178-8051,1432-2064},
   MRCLASS = {60K35 (60H15 82C22)},
  MRNUMBER = {4869258},
       DOI = {10.1007/s00440-024-01314-z},
       URL = {https://doi.org/10.1007/s00440-024-01314-z},
}

@book {js03,
    AUTHOR = {Jacod, Jean and Shiryaev, Albert N.},
     TITLE = {{Limit Theorems for Stochastic Processes}},
    SERIES = {{Grundlehren Math. Wiss.}},
    VOLUME = {288},
   EDITION = {Second},
 PUBLISHER = {Springer-Verlag, Berlin},
      YEAR = {2003},
      ISBN = {3-540-43932-3},
   MRCLASS = {60-02 (60F17 60G48 60H05)},
  MRNUMBER = {1943877},
MRREVIEWER = {Dominique\ L\'epingle},
       DOI = {10.1007/978-3-662-05265-5},
       URL = {https://doi.org/10.1007/978-3-662-05265-5},
}

@article {gpv88,
    AUTHOR = {Guo, M. Z. and Papanicolaou, G. C. and Varadhan, S. R. S.},
     TITLE = {{Nonlinear Diffusion Limit for a System with Nearest Neighbor Interactions}},
   JOURNAL = {Comm. Math. Phys.},
  FJOURNAL = {Communications in Mathematical Physics},
    VOLUME = {118},
      YEAR = {1988},
    NUMBER = {1},
     PAGES = {31--59},
      ISSN = {0010-3616,1432-0916},
   MRCLASS = {60K35 (60J60)},
  MRNUMBER = {954674},
MRREVIEWER = {F.\ L.\ Spitzer},
       URL = {http://projecteuclid.org/euclid.cmp/1104161907},
}

@article {bmns17,
    AUTHOR = {Baldasso, Rangel and Menezes, Ot\'avio and Neumann, Adriana and Souza, Rafael R.},
     TITLE = {{Exclusion Process with Slow Boundary}},
   JOURNAL = {J. Stat. Phys.},
  FJOURNAL = {Journal of Statistical Physics},
    VOLUME = {167},
      YEAR = {2017},
    NUMBER = {5},
     PAGES = {1112--1142},
      ISSN = {0022-4715,1572-9613},
   MRCLASS = {60K35 (35K55)},
  MRNUMBER = {3647054},
MRREVIEWER = {Halim\ Zeghdoudi},
       DOI = {10.1007/s10955-017-1763-5},
       URL = {https://doi.org/10.1007/s10955-017-1763-5},
}

@book {eva10,
    AUTHOR = {Evans, Lawrence C.},
     TITLE = {{Partial Differential Equations}},
    SERIES = {{Graduate Studies in Mathematics}},
    VOLUME = {19},
   EDITION = {Second},
 PUBLISHER = {American Mathematical Society, Providence, RI},
      YEAR = {2010},
      ISBN = {978-0-8218-4974-3},
   MRCLASS = {35-01},
  MRNUMBER = {2597943},
MRREVIEWER = {Diego\ M.\ Maldonado},
       DOI = {10.1090/gsm/019},
       URL = {https://doi.org/10.1090/gsm/019},
}

@article {fgn13,
    AUTHOR = {Franco, Tertuliano and Gonçalves, Patrícia and Neumann, Adriana},
     TITLE = {{Hydrodynamical Behavior of Symmetric Exclusion with Slow Bonds}},
   JOURNAL = {Ann. Inst. Henri Poincar\'e{} Probab. Stat.},
  FJOURNAL = {Annales de l'Institut Henri Poincar\'e{} Probabilit\'es et Statistiques},
    VOLUME = {49},
      YEAR = {2013},
    NUMBER = {2},
     PAGES = {402--427},
      ISSN = {0246-0203,1778-7017},
   MRCLASS = {60K35 (35K15)},
  MRNUMBER = {3088375},
       DOI = {10.1214/11-AIHP445},
       URL = {https://doi.org/10.1214/11-AIHP445},
}

@article {fgn15,
    AUTHOR = {Franco, Tertuliano and Gonçalves, Patrícia and Neumann, Adriana},
     TITLE = {{Phase Transition of a Heat Equation with Robin's Boundary Conditions and Exclusion Process}},
   JOURNAL = {Trans. Amer. Math. Soc.},
  FJOURNAL = {Transactions of the American Mathematical Society},
    VOLUME = {367},
      YEAR = {2015},
    NUMBER = {9},
     PAGES = {6131--6158},
      ISSN = {0002-9947,1088-6850},
   MRCLASS = {35K20 (35K05 60K35)},
  MRNUMBER = {3356932},
MRREVIEWER = {Ross\ Pinsky},
       DOI = {10.1090/S0002-9947-2014-06260-0},
       URL = {https://doi.org/10.1090/S0002-9947-2014-06260-0},
}

@article {el15,
    AUTHOR = {Etheridge, Alison M. and Labbé, Cyril},
     TITLE = {{Scaling Limits of Weakly Asymmetric Interfaces}},
   JOURNAL = {Comm. Math. Phys.},
  FJOURNAL = {Communications in Mathematical Physics},
    VOLUME = {336},
      YEAR = {2015},
    NUMBER = {1},
     PAGES = {287--336},
      ISSN = {0010-3616,1432-0916},
   MRCLASS = {82C26 (60H15 60J10)},
  MRNUMBER = {3322375},
       DOI = {10.1007/s00220-014-2243-2},
       URL = {https://doi.org/10.1007/s00220-014-2243-2},
}

@article {lab18,
    AUTHOR = {Labbé, Cyril},
     TITLE = {{On the Scaling Limits of Weakly Asymmetric Bridges}},
   JOURNAL = {Probab. Surv.},
  FJOURNAL = {Probability Surveys},
    VOLUME = {15},
      YEAR = {2018},
     PAGES = {156--242},
      ISSN = {1549-5787},
   MRCLASS = {60K35 (60H15 82C24)},
  MRNUMBER = {3856167},
MRREVIEWER = {Patr\'icia\ Gon\c calves},
       DOI = {10.1214/17-PS285},
       URL = {https://doi.org/10.1214/17-PS285},
}

@article {pug08,
    AUTHOR = {Pugach\"ev, O. V.},
     TITLE = {{On Mosco Convergence of Diffusion Dirichlet Forms}},
   JOURNAL = {Teor. Veroyatn. Primen.},
  FJOURNAL = {Teoriya Veroyatnoste\u i\ i ee Primeneniya},
    VOLUME = {53},
      YEAR = {2008},
    NUMBER = {2},
     PAGES = {277--292},
      ISSN = {0040-361X,2305-3151},
   MRCLASS = {60J60 (31C25)},
  MRNUMBER = {3691819},
       DOI = {10.1137/S0040585X97983547},
       URL = {https://doi.org/10.1137/S0040585X97983547},
}

@article{kpz86,
     TITLE = {{Dynamic Scaling of Growing Interfaces}},
    AUTHOR = {Kardar, Mehran and Parisi, Giorgio and Zhang, Yi-Cheng},
    JOURNAL = {Phys. Rev. Lett.},
    VOLUME = {56},
     ISSUE = {9},
     PAGES = {889--892},
      YEAR = {1986},
 PUBLISHER = {American Physical Society},
       DOI = {10.1103/PhysRevLett.56.889},
       URL = {https://link.aps.org/doi/10.1103/PhysRevLett.56.889}
}

@article {qs23,
    AUTHOR = {Quastel, Jeremy and Sarkar, Sourav},
     TITLE = {{Convergence of Exclusion Processes and the KPZ Equation to the KPZ Fixed Point}},
   JOURNAL = {J. Amer. Math. Soc.},
  FJOURNAL = {Journal of the American Mathematical Society},
    VOLUME = {36},
      YEAR = {2023},
    NUMBER = {1},
     PAGES = {251--289},
      ISSN = {0894-0347,1088-6834},
   MRCLASS = {60K35 (60H15 82C24)},
  MRNUMBER = {4495842},
MRREVIEWER = {Alex\ Mikael\ Karrila},
       DOI = {10.1090/jams/999},
       URL = {https://doi.org/10.1090/jams/999},
}

@article {ekkm98,
    AUTHOR = {Evans, M. R. and Kafri, Y. and Koduvely, H. M. and Mukamel, D.},
     TITLE = {{Phase Separation in One-Dimensional Driven Diffusive Systems}},
   JOURNAL = {Phys. Rev. Lett.},
    VOLUME = {80},
      YEAR = {1998},
    NUMBER = {3},
     PAGES = {425--429},
       DOI = {10.1103/PhysRevLett.80.425},
       URL = {https://link.aps.org/doi/10.1103/PhysRevLett.80.425},
}

@article {spo14,
    AUTHOR = {Spohn, Herbert},
     TITLE = {{Nonlinear Fluctuating Hydrodynamics for Anharmonic Chains}},
   JOURNAL = {J. Stat. Phys.},
  FJOURNAL = {Journal of Statistical Physics},
    VOLUME = {154},
      YEAR = {2014},
    NUMBER = {5},
     PAGES = {1191--1227},
      ISSN = {0022-4715,1572-9613},
   MRCLASS = {82C20},
  MRNUMBER = {3176405},
       DOI = {10.1007/s10955-014-0933-y},
       URL = {https://doi.org/10.1007/s10955-014-0933-y},
}

@article {pss15,
    AUTHOR = {Popkov, V. and Schmidt, J. and Sch\"utz, G. M.},
     TITLE = {{Universality Classes in Two-Component Driven Diffusive
              Systems}},
   JOURNAL = {J. Stat. Phys.},
  FJOURNAL = {Journal of Statistical Physics},
    VOLUME = {160},
      YEAR = {2015},
    NUMBER = {4},
     PAGES = {835--860},
      ISSN = {0022-4715,1572-9613},
   MRCLASS = {82C27 (35K59 35R09)},
  MRNUMBER = {3373643},
MRREVIEWER = {Carlos\ Escudero},
       DOI = {10.1007/s10955-015-1241-x},
       URL = {https://doi.org/10.1007/s10955-015-1241-x},
}

@article {psss16,
    AUTHOR = {Popkov, V. and Schadschneider, A. and Schmidt, J. and Sch\"utz, G. M.},
     TITLE = {{Exact Scaling Solution of the Mode Coupling Equations for Non-Linear Fluctuating Hydrodynamics in One Dimension}},
   JOURNAL = {J. Stat. Mech. Theory Exp.},
  FJOURNAL = {Journal of Statistical Mechanics: Theory and Experiment},
      YEAR = {2016},
    NUMBER = {9},
    VOLUME = {2016},
     PAGES = {093211},
      ISSN = {1742-5468},
   MRCLASS = {82C22 (82D15)},
  MRNUMBER = {3562237},
       DOI = {10.1088/1742-5468/2016/09/093211},
       URL = {https://doi.org/10.1088/1742-5468/2016/09/093211},
}

@incollection {sch18,
    AUTHOR = {Sch\"utz, G. M.},
     TITLE = {{On the Fibonacci Universality Classes in Nonlinear Fluctuating Hydrodynamics}},
 BOOKTITLE = {{From Particle Systems to Partial Differential Equations}},
    SERIES = {Springer Proc. Math. Stat.},
    VOLUME = {258},
     PAGES = {149--167},
 PUBLISHER = {Springer, Cham},
      YEAR = {2018},
      ISBN = {978-3-319-99689-9; 978-3-319-99688-2},
   MRCLASS = {76M28 (35Q35)},
  MRNUMBER = {3903314},
MRREVIEWER = {Huazhong\ Tang},
       DOI = {10.1007/978-3-319-99689-9_2},
       URL = {https://doi.org/10.1007/978-3-319-99689-9_2},
}

@article {cde03,
    AUTHOR = {Clincy, M. and Derrida, B. and Evans, M. R.},
     TITLE = {{Phase Transition in the {$ABC$} Model}},
   JOURNAL = {Phys. Rev. E (3)},
  FJOURNAL = {Physical Review E. Statistical, Nonlinear, and Soft Matter Physics},
    VOLUME = {67},
      YEAR = {2003},
    NUMBER = {6},
     PAGES = {066115, 8},
      ISSN = {1539-3755,1550-2376},
   MRCLASS = {82C26},
  MRNUMBER = {1995892},
       DOI = {10.1103/PhysRevE.67.066115},
       URL = {https://doi.org/10.1103/PhysRevE.67.066115},
}

@article {bdlw08,
    AUTHOR = {Bodineau, T. and Derrida, B. and Lecomte, V. and  Wijland, F. van},
     TITLE = {{Long Range Correlations and Phase Transitions in Non-Equilibrium Diffusive Systems}},
   JOURNAL = {J. Stat. Phys.},
  FJOURNAL = {Journal of Statistical Physics},
    VOLUME = {133},
      YEAR = {2008},
    NUMBER = {6},
     PAGES = {1013--1031},
      ISSN = {0022-4715,1572-9613},
   MRCLASS = {82C26 (82C70)},
  MRNUMBER = {2462009},
MRREVIEWER = {Christian\ Dogbe},
       DOI = {10.1007/s10955-008-9647-3},
       URL = {https://doi.org/10.1007/s10955-008-9647-3},
}

@article{ho25,
     TITLE = {{Nonlinear Fluctuations for a Chain of Weakly Anharmonic Oscillators with Stochastic Perturbation}}, 
    AUTHOR = {Kohei Hayashi and Stefano Olla},
      YEAR = {2025},
   JOURNAL = {ArXiv e-prints},
       DOI = {10.48550/arXiv.2510.12922},
       URL = {https://doi.org/10.48550/arXiv.2510.12922},
}

@article {mqr21,
    AUTHOR = {Matetski, Konstantin and Quastel, Jeremy and Remenik, Daniel},
     TITLE = {{The KPZ Fixed Point}},
   JOURNAL = {Acta Math.},
  FJOURNAL = {Acta Mathematica},
    VOLUME = {227},
      YEAR = {2021},
    NUMBER = {1},
     PAGES = {115--203},
      ISSN = {0001-5962,1871-2509},
   MRCLASS = {60K35},
  MRNUMBER = {4346267},
       DOI = {10.4310/acta.2021.v227.n1.a3},
       URL = {https://doi.org/10.4310/acta.2021.v227.n1.a3},
}

@article {rav92,
    AUTHOR = {Ravishankar, K.},
     TITLE = {{Fluctuations from the Hydrodynamical Limit for the Symmetric
              Simple Exclusion in {${\bf Z}^d$}}},
   JOURNAL = {Stochastic Process. Appl.},
  FJOURNAL = {Stochastic Processes and their Applications},
    VOLUME = {42},
      YEAR = {1992},
    NUMBER = {1},
     PAGES = {31--37},
      ISSN = {0304-4149,1879-209X},
   MRCLASS = {60K35 (82C40)},
  MRNUMBER = {1172505},
MRREVIEWER = {Pablo\ A.\ Ferrari},
       DOI = {10.1016/0304-4149(92)90024-K},
       URL = {https://doi.org/10.1016/0304-4149(92)90024-K},
}

@article {bfs21,
    AUTHOR = {Bernardin, C. and Funaki, T. and Sethuraman, S.},
     TITLE = {{Derivation of Coupled KPZ-Burgers Equation from
              Multi-Species Zero-Range Processes}},
   JOURNAL = {Ann. Appl. Probab.},
  FJOURNAL = {The Annals of Applied Probability},
    VOLUME = {31},
      YEAR = {2021},
    NUMBER = {4},
     PAGES = {1966--2017},
      ISSN = {1050-5164,2168-8737},
   MRCLASS = {60K35 (60F17 60G60 60H15 82C22)},
  MRNUMBER = {4312852},
       DOI = {10.1214/20-aap1639},
       URL = {https://doi.org/10.1214/20-aap1639},
}

@article {ch58,
    AUTHOR = {J. W. Cahn and J. E. Hilliard},
     TITLE = {{Free Energy of a Nonuniform System. I. Interfacial Free Energy}},
   JOURNAL = {J. Chem. Phys.},
  FJOURNAL = {Journal of Chemical Physics},
    VOLUME = {28},
      YEAR = {1958},
    NUMBER = {2},
     PAGES = {258--267},
       DOI = {10.1063/1.1744102},
       URL = {https://doi.org/10.1063/1.1744102},
}

\end{document}